\documentclass{article}
\usepackage{graphicx} % Required for inserting images
\usepackage{fancyhdr}
\usepackage{latexsym}
\usepackage{mathrsfs}
\usepackage{cancel}
\usepackage{color}
\usepackage{multirow}
\usepackage{eso-pic}
\usepackage{dsfont}
\usepackage{amssymb}
\usepackage{graphicx}
\usepackage{setspace}
\usepackage{geometry}
\usepackage{amsthm}
\usepackage{hyperref}
\usepackage[intlimits]{amsmath}
\usepackage{hyperref}
\usepackage{tikz-cd}
\usepackage{amssymb,amsfonts}
\usepackage[all,arc]{xy}
\usepackage{enumerate}
\usepackage{mathrsfs}
\usepackage{tikz}
\usepackage{bbm}
\usepackage{graphicx}
\usepackage{mathtools}
\usetikzlibrary{braids}
\usepackage{authblk}

\newcommand{\wt}{\widetilde}
\newcommand{\wh}{\widehat}

\newcommand{\End}{\mathrm{End}}

\newcommand{\Hom}{\mathrm{Hom}}

\newcommand{\Mod}{\mathrm{-Mod}}

\newcommand{\Gr}{\mathrm{Gr}}

\newcommand{\lp}{\left(}
\newcommand{\rp}{\right)}

\newcommand{\Ind}{\mathrm{Ind}}

\newcommand{\Coh}{\mathrm{Coh}}

\newcommand{\IndCoh}{\mathrm{IndCoh}}

\newcommand{\lpp}{(\!(}
\newcommand{\rpp}{)\!)}
\newcommand{\lbb}{[\![}
\newcommand{\rbb}{]\!]}

\newcommand{\match}{\rhd\!\!\!\lhd }

\newcommand{\pd}{{\partial}}

\newcommand{\lag}{\langle}
\newcommand{\rag}{\rangle}

\newcommand{\C}{\mathbb C}
\newcommand{\R}{\mathbb R}
\newcommand{\Z}{\mathbb Z}

\newcommand{\N}{\mathbb N}

\newcommand{\fd}{\mathfrak{d}}
\newcommand{\fg}{\mathfrak{g}}
\newcommand{\fh}{\mathfrak{h}}
\newcommand{\fp}{\mathfrak{p}}

\newcommand{\CA}{{\mathcal A}}

\newcommand{\CC}{{\mathcal C}}

\newcommand{\CE}{{\mathcal E}}

\newcommand{\CK}{{\mathcal K}}

\newcommand{\CM}{{\mathcal M}}
\newcommand{\CN}{{\mathcal N}}
\newcommand{\CO}{{\mathcal O}}

\newcommand{\CR}{{\mathcal R}}

\newcommand{\CV}{{\mathcal V}}

\newcommand{\CY}{{\mathcal Y}}

\newcommand{\norm}[1]{{{:\!{#1}\!:}}}
\newcommand{\be}{\begin{equation}}
\newcommand{\ee}{\end{equation}}

\newcommand{\btik}{\begin{tikzcd}}
\newcommand{\etik}{\end{tikzcd}}

\newtheorem{Def}{Definition}[section]
\newtheorem{Thm}[Def]{Theorem}
\newtheorem{Prop}[Def]{Proposition}
\newtheorem{Cor}[Def]{Corollary}
\newtheorem{Lem}[Def]{Lemma}
\newtheorem{Rem}[Def]{Remark}

\newtheorem{Exp}[Def]{Example}
\newtheorem{Conj}[Def]{Conjecture}

\title{Yangian for cotangent Lie algebras and spectral $R$-matrices}

\author[1]{Raschid Abedin}
\author[2]{Wenjun Niu}

\affil[1]{Department of Mathematics, ETH Zürich, 8092 Zürich, Schweiz,  raschid.abedin@math.ethz.ch}

\affil[2]{Perimeter Institute for Theoretical Physics, 31 Caroline St N, Waterloo, ON N2L 2Y5, Canada, wjniu950925@gmail.com}

\date{\today}

\begin{document}

\maketitle

\abstract{In this paper, we present a canonical quantization of Lie bialgebra structures on the formal power series \(\fd[\![t]\!]\) with coefficients in the cotangent Lie algebra \(\fd = T^*\fg = \fg \ltimes \fg^*\) to a simple complex Lie algebra \(\fg\). We prove that these quantizations produce twists to the natural analog of the Yangian for \(\fd\). Moreover, we construct spectral \(R\)-matrices for these twisted Yangians as compositions of twisting matrices. 

The motivation for the construction of these twisted Yangians over \(\fd\) comes from a certain 4d holomorphic-topological gauge theory. More precisely, we show that pertubative line operators in this theory can be realized as representations of these Yangians. Moreover, the comultiplications of these Yangians correspond to the monodial structure of the category of line operators.}

\tableofcontents

\numberwithin{equation}{section}

\newpage

\section{Introduction}

\subsection{Main results}

Let $\fg$ be a finite-dimensional complex simple Lie algebra and 
\be\label{eq:intro_takiff}
\fd:=T^*\fg=\fg\ltimes \fg^* \cong \fg[\epsilon]/\epsilon^2\fg[\epsilon],
\ee
which we call the \textit{cotangent Lie algebra} of $\fg$. One can identify $\fd$ with the double $D(\fg)$ of $\fg$ with the trivial Lie coalgebra structure. This paper is concerned with the quantization of Lie bialgebra structures on the loop algebra $\fd(\CO) \coloneqq \fd \otimes \CO$, where $\CO:=\C\lbb t\rbb$ is the ring of Taylor power series. 

More precisely, following \cite{drinfeld1986quantum}, we consider the (topological) Lie bialgebra structure of $\fd(\CO)$ induced by Yang's $r$-matrix:
\be
\gamma:=\frac{C}{t_1-t_2},\qquad C \in \fd \otimes \fd \text{ is the quadratic Casimir of } \fd. 
\ee
It is well-known that the double of $\fd(\CO)$ with respect to this Lie bialgebra structure is simply $\fd(\CK)$, with Manin triple given by $(\fd(\CK), \fd(\CO), t^{-1}\fd[t^{-1}])$. Here, $\CK:=\C\lpp t\rpp$ is the field of Laurent series. More generally, we will comment on how given a splitting of $\fg(\CK)$ into Lie subalgebras $\fg(\CK)=\fg(\CO)\oplus W$, one obtains a Lie bialgebra structure on $\fd(\CO)$ by considering the Manin triple:
\be
\fd(\CK)=\fd(\CO)\oplus \lp W\oplus W^\perp\rp.
\ee
Here, $W^\perp\subseteq \fg^*(\CK)$ is the subspace which pairs trivially with \(W\) with respect to the pairing induced by the residue map on \(\CK\). Since $\fg$ is simple, such a $W$ corresponds to a not necessarily skew-symmetric classical spectral $r$-matrix \(r\) valued in $\fg$, which gives rise to a skew-symmetric classical spectral $r$-matrix $\rho$ valued in $\fd$. Moreover, every \(r\)-matrix with values in \(\fd\) that respects the \(\epsilon\)-grading induced by the last identification in \eqref{eq:intro_takiff} is of this form. We write $W=\fg(r)$ and denote the cobracket induced by \(\rho\) on $\fd(\CO)$ as $\delta_\rho$. In particular, $\delta_\gamma$ is the one defined by Yang's $r$-matrix. 

In this paper, we give a new construction of a quantization of $(\fd(\CO),\delta_\rho)$ over $\C\lbb\hbar\rbb$. More precisely, we prove the following statement:

\begin{Thm}[Theorem \ref{Thm:quantizerho}]\label{Thm:quantization_intro}
    For each \(\epsilon\)-graded \(r\)-matrix $\rho$ with values in \(\fd\), we construct an \(\epsilon\)-graded quantization $\CA_\hbar (\fd, \rho)$ of $(\fd(\CO),\delta_\rho)$, i.e.\ a Hopf algebra $\CA_\hbar (\fd, \rho)$ over $\C\lbb\hbar\rbb$ such that:
    \begin{enumerate}
        \item $\CA_\hbar(\fd, \rho)/\hbar\CA_\hbar(\fd, \rho)=U(\fd(\CO))$ as Hopf algebras;

        \item $\Delta_{\rho,\hbar}-\Delta_{\rho,\hbar}^{\textnormal{op}}=\hbar \delta_\rho + \CO(\hbar^2)$;

        \item \(\CA_\hbar(\fd, \rho)\) is graded with respect to the grading on $\fd$ under which $\fg$ is in degree $0$, $\fg^*$ is in degree $2$, and $\hbar$ has degree $-2$.
        
    \end{enumerate}
Moreover, for two different \(\epsilon\)-graded \(r\)-matrices $\rho_1,\rho_2$ with values in \(\fd\), there exists a canonical algebra identification $\CA_\hbar(\fd, \rho_1)\cong \CA_\hbar(\fd, \rho_2)=\CA_\hbar(\fd)$ and an element $F\in \CA_\hbar(\fd)\otimes_{\C[\![\hbar]\!]}\CA_\hbar(\fd)$ such that:
        \be
F\Delta_{\rho_1,\hbar}F^{-1}=\Delta_{\rho_2,\hbar},\qquad (\Delta_{\rho_2,\hbar}\otimes 1)(F)F^{12}=(1\otimes \Delta_{\rho_2,\hbar})(F) F^{23}. 
        \ee
    Both the topological Hopf algebra structure as well as the twisting are well-defined after evaluating at $\hbar=\xi$ for any $\xi\in \C$. 
    
\end{Thm}

The inspiration of this quantization comes from considering the geometry of the moduli space of $G$-bundles on a formal bubble, a.k.a.\ the \textit{equivariant affine Grassmannian}:
\be
[G(\CO)\!\setminus\!\Gr_G]=G(\CO)\!\setminus\!G(\CK)/G(\CO),
\ee
together with its convolution monoidal structure. More precisely, let $\wh\Gr_G$ be the formal completion of the affine Grassmannian at the identity coset $[e]$ and consider the category of $G(\CO)$-equivariant coherent sheaves on $\wh \Gr_G$. This category has the structure of a monoidal category via a convolution diagram, which we will recall in Section \ref{sec:RavGrG}. Given a splitting $\fg(r)$ of $\fg(\CO)\to \fg(\CK)$, one can identify $\wh\Gr_G$ with $\wh\fg(r)$, which is the formal completion of the ind-vector space $\fg(r)$ at $0$. Therefore one can identify sheaves on $\wh\Gr_G$ with smooth modules of the topological algebra $\C[\wh\fg(r)]$. As an algebra, $\CA_\hbar(\fd,\rho)$ evaluated at $\hbar=1$ is simply the smashed product $U(\fg(\CO))\# \C[\wh\fg(r)]$. We construct the coproduct on $\CA_\hbar(\fd, \rho)$ that mimics the convolution product on the double quotient, which provides the canonical quantization of $(\fd(\CO),\delta_\rho)$. We will in fact show the following statement.

\begin{Prop}[Proposition \ref{Prop:abequiv}, Proposition \ref{Prop:E1equiv}]\label{Prop:E1equivintro}
There is an equivalence of monoidal categories
\be
\Coh_{G(\CO)}(\wh \Gr_G)^{\heartsuit}\simeq \CA_1(\fd, \rho)\Mod_{G},
\ee
where the right-hand side is the category of finite-dimensional smooth modules of $\CA_1(\fd, \rho)$ (the algebra evaluated at $\hbar=1$), such that the action of $\fg$ integrates to an algebraic action of $G$. 

\end{Prop}

In the special case when $\rho$ depends only on $t_1-t_2$, the Hopf algebra $\CA_\hbar(\fd,\rho)$ admits a Hopf derivation $T$, which is the time-translation $\pd_t$ action inherited from $\fd(\CK)$. Moreover, when $\rho=\gamma$ is Yang's $r$-matrix, then the quantization $\CA_\hbar(\fd, \gamma)$ is not only $\epsilon$-graded, but also graded by loop rotation $t^n\to c^nt^n$. In this specific case, we have the following uniqueness result.

\begin{Thm}[Theorem \ref{Thm:yangunique}]\label{Thm:yanguniqueintro}
    There is, up to isomorphism, a unique quantization of $(\fd(\CO),\delta_\gamma)$ which is both \(\epsilon\)-graded and loop graded.    
\end{Thm}

 Consequently, $\CA_\hbar(\fd, \gamma)$ is isomorphic to the Yangian $Y_\hbar(\fd)$ (as defined by \cite{etingof1996quantization, etingof1998quantization}) as a bi-graded Hopf algebra, since the completion of the latter is a bi-graded quantization of $(\fd(\CO),\delta_\gamma)$. Just like the Yangian for the simple Lie algebra $\fg$, it is expected that $Y_\hbar(\fd)$ should possess a spectral $R$-matrix. Our next main result is the construction of this spectral $R$-matrix. However, due to convergence issues this $R$-matrix is constructed in the dense Hopf subalgebra $ Y^\circ_\hbar(\fd)$ generated by $\fd[t]$. 
 
\begin{Thm}[Theorem \ref{Thm:fullR}]\label{Thm:fullRintro}
 Denote by $\tau_z$ the action of $e^{zT}$. There exists an element
    \be
R(z)\in (Y^\circ_\hbar(\fd) \otimes_{\C[\![\hbar]\!]} Y_\hbar^\circ(\fd))(\!( z^{-1})\!)
    \ee
   such that
   \be
(\tau_z\otimes 1)\Delta_{\hbar}^{\textnormal{op}}(a)=R(z) (\tau_z\otimes 1) \Delta_{\hbar}(a) R(z)^{-1},
   \ee
for any element $ a\in Y_\hbar^\circ(\fd)$. Here, the equality is in $(Y_\hbar(\fd)\otimes_{\C\lbb\hbar\rbb} Y_\hbar(\fd))\lpp z^{-1}\rpp$. This $R$ satisfies the quantum Yang-Baxter equation:
   \be
R^{12}(z_1)R^{13}(z_1+z_2)R^{23}(z_2)=R^{23}(z_2)R^{13}(z_1+z_2)R^{12}(z_1).
   \ee
   
\end{Thm}

In fact, an analog of Theorem \ref{Thm:fullRintro} exists for any $\CA_\hbar(\fd, \rho)$, where $\rho$ is an \(\epsilon\)-graded \(r\)-matrix that depends only on $t_1-t_2$; see Section \ref{sec:twisting_R_matrices}.
The construction of $R$ in Theorem \ref{Thm:fullRintro} again uses the relation to the equivariant affine Grassmannian, but this time taking inspiration from its factorization structure \cite{beilinson2004chiral}. This construction also gains inspiration from \cite{gautam2021meromorphic}, where the full $R$-matrix is constructed from meromorphic $R$-matrices of the Yangian. 

More precisely, we show that $ Y^\circ_\hbar(\fd)$ admits another coproduct, this time with meromorphic dependence:
\be
\Delta_z \colon  Y^\circ_\hbar(\fd)\to ( Y^\circ_\hbar(\fd)\otimes_{\C\lbb\hbar\rbb}  Y^\circ_\hbar(\fd))\lpp z^{-1}\rpp.
\ee
This coproduct is constructed from the intertwining map of the affine Kac-Moody vertex algebra $V_0(\fg)$. We then construct a twisting similar to \cite{gautam2021meromorphic}, which is an element
\be
R_s(z)\in  (Y^\circ_\hbar(\fd)\otimes_{\C\lbb\hbar\rbb}  Y^\circ_\hbar(\fd))\lbb z^{-1}\rbb,
\ee
such that:
\begin{itemize}
    \item \((\tau_z\otimes 1)\Delta_\hbar=R_s(z)\Delta_zR_s(z)^{-1}\);
    \item \((\Delta_{z_1}\otimes 1)(R_s(z_2)^{-1})R_{s}^{12}(z_1)^{-1}=(1\otimes \Delta_{z_2})(R_s(z_1+z_2)^{-1}) R_{s}^{23}(z_2)^{-1}\). 
\end{itemize}
The full spectral $R$-matrix $R(z)$ is constructed from $R_s$ by:
\be
R(z)=R_{s}^{21}(-z)R_s(z)^{-1}.
\ee
The quantum Yang-Baxter equation of $R(z)$ follows from the cocycle condition satisfied by $R_s(z)$. 
Since \(\CA_\hbar(\fd,\rho)\) is obtained from \(Y_\hbar(\fd)\) by twisting according to Theorem \ref{Thm:quantization_intro}, we obtain the aforementioned analog of Theorem \ref{Thm:fullRintro} in Section \ref{sec:twisting_R_matrices} for \(\epsilon\)-graded \(r\)-matrices \(\rho\) with coefficients in \(\fd\) other than \(\gamma\).

The construction of $\Delta_z$ makes use of the intertwining operator $\CY$ of the vertex algebra $V_0(\fg)$. In fact, we comment in Section \ref{sec:MeroR} how $Y_\hbar(\fd)$ can be identified with the continuous dual of $S(t^{-1}\fg^*[t^{-1}])\otimes V_0(\fg)$, under which $\Delta_z$ is the dual of the intertwining operator on this vertex algebra (the vertex algebra structure on $S(t^{-1}\fg^*[t^{-1}])$ is commutative). Although we are not able to directly compare this coproduct with the factorization structure of $\Coh_{G(\CO)}(\wh\Gr_G)$, we show in Section \ref{sec:RavGrG} that one can naturally obtain this vertex algebra from the equivariant affine Grassmannian.

\begin{Prop}[Proposition \ref{Prop:dualVOA}]\label{Prop:dualVOAintro}
     Let $\pi: \wh{\Gr}_G\to [\wh{G}(\CO)\!\setminus\!\wh{\Gr}_G]$ be the natural projection, where $\wh{G}(\CO)$ is the formal completion of $G(\CO)$ at identity, and let $\omega$ be the dualizing sheaf of $\wh \Gr_G$. Then one can identify
    \be
\Gamma(\wh{\Gr}_G, \pi^*\pi_*(\omega))\cong \C[\wh{G}(\CO)]\otimes V_0(\fg),
    \ee
    as vertex algebras, where the vertex algebra $\C[\wh{G}(\CO)]$ is commutative. 
\end{Prop}

Note that $\C[\wh{G}(\CO)]$ is a completion of $S(t^{-1}\fg^*[t^{-1}])$. In \cite{cautis2023canonical}, the authors obtained a renormalized $r$-matrix using the monoidal factorization structure of $G(\CO)\!\setminus \!\Gr_G$. Although we are not able to prove this, we present the following conjecture:

\begin{Conj}[Conjecture \ref{Conj:merorfac}]\label{Conj:merofacintro}
Under the equivalence of Proposition \ref{Prop:E1equivintro}, the renormalized $r$-matrix of \cite{cautis2019cluster} corresponds to the lowest non-trivial loop degree part of the quantum R-matrix $R(z)$, acting on a tensor product of smooth modules. 
\end{Conj}

\begin{Rem}
    We must remark that in our considerations, all algebraic structures are taken to be continuous with respect to the topology coming from loop grading. For instance, tensor products \(\otimes\) are completed with respect to this topology and \((\cdot)^*\) denotes the continuous dual. In particular, the cobracket $\delta_\rho$ is \textbf{not} a cobracket for Lie algebras, but rather for topological Lie algebras, since $\fd(\CO)\otimes \fd(\CO) \cong (\fd \otimes \fd)[\![t_1,t_2]\!]$ is the completed tensor product. The constructions of Hopf algebra structures we present in this paper naturally respect this topology, and we will sweep this subtlety under the rug to ensure the cleanliness of the presentation. 
    
\end{Rem}

\begin{Rem}
    All the above statements made, except the uniqueness statement in Theorem \ref{Thm:yangunique}, are true for any finite-dimensional DG Lie algebra $\fg$ and its cotangent Lie algebra. In particular, one can construct spectral $R$-matrix for $Y_\hbar^\circ(\fd)$ for any DG Lie algebra $\fg$ and $\fd=T^*\fg$. As we will comment in the next section, these algebras are closely related to a class of quantum field theories called 4d $\CN=2$ gauge theories. 
    
\end{Rem}

\subsection{Physical motivation: Kapustin twist of 4d $\CN=2$ gauge theories}

Let $G$ be a complex reductive Lie group with compact form $G_c$, and $V_c$ a finite-dimensional representation of $G_c$ whose complexification is $V$. For such a pair, physicists have constructed a 4-dimensional quantum field theory with $\CN=2$ supersymmetries. Such a theory admits a variety of twists, labelled by nilpotent elements in the supersymmetry algebra, which are amenable to mathematical studies. For instance, in \cite{witten1988topological}, the author used a topological twist of the theory associated to $G=SL(2)$ and $V=0$ to reproduce the Donaldson invariants of 4-manifolds. 

The twisted theory that motivated the study of this paper is the Kapustin twist, or holomorphic-topological twist (HT twist in short) considered in \cite{kapustin2006holomorphic, kapustin2006wilson}. The name HT twist is given because the twisted theory is a holomorphic-topological theory, which requires the space-time to be locally of the form $\R^2\times \C$. The theory is \textit{holomorphic} along $\C$ and \textit{topological} along $\R^2$. This twisted theory has received much attention from mathematical communities, due to its relation to chiral algebras \cite{beem2015infinite, jeong2019scft,oh2019chiral, oh2020poisson, butson2021equivariant, dedushenko20234d}, integrable systems and quantum affine algebras \cite{bezrukavnikov2005equivariant, finkelberg2018comultiplication, finkelberg2019multiplicative}, wall crossing and cluster structures \cite{kontsevich2008stability, gaiotto2013framed, dimofte2014gauge, cordova2016infrared, cautis2019cluster, cautis2023canonical}, as well as its relation to geometric Langlands correspondence \cite{kapustin2006holomorphic,kapustin2006wilson, nekrasov2006seiberg, jeong2024di}.  

Associated to this holomorphic topological field theory is a category, physically the category of line operators of the theory. By this we mean the category whose objects are line observables stretched in $\R^2$ and at a point in $\C$, and morphisms are local observables at junctions of lines. The holomorphic-topological nature of the theory guarantees that this category has the structure of a \textit{monoidal category} coming from collision in the $\R^2$ plane. Moreover, it has the structure of a \textit{chiral category} coming from operator product expansion (OPE) in the $\C$ plane. The understanding of these structures and their interplay should help tremendously in understanding the algebraic and geometric properties of the theory. 

Based on the proposal of \cite{kapustin2006holomorphic, kapustin2006wilson}, as well as inspirations from 3d gauge theories \cite{nakajima2016towards, braverman2018towards}, this category is given a geometric definition in \cite{cautis2019cluster, cautis2023canonical}. The idea of \cite{kapustin2006holomorphic, kapustin2006wilson} is that the dimensional reduction of the theory along $\C$ (or any complex curve $\Sigma$) gives a 2d B-model whose target is a generalization of Hitchin moduli space. This generalized Hitchin moduli space is:
\be
\CM_{G,V}(\Sigma):=\mathrm{Maps}(\Sigma, V/G),
\ee
namely the moduli space of holomorphic $G$-bundles on $\Sigma$, together with a section of the associated $V$ bundle. Upon reduction, line operators are identified with coherent sheaves on the moduli space $\CM_{G,V}(S^2)$ for a 2-sphere and acts on the above space via Hecke modifications. In algebraic formulation (as in \cite{nakajima2016towards, braverman2018towards}), one replace $S^2$ with the so-called formal bubble $\mathbb{B}:=\mathbb{D}\cup_{\mathbb{D}^\times}\mathbb{D}$, where $\mathbb{D}:=\mathrm{Spec}(\CO)$ and $\mathbb{D}^\times:=\mathrm{Spec}(\CK)$. Using this definition, one defines the space $\CM_{G, V}(\mathbb{B})=[G(\CO)\!\setminus\!\CR_{G,V}]$, where $\CR_{G, V}$ is defined via the following Cartesian diagram (known as the BFN space, after Braverman-Finkelberg-Nakajima):
\be
\btik
\mathcal{R}_{G,V} \rar \arrow[d] &  V(\mathcal{O})\arrow[d] \\  G(\mathcal{K})\times_{G(\mathcal{O})}V(\mathcal{O}) \rar &  V(\mathcal{K})
\etik
\ee
Note that when $V=0$, $\CR_{G,V}=\Gr_G$ and the moduli space $\CM_{G, 0}(\mathbb{B})$ is precisely the equivariant affine Grassmannian. 

The mathematical proposal of the category of line operators is therefore $\Coh_{G(\CO)}(\CR_{G,V})$. Beautifully, this category has the structures needed for the holomorphic-topological nature of the theory. First of all, it is a chiral category, since both $\CR_{G,V}$ and $G(\CO)$ are built out of formal disks and formal punctured disks, and therefore admit factorization structure. In the case when $V=0$, $\CR_{G,V}=\Gr_G$ and the factorization structure here was constructed from the famous Beilinson-Drinfeld Grassmannian \cite{beilinson2004chiral}. Secondly, it is a monoidal category, defined via a convolution diagram in \cite{braverman2018towards}. Moreover, they are compatible in the sense that the monoidal multiplication and chiral OPE are distributive with respect to each other. In \cite{cautis2019cluster, cautis2023canonical}, it is shown that in special cases, this category also has a cluster structure, matching the physical expectation. We also note that in \cite{niu2022local}, this category was used to reproduce the Poisson vertex algebra of \cite{oh2020poisson}, as well as Schur indices of \cite{cordova2016infrared}. 

However, much is still unknown about this category. For example, in \cite{cautis2019cluster, cautis2023canonical}, many tensor products were computed, and a renormalized $r$-matrix was constructed, but the full interplay between the two structures is still lacking. The main difficulty is due to the complicated geometry of the affine Grassmannian $\Gr_G$. 

On the other hand, it is known to experts that the HT twist of a 4d $\CN=2$ gauge theory can be viewed as some version of a 4d Chern-Simons theory (more precisely, Chern-Simons theory for the supergroup $T^*(G\ltimes V[-1])$). The first hint at this appears in the work of \cite{costello2013supersymmetric, costello2014integrable}. In \cite{costello2018gauge, costello2018gauge2, costello2019gauge}, it was argued, based on physical grounds, that at least perturbatively (and for simple Lie algebras $\fg$), line operators in such theories can be identified with modules of the Yangian. 

From this perspective, the interplay between chiral OPE and monoidal multiplication is purely encoded in the spectral $R$-matrix of Drinfeld \cite{drinfel1990hopf}. Moreover, in \cite{gautam2021meromorphic}, this $R$-matrix is given a decomposition into two meromorphic $R$-matrices, making this interplay even more explicit.  

The constructions of this paper are inspired by such Chern-Simons theory considerations. Geometrically, we don't consider the full affine Grassmannian (or the BFN space), but rather its formal completion along the identity coset:
\be
\wh \CM_{G,V}(\mathbb{B}):=[\wh{G}(\CO)\!\setminus\!\wh{\CR}_{G, V}].
\ee
Physically, this should correspond to ignoring non-perturbative line operators. Choosing a splitting of $\fg(\CK)$ (or generally $\fg\ltimes V[-1](\CK)$) and identifying $\Coh(\wh\CM_{G,V}(\mathbb B))$ with modules of $\CA_\hbar(\fd, \rho)$ corresponds to choosing a fiber functor for the Chern-Simons theory by picking a vacuum at infinity. 

From this perspective, Theorem \ref{Thm:fullRintro}, especially its construction which resembles so much the work of \cite{gautam2021meromorphic}, gives an explicit interplay between chiral OPE and monoidal structure of line operators, at least perturbatively. For simple $\fg$, Theorem \ref{Thm:yanguniqueintro} also confirms the relation between 4d $\CN=2$ pure gauge theory of $\fg$ and Chern-Simons theory of $\fd=T^*\fg$. 

For the case when $V=0$, namely for pure gauge theory with gauge group $G$, Proposition \ref{Prop:E1equivintro} shows that the coproduct of $Y_\hbar(\fd)$ indeed match the monoidal structure of the equivariant affine Grassmannian, and moreover Proposition \ref{Prop:dualVOAintro} indicates that the meromorphic coproduct $\Delta_z$ of $Y_\hbar(\fd)$ matches the factorization structure of this space. All these lead to our formulation of Conjecture \ref{Conj:merofacintro}. 

Based on all these physical considerations, as well as Theorem \ref{Thm:yanguniqueintro}, it is natural to post the following conjecture, which is also hinted at in \cite{costello2014integrable}.

\begin{Conj}
    For any finite-dimensional DG Lie algebra $\fg$, let $\fd=T^*\fg$. The Hopf algebra $\CA_\hbar(\fd, \gamma)$ is isomorphic to $Y_\hbar(\fd)$ as a Hopf algebra. 
    
\end{Conj}

\subsection{Outlook: quantized Hitchin systems and the geometric Langlands correspondence}

In \cite{kapustin2006wilson, kapustin2006holomorphic}, the author considered the HT theory on $\Sigma\times C$ where $\Sigma$ is a smooth complex curve and $C$ a Riemann surface. The dimensional reduction on $\Sigma$ gives rise to a 2d B-model (on $C$) whose target is a natural generalization of the famous Hitchin moduli space:
\be
\CM_{G,V}(\Sigma)=\mathrm{Maps} (\Sigma, V/G),
\ee
namely now the moduli space of $G$-bundles on $\Sigma$ with a section of the associated $V$-bundle. The category of boundary conditions of this 2d B-model is $\Coh (\CM_{G,V}(\Sigma))$.  Choosing a point $z\in \Sigma$ and let $\Sigma^o$ denote the complement, then loop-group uniformization gives the following presentation of $\CM_{G,V}$:
\be
\CM_{G,V}(\Sigma)=\mathrm{Maps} (\mathbb{D}_z, V/G)\times_{\mathrm{Maps}(\mathbb{D}_z^\times, V/G)} \mathrm{Maps}(\Sigma^o, V/G).
\ee
Here $\mathbb{D}_z$ is the formal neighborhood of $z$ and $\mathbb{D}_z^\times$ the formal punctured neighborhood. Given such $z$, there is a natural action of $\Coh (\CM_{G,V}(\mathbb{B}))$ on $\Coh (\CM_{G,V}(\Sigma))$, via the so-called Hecke modifications. This action, especially its eigenspaces, is central to the study of geometric Langlands correspondence. 

A particularly important case is when \(V = \fg\) is the adjoint representation. In this case \(\CM_{G,\fg}(\Sigma)\) is the usual moduli space of Higgs bundles and the previous discussion is related to the Beilinson-Drinfeld quantization of Hitchin's integrable system and its application to the geometric Langlands correspondence from \cite{beilinson_drinfeld_quantization}. In the language of two-dimensional conformal field theory, the \(D\)-bundles on the moduli space of \(G\)-bundles on \(\Sigma\) in the geometric Langlands correspondence are given by the bundles of conformal blocks. 
In the punctured case, Felder \cite{felder_kzb} expressed the connections on these bundles using a classical dynamical \(r\)-matrix, which is shown in \cite{abedin2024r} to be the \(r\)-matrix of the punctured Hitchin system in the sense of the \(r\)-matrix approach in the theory of integrable models. The quadratic hamiltonians of the punctured Hitchin 
system as well as their quantizations to differential operators on the bundle of conformal blocks can consequently be expressed explicitly using this classical dynamical \(r\)-matrix.

Let us consider again the simplification to the formal completion at the identity coset. In particular, let us replace $\CM_{G,V}(\Sigma)$ with the double coset $\wh \CM_{G,V}(\Sigma)=\wh{\fh}(\CO)\!\setminus\!\wh{\fh}(\CK)/\wh{\fh}(\Sigma^o)$, where $\fh=\fg\ltimes V[-1]$. We hope that our construction can be adjusted to this setting in order to represent $\Coh (\wh \CM_{G,V}(\Sigma))$ using an explicit Hopf algebroid $\CA_\hbar(\Sigma)$ such that one can understand the action of Hecke modification in this simplified case as pull-back via an explicit algebra homomorphism
\be
\Delta_\hbar^\Sigma: \CA_\hbar(\Sigma)\longrightarrow  \CA_\hbar(\fd, \gamma)\otimes_{\C\lbb\hbar\rbb}\CA_\hbar(\Sigma)
\ee
satisfying the natural associativity condition.
In the case of \(V = \fg\), we hope that this approach might also lead to a quantization of the classical dynamical \(r\)-matrix from \cite{felder_kzb} to a quantum dynamical \(R\)-matrix. The action \(\Delta_\hbar^\Sigma\) could then be used to understand the relation between the $R$-matrix of $ \CA_\hbar(\fd, \gamma)$ and this quantization. We hope this approach could lead to applications in the geometric Langlands correspondence.

\subsection{Structure of the paper}

The paper is structured as follows. In Section \ref{sec:Takiff}, we introduce the topological Lie algebra $\fd(\CO)$, and its various coalgebra structures coming from a splitting of $\fg(\CO)\to\fg(\CK)$ supplemented by an $r$-matrix. In Section \ref{sec:quantakiff}, we construct a quantization of $U(\fd(\CO))$, and prove the uniqueness statement in the case of Yang's $r$-matrix $\gamma = \gamma_\fd$, from which we justify calling the quantization the Yangian $Y_\hbar(\fd)$. In Section \ref{sec:MeroR}, we construct the spectral $R$-matrix as a product of two meromorphic twisting matrices. In Section \ref{sec:RavGrG}, we comment on the relation between $Y_\hbar(\fd)$ and the geometry of equivariant affine Grassmannian.

\subsection{Acknowledgements}

R.A. would like to thank the Perimeter Institute for Theoretical Physics for its hospitality during the initial discussion of this project. The work of R.A. was supported by the DFG grant AB 940/1--1 and as part of the NCCR SwissMAP, a National
Centre of Competence in Research, funded by the Swiss
National Science Foundation (grant number 205607).

W.N. would like to express his sincere gratitude to Kevin Costello for suggesting this research project, and to Thomas Creutzig, Tudor Dimofte, Davide Gaiotto, Justin Hilburn and Harold Williams for stimulating discussions related to this topic. W.N. would also like to thank his friend Don Manuel for many encouragements. WN's research is supported by Perimeter Institute for Theoretical Physics. Research at Perimeter Institute is supported in part by the Government of Canada through the
Department of Innovation, Science and Economic Development Canada and by the Province
of Ontario through the Ministry of Colleges and Universities.

\section{Cotangent Lie algebras}\label{sec:Takiff}
In this section, we set up the classical objects of our examination: the loop algebra \(\fd(\CO)\) with coefficients in the cotangent Lie algebra \(\fd:=\fg \ltimes \fg^*\) of a finite-dimensional simple complex Lie algebra. We show that every generalized \(r\)-matrix \(r\) with coefficients in \(\fg\) defines a topological Lie bialgebra structure \(\delta\) on \(\fd(\CO)\). This paper will be concerned with the quantization of these Lie bialgebra structures, as well as a meromorphic analog of it. More precisely, we seek a Hopf algebra \(\CA_\hbar\) over \(\C[\![\hbar]\!]\) such that:
\begin{itemize}
    \item \(\CA_\hbar/\hbar\CA_\hbar\) is isomorphic to \(U(\fd(\CO))\) as a Hopf algebra;
    
    \item The coproduct \(\Delta_\hbar\) of \(\CA_\hbar\) satisfies \(\Delta_\hbar - \Delta_\hbar^{\textnormal{op}} = \hbar \delta + \CO(\hbar^2)\);
    
\end{itemize}
Moreover, when $r$ is a function of $t_1-t_2$, $\CA_\hbar$ will carry an action of time translation $T$, with which we can define the shifted coproduct $\Delta_{\hbar, z}=(\tau_z\otimes 1)\Delta_\hbar$ where $\tau_z=e^{zT}$. In this case, we seek to construct a quantum $R$ matrix $R(z)$ which is an element in  $(\CA_\hbar\otimes_{\C\lbb\hbar\rbb}\CA_\hbar)\lbb z^{-1}\rbb$ such that:
\begin{itemize}
    \item $(\tau_z\otimes 1) \Delta_{\hbar}^{\textnormal{op}}(a)=R(z) (\tau_z\otimes 1) \Delta_{\hbar}(a) R(z)^{-1}$.

    \item $R(z)$ satisfies quantum Yang-Baxter equation. 
    
\end{itemize}

A particularly important case is the Lie algebra associated to Yang's \(r\)-matrix \(\gamma\). In this case, we show that this quantization is the unique graded quantization of the associated Lie bialgebra, with respect to both the loop grading and an extra grading we call $\epsilon$-grading. Therefore, we refer to the resulting Hopf algebra as the Yangian of $\fd$. 

\subsection{Cotangent Lie algebras}
Let \(\fg\) be a finite-dimensional simple complex Lie algebra and consider the dual space \(\fg^*\) as \(\fg\)-module with respect to the coadjoint action \(\textnormal{ad}^*\), i.e.\ for every \(a \in \fg\) and \(f \in \fg^*\) the linear form \(a\cdot f \in \fg^*\) is defined by
\be
    (\textnormal{ad}^*(a) f)(b) = -f([a,b]) \textnormal{ for all }b \in \fg. 
\ee
The \emph{cotangent Lie algebra} \(\fd \coloneqq \fg \ltimes \fg^*\) is the semidirect product of \(\fg\) and \(\fg^*\). In particular, \(\fd\) is the direct sum of \(\fg\) and \(\fg^*\) as a vector space and is equipped with the Lie bracket 
\be
    [a_1 + f_1, a_2 + f_2] = [a_1,a_2] + \textnormal{ad}^*(a_1)f_2 - \textnormal{ad}^*(a_2)f_1
\ee
for all \(a_1,a_2 \in \fg\) and \(f_1,f_2 \in \fg^*\). Observe that \(\fg \subseteq \fd\) is a subalgebra and \(\fg^* \subseteq \fd\) is an abelian ideal. To keep track of the latter fact, we make use of the \(\epsilon\)-grading defined by
\be
    \textnormal{deg}_\epsilon(\fg) = 0\,,\qquad \textnormal{deg}_\epsilon(\fg^*) = 2.
\ee
Let us note that the cotangent Lie algebra $\fd$ is simply the classical double of the Lie bialgebra \(\fg\) equipped with the trivial cobracket \(0 \colon \fg \to \fg \otimes \fg\).

Let us once and for all chose a basis \(\{I_a\}_{a = 1}^{d} \subseteq \fg\) of \(\fg\) and denote by $\{I^a\}$ the dual basis. Let \(f_{ab}^c\) be the structural constant of $\fg$ under the basis \(\{I_a\}_{a = 1}^d\), then the commutation relation of the  algebra \(\fd\) is given by:
\begin{equation}
    [I_a,I_b] = \sum_{c = 1}^d f_{ab}^c I_c\,,\, [I_a,I^b] = -\sum_{c = 1}^df_{ac}^bI^c \textnormal{ and }[I^a,I^b] = 0.
\end{equation}

The Lie algebra $\fd$ comes equipped with a natural non-degenerate invariant symmetric bilinear form \(\kappa \colon \fd \times \fd \to \C\) defined by
\be
    \kappa(a_1 + f_1,a_2 + f_2) = f_1(a_2) + f_2(a_1)
\ee
for all \(a_1,a_2 \in \fg\) and \(f_1,f_2 \in \fg^*\).

Let us now fix some invariant bilinear form \(\kappa_0\) of \(\fg\), e.g.\ its Killing form, then 
\be
    \fd \cong \fg[\epsilon]/\epsilon^2 \fg[\epsilon]
\ee
as Lie algebras. Under this identification, 
\be
    \kappa(a_1 + \epsilon b_1,a_2 + \epsilon b_2) = \kappa_0(a_1,b_2) + \kappa_0(a_2,b_1)
\ee
for all \(a_1,a_2,b_1,b_2 \in \fg\). The \(\epsilon\)-grading mentioned above is now defined by \(\deg_\epsilon(1) = 0\) and \(\deg_\epsilon(\epsilon) = 2\), hence the name. In the following, we will treat \(\fg \ltimes \fg^*\) and \(\fg[\epsilon]/\epsilon^2\fg[\epsilon]\) as equal. In particular, we identify \(\fg^*\) and \(\epsilon \fg\). Moreover, we assume that \(\{I_a\}_{a = 1}^n\) is orthonormal with respect to \(\kappa_0\) and identify \(I^a\) with \(\epsilon I_a\).

In this work, we consider the topological Lie algebras \(\fd(\CO) \coloneqq \fd \otimes \CO\) with coefficients in \(\fd\) for \(\CO \coloneqq \C[\![t]\!]\). Here, topological simply refers to the fact that \(\fd(\CO)\) comes equipped with the \((t)\)-adic topology and the Lie bracket is continuous with respect to this topology.
An element in \(\fd(\CO)\) is of the form:
\be
\sum_{i \geq 0} J_i t^i, \textnormal{ with } J_i \in \fd
\ee
and the commutation relation is given by:
\be
\left[\sum_{i \geq 0} J_i t^i, \sum_{i \geq 0} J_i' t^i\right]=\sum_{i\geq 0}\sum_{j+k=i}[J_j, J_k'] t^i.
\ee
For each $J\in \fd$, we introduce generating series $J(u)=\sum_{n = 0}^\infty J_n u^{-n-1} \in \fd(\CO)[\![u^{-1}]\!]$, where $J_n = Jt^n \in \fd(\CO)$. In other words, we write:
\be
J(u)=\sum_{n = 0}^\infty J t^nu^{-n-1}=\frac{1}{u}\frac{J}{1-\frac{t}{u}}=\frac{J}{u-t}.
\ee
We can deduce that
\be
\begin{split}
\left[\frac{J}{u-t}, \frac{J'}{u-t}\right]&= \sum_{n,m = 0}^\infty [J_n, J'_m]u^{-m-n-2}=\sum_{k = 0}^\infty \sum_{m+n=k} [J,J']_{k} u^{-k-2}\\&=\sum_{k = 0}^\infty (k+1)[J,J']t^k u^{-k-2} =
\frac{[J,J']}{(u-t)^2}.    
\end{split}
\ee
The coefficients of \(I_a(u)\) and \(I^a(u)\) form a topological basis of \(\fd(\CO)\) and we can write
\be
    [I_a(u),I_b(u)] = \sum_{c = 1}^n\frac{f_{ab}^c I_c(u)}{u-t}, [I_a(u),I^b(u)] = -\sum_{c = 1}^n\frac{f_{ac}^b I^c(u)}{u-t} \textnormal{ and }[I^a(u),I^b(u)] = 0.
\ee

The Lie algebra \(\fd(\CO)\) is topologically graded by \(\Z \times \Z\) via
\be
\fd(\CO)_{(k,\ell)} = \begin{cases}
    z^k\fg & k \ge 0,\ell = 0;\\
    \epsilon z^k\fg & k \ge 0,\ell = 2;\\
    0 & \textnormal{otherwise}.
\end{cases}
\ee
In particular, the following holds:
\be
\begin{split}
    \fd(\CO) = \prod_{k,\ell \in \Z} \fd(\CO)_{(k,\ell)} \textnormal{ and } [\fd(\CO)_{(k_1,\ell_1)},\fd(\CO)_{(k_2,\ell_2)}] = \fd(\CO)_{(k_1+k_2,\ell_1+\ell_2)}.
\end{split}
\ee

The completed universal enveloping algebra 
\be
    U(\fd(\CO)) \coloneqq \varprojlim_k U(\fd(\CO)/z^k\fd(\CO))
\ee 
of the topological Lie algebra \(\fd(\CO)\) inherits the topological \(\Z \times \Z\)-grading from \(\fd(\CO)\). Moreover, it can be written as a completed smashed product 
\be
    U(\fd(O)) \cong U(\fg(\CO)) \# S(\fg^*(\CO)) \coloneqq \varprojlim_k \left(U(\fg(\CO)/z^k\fg(\CO)) \# S(\fg^*(\CO)/z^k\fg^*(\CO))\right). 
\ee 
Indeed, since \(\fg^*(\CO)\) is abelian \(U(\fg^*(\CO)) = S(\fg^*(\CO)) \coloneqq \varprojlim_k S(\fg^*(\CO)/z^k\fg^*(\CO))\) is the completed symmetric algebra and since \(\fd(\CO) = \fg(\CO) \oplus \fg^*(\CO)\) we have 
\be
    U(\fd(\CO)) = U(\fg(\CO)) \otimes S(\fg^*(\CO)) \coloneqq \varprojlim_k \left(U(\fg(\CO)/z^k\fg(\CO)) \otimes S(\fg^*(\CO)/z^k\fg^*(\CO))\right)
\ee
as \(\C\)-algebras due to the PBW theorem. Clearly, \(U(\fg(\CO))\) is a cocommutative Hopf subalgebra of \(U(\fd(\CO))\) and \(S(\fg^*(\CO))\) is a commutative and cocommutative Hopf ideal of \(U(\fd(\CO))\). Furthermore, 
\be
    (x \# f)(y \# g) = x y \# f g -  x \# (\textnormal{ad}^*(y)f) g = \sum_{(y)} xy_{(1)} \# (f \cdot y_{(2)}) g 
\ee
is easily verified for \(x \in U(\fg(\CO)), f,g \in S(\fg^*(\CO))\) and \(y \in \fg(\CO)\). 

Let us note that the derivation \(T \coloneqq \partial_t\) of \(\fd(O)\) extends to a Hopf algebra derivation of \(U(\fd(O))\) which will become relevant later.

\subsection{Classical $r$-matrices and Lie bialgebra structures on \(\fd(\CO)\)}\label{sec:lie_bialgebra_structures}
A topological Lie bialgebra structure on \(\fd(\CO)\) is by definition a linear map 
\be
    \delta \colon \fd(\CO) \to \fd(\CO) \otimes \fd(\CO),
\ee
where \(\otimes\) denotes the completed tensor product, satisfying the following: 
\begin{enumerate}
    \item \(\delta\) is a 1-cocycle:
\be\label{eq:delta_cocycle}
    \delta([x,y]) = [x \otimes 1 + 1 \otimes x,\delta(y)]-[y \otimes 1 + 1 \otimes y,\delta(x)]\,,\qquad \forall x,y \in \fd(\CO);
\ee
\item \(\delta\) satisfies the co-Jacobi identity:
\be
    (\delta \otimes 1)\delta -(1 \otimes \delta)\delta +  (\tau \otimes 1)(1 \otimes \delta)\delta = 0.
\ee
Here, \(\tau\) is the tensor flip.
\end{enumerate}
Let us note that the cocycle condition \eqref{eq:delta_cocycle} combined with \(\fg = [\fg,\fg]\) implies that \(\delta\) is continuous in the \((t)\)-adic topology. In particular, \(\delta\) is completely determined by its values on the dense subset \(\fd[t]\) of \(\fd(\CO) = \fd[\![t]\!]\).

The simplest non-trivial topological Lie bialgebra structure on \(\fd(O)\) is given by
\begin{equation}
    \delta_\gamma(x) = [x(t_1) \otimes 1 + 1 \otimes x(t_2),\gamma(t_1,t_2)],
\end{equation}
where \(\gamma\) is Yang's \(r\)-matrix for \(\fd\). More precisely, the canonical invariant bilinear form \(\kappa\) of \(\fd\) defines an invariant symmetric tensor \(C = C_\fd =\sum_{a = 1}^n (I_a \otimes I^a + I^a \otimes I_a) \in \fd \otimes \fd\) and 
\begin{equation}\label{eq:yang_rmat_d}
    \gamma(t_1,t_2) = \frac{C}{t_1-t_2}.
\end{equation}
The topological classical double of \((\fd(\CO),\delta)\) is simply \(\fd(\CK) = \fd \otimes \CK\), where \(\CK = \CO[t^{-1}] = 
\C(\!(t)\!)\). Indeed, this follows by the fact that the Manin triple (i.e.\ Lagrangian Lie algebra splitting) \(\fd(\CK) = \fd(\CO) \oplus t^{-1}\fd[t^{-1}]\) determines \(\delta_\gamma\) in the sense that
\be
    \textnormal{res}_{t = 0}\kappa(\delta(x),w_1 \otimes w_2) = \textnormal{res}_{t = 0}\kappa(x,[w_1,w_2])
\ee
for all \(x \in \fd(\CO), w_1,w_2 \in t^{-1}\fd[t^{-1}]\).
Observe that \(\delta_\gamma\) is \(\epsilon\)-graded in the sense that \(\fd(\CO)' = t^{-1}\fd[t^{-1}]\) is \(\epsilon\)-graded.

We can also consider more general topological Lie bialgebra structures defined by \(r\)-matrices. To do so, let us recall the well-known correspondence of \(r\)-matrices with coefficients in a finite-dimensional complex Lie algebra \(\fp\) equipped with a non-degenerate invariant bilinear form \(\beta\) and subalgebras of \(\fp(\CK)\); see e.g.\ \cite{cherednik_becklund_darboux,etingof_schiffmann,skrypnyk_infinite_dimensional_Lie_algebras}. The following statements are proven exactly as their analogs in \cite[Section 1]{abedin_universal_geometrization} for semisimple \(\fp\):

\begin{itemize}
    \item For any solution \be\label{eq:normal_form}
            r(t_1,t_2) = \gamma_\fp + g(t_1,t_2) = \sum_{k=0}^\infty \sum_{i = 1}^{d}r_{k,i}(t_1) \otimes b_it_2^k\,,\qquad g \in \fp(\CO)\otimes \fp(\CO),
        \ee
        of the generalized classical Yang-Baxter equation 
        \be\label{eq:GCYBE}
            [r^{12}(t_1,t_2),r^{13}(t_1,t_3)] + [r^{12}(t_1,t_2),r^{23}(t_2,t_3)] + [r^{32}(t_3,t_2),r^{13}(t_1,t_3)] = 0
        \ee
        the subspace \(\fp(r) \coloneqq \bigoplus_{n = 0}^\infty \bigoplus_{a = 1}^d \C r_{a,n} \subseteq \fp(\CK)\) is a subalgebra complementary to \(\fp(\CO)\). Here, \(\{b_i\}_{a = 1}^d \subseteq \fp\) is an orthonormal basis with respect to the pairing \(\beta\) of \(\fp\), \(C_\fp = \sum_{i = 1}^db_i \otimes b_i\in \fp \otimes \fp\) is the quadratic Casimir element of \(\fp\), and \(\gamma_\fp = \frac{C_\fp}{t_1-t_2}\) is Yang's \(r\)-matrix for \(\fp\). Solutions of \eqref{eq:GCYBE} of the form \eqref{eq:normal_form} will be called generalized \(r\)-matrices with coefficients in \(\fp\).

        \item The equality \(\fp(r) = \fp(r)^\bot\), where the orthogonal complement is taken with respect to the bilinear form \((x,y) \mapsto \textnormal{res}_{t = 0}\beta(x(t),y(t))\), holds if and only if \(r\) is skew-symmetric, i.e.\ \(r(t_1,t_2) = -\tau(r(t_2,t_1))\). In this case, \(r\) solves the 
        classical Yang-Baxter equation 
        \be\label{eq:CYBE}
            [r^{12}(t_1,t_2),r^{13}(t_1,t_3)] + [r^{12}(t_1,t_2),r^{23}(t_2,t_3)] + [r^{13}(t_1,t_3),r^{23}(t_2,t_3)] = 0
        \ee
        and \(r\) is simply called \(r\)-matrix with coefficients in \(\fp\). 

        \item The subalgebra \(\fp(r)\) is stable under the derivation \(\partial_t\) if and only if \(r\) depends on the difference \(t_1-t_2\) of its variables, i.e.\ \(r(t_1,t_2) = \Tilde{r}(t_1-t_2)\) for some \(\Tilde{r} \in (\fp \otimes \fp)(\!(t)\!)\). By abuse of notation, we simply write \(r(t_1,t_2) = r(t_1-t_2)\) in this case.
\end{itemize}
For example, Yang's \(r\)-matrix \(\gamma_\fp\) with coefficients in \(\fp\) is indeed a solution to the classical Yang-Baxter equation for any \(\fp\) and \(\fp(\gamma) = t^{-1}\fp[t^{-1}] \eqqcolon \fp_{<0}\). 
Using the facts above for \(\fp \in \{\fg,\fd\}\), we obtain the following general result about topological Lie bialgebra structures on \(\fd(\CO)\).

\begin{Prop}
    \label{lem:lie_bialgebra_structures}
    There is a bijection between:
    \begin{enumerate}
        \item \(\epsilon\)-graded topological Lie bialgebra structures \(\delta\) on \(\fd(\CO)\) with topological double \(\fd(\CK)\);

        \item Generalized \(r\)-matrices with coefficients in \(\fg\).
    \end{enumerate}
    Explicitly, for every generalized \(r\)-matrix \(r\) with coefficients in \(\fg\) the series
    \begin{equation}\label{eq:rhoandr}
        \rho = (1 \otimes \epsilon)r(t_1,t_2) - (\epsilon \otimes 1)\tau(r(t_2,t_1)) \in \fd(\CK) \otimes \fd(\CK)
    \end{equation}  
    is an \(r\)-matrix with coefficients in \(\fd\) and 
    \be\label{eq:delta_r}
        \delta_\rho(x) = [\rho(t_1,t_2),x(t_1)\otimes 1 + 1 \otimes x(t_2)]
    \ee
    is the associated topological Lie bialgebra structure on \(\fd(\CO)\).
\end{Prop}
\begin{proof}
    By definition, an \(\epsilon\)-graded topological Lie bialgebra structure on \(\fd(\CO)\) with topological double \(\fd(\CK)\) is determined by an Lagrangian \(\epsilon\)-graded subalgebra \(V \subseteq \fd(\CK)\) complementary to \(\fd(\CO)\). In particular, \(V = W \oplus \epsilon W^\bot\) for a Lie subalgebra \(W \subseteq \fg(\CK)\) complementary to \(\fg(\CO)\). In particular, there exists a unique generalized \(r\)-matrix \(r\) with coefficients in \(\fg\) such that \(W = \fg(r)\). It is easy to see that \(V = \fd(\rho)\). Since \(V\) is a Lagrangian subalgebra, we deduce that \(\rho\) is an \(r\)-matrix with coefficients in \(\fd\).

    In order to see that \eqref{eq:delta_r} is indeed the Lie cobracket dual to the Lie bracket of \(\fg(r) \oplus \epsilon \fg(r)^\bot\), let us rewrite \eqref{eq:CYBE} as
    \begin{equation}
        \begin{split}
            &\sum_{i,j = 0}^\infty \sum_{a,b = 1}^d \left([r_{a,i},r_{b,j}] \otimes I^a_{i} \otimes I^b_j + [r_{a,i},\epsilon \overline{r}_{b,j}] \otimes I_i^a \otimes I_{b,j} + [\epsilon \overline{r}_{a,i}, r_{b,j}] \otimes I_{a,i} \otimes I^j_b\right)
            \\&= -\sum_{n = 0}^\infty \sum_{a = 1}^d \left(r_{a,n} \otimes [I^a_n \otimes 1 + 1 \otimes I^a_n,\rho] + \epsilon \overline{r}_{a,n} \otimes [I_{a,n} \otimes 1 + 1 \otimes I_{a,n}, \rho]\right)
            \\& = \sum_{n = 0}^\infty \left(r_{a,n} \otimes \delta_\rho(I^a_n) + \epsilon \overline{r}_{a,n} \otimes \delta_{\rho}(I_{a,n})\right),
        \end{split}
    \end{equation}
    where we wrote \(\tau(r) = \sum_{n = 0}^\infty \sum_{a = 1}^d \overline{r}_{a,n} \otimes I_{a,n}\).
\end{proof}

In the following, we refer to \(r\)-matrices with coefficients in \(\fd\) of the form \eqref{eq:rhoandr} as \(\epsilon\)-graded.

\subsection{Rational \(r\)-matrices and Lie bialgebra structures on \(\fd[t]\)}
It is easy to see that the Lie bialgebra structures \(\delta_\rho\) on \(\fd(\CO)\) for an \(\epsilon\)-graded \(r\)-matrix \(\rho\) with coefficients in \(\fd\) restricts to a (usual non-topological) Lie bialgebra structure on \(\fd[t] \subseteq \fd(\CO) = \fd[\![t]\!]\) if and only if \(\rho\) is a rational function of its variables. Equivalently, the underlying generalized \(r\)-matrix \(r\) with coefficients in \(\fg\) is of the form \(r \in \gamma_\fg + (\fg \otimes \fg)[t_1,t_2]\). The (generalized) \(r\)-matrices \(\rho\) and \(r\) are referred to as rational under this circumstance. On the level of subalgebras, this is equivalent to the fact that \(\fg(r)\) is bounded: \(t^{-N}\fg_{<0} \subseteq \fg(r) \subseteq t^{N}\fg_{<0}\) for some \(N \in \N\). Proposition \ref{lem:lie_bialgebra_structures} has the following adjustment to this setting.

\begin{Cor}
    \label{cor:lie_bialgebra_structures}
    The bijection in Proposition \ref{lem:lie_bialgebra_structures} gives rise to a bijection between:
    \begin{enumerate}
        \item \(\epsilon\)-graded Lie bialgebra structures \(\delta\) on \(\fd[t]\) with (non-topological) double \(\fd(\!(t^{-1})\!)\);

        \item Rational generalized \(r\)-matrices with coefficients in \(\fg\).
    \end{enumerate}
    \end{Cor}

\subsection{Classical \(r\)-matrices and quasi-classical vertex algebras}\label{sec:quasi_classical_vertex}
Let \(r\) be a generalized \(r\)-matrix with coefficients in \(\fg\) and \(\rho\) its associated \(r\)-matrix with coefficients in \(\fd\). As mentioned in Section \ref{sec:lie_bialgebra_structures}, the subalgebra \(\fd(\rho) \subseteq \fd(\CK)\) is stable under the application of the differential operator \(T = \partial_t\) if and only if \(\rho\), and consequently \(r\), depends on the difference of its variables, i.e.\ \(\rho(t_1,t_2) = \rho(t_1-t_2)\).
In this case, \(S(\fd(\rho))\) becomes a commutative vertex algebra if equipped with the natural extension of the derivation \(T \colon \fd(\rho) \to \fd(\rho)\). The vertex operation is simply given by \(Y(A,z)B \coloneqq Y(z)(A \otimes B) \coloneqq e^{zT}(A)B\) for any \(A,B \in S(\fd(\rho))\).

The \(r\)-matrix \(\rho\), or more specifically 
\begin{equation}\label{eq:def_rho}
    \varrho(z) \coloneqq \rho(t_1 + z-t_2) \in (\fd \otimes \fd)[t_1,t_2][\![z,z^{-1}]\!], 
\end{equation}
defines a quasi-classical structure on the vertex algebra \(S(\fd(\rho))\) in the sense of \cite{Etingof2000quantumvertex}. In particular, for any \(a \in \fd(\CO)\), we can consider \(\tau_z(a) \coloneqq a(t+z)\) as a pseudoderivative of \(S(\fd(\rho))\) and therefore \(\varrho\) defines a linear map \(S(\fd(\rho)) \otimes S(\fd(\rho)) \to (S(\fd(\rho)) \otimes S(\fd(\rho)))(\!(z)\!)\) such that:
\begin{enumerate}
    \item \([\varrho^{12}(z_1-z_2),\varrho^{13}(z_1-z_3)] + [\varrho^{12}(z_1-z_2),\varrho^{23}(z_2-z_3)] + [\varrho^{13}(z_1-z_3),\varrho^{23}(z_2-z_3)] = 0\);
    \item \(\varrho^{21}(z) = -\varrho(-z)\);
    \item \([T \otimes 1,\varrho(z)] = -\partial_z \varrho(z)\);
    \item \(\varrho(z_1)(Y(z_2) \otimes 1) = Y(z_2)(\varrho^{23}(z_1) + \varrho^{13}(z_1+z_2))\).
\end{enumerate}
Observe that the definition of the pseudoderivatives associated to \(\varrho\) combined with \eqref{eq:def_rho} implies that \(\varrho\) really takes values in \((S(\fd(\rho)) \otimes S(\fd(\rho))(\!(z)\!)\) and not \((S(\fd(\rho)) \otimes S(\fd(\rho)))[\![z,z^{-1}]\!]\).

Another point of view is that \(\fd(\CK)\) is also equipped with a meromorphic Lie bialgebra structure, in the sense that there is a co-bracket $\delta_{\rho,z}\colon \fd(\CK) \to (\fd(\CK)\otimes \fd(\CK))[\![z,z^{-1}]\!]$, by setting:
\be
\delta_{\rho,z} (x) =[\rho(t_1+z-t_2),x(t_1+z)\otimes 1+1\otimes x(t_2)] = (\tau_z \otimes 1)\delta_\rho(x).
\ee
Let us note that \(\delta_{\rho,z}(\fd(\CO)) \subseteq (\fd(\CO)\otimes \fd(\CO))[\![z]\!]\). More precisely, we have the following result.

\begin{Lem}
    The map \(\delta_{\rho,z}\) satisfies the following translated analogs of the Lie bialgebra properties:
\begin{enumerate}
    \item \(\delta_{\rho,z}([x,y])=[\tau_z (x)\otimes 1+1\otimes x,\delta_{\rho,z} (y)]-[\tau_z(y)\otimes 1+1\otimes y,\delta_{\rho,z}(x)]\) for all \(x,y\in\fd(\CO)\);
    \item \((\delta_{\rho,z_1-z_2}\otimes 1)\delta_{\rho,z_2}-(1\otimes \delta_{z_2})\delta_{\rho,z_1}+(\tau \otimes 1)(1\otimes \delta_{\rho,z_1})\delta_{\rho,z_2} = 0.\)
\end{enumerate}
\end{Lem}
\begin{proof}
    The identity 1.\ follows immediately from applying \(\tau_z \otimes 1\) to the cocycle condition of \(\delta\). Observe that 
    \begin{equation}
        (\tau_{z_1} \otimes \tau_{z_2})\delta_\rho(x) = [\varrho(t_1-t_2 + z_1-z_2), x(t_1+z_1)\otimes 1 + 1 \otimes x(t_2 + z_2)] = \delta_{\rho,z_1-z_2}(\tau_{z_2}(x))
    \end{equation}
    Therefore, we can calculate
    \begin{equation*}
        \begin{split}
            &(\tau_{z_1} \otimes \tau_{z_2} \otimes 1)(\delta_\rho \otimes 1)\delta = (\delta_{\rho,z_1-z_2} \otimes 1)(\tau_{z_2} \otimes 1)\delta = (\delta_{\rho,z_1 - z_2} \otimes 1)\delta_{\rho,z_2};\\
            &(\tau_{z_1} \otimes \tau_{z_2} \otimes 1)(1 \otimes \delta_\rho)\delta_\rho= (1 \otimes (\tau_{z_2} \otimes 1)\delta_\rho)(\tau_{z_1} \otimes 1)\delta_\rho = (1 \otimes\delta_{\rho,z_2})\delta_{\rho,z_1};\\
            & (\tau_{z_1} \otimes \tau_{z_2} \otimes 1)(\tau \otimes 1)(1 \otimes \delta_\rho)\delta_\rho = (\tau \otimes 1)(\tau_{z_2}\otimes \tau_{z_1}\otimes1)(1\otimes \delta_\rho)\delta_\rho = (\tau \otimes 1)(1 \otimes \delta_{\rho,z_1})\delta_{\rho,z_2}.
        \end{split}
    \end{equation*}
    proving 2.
\end{proof}
Let us note that if \(\rho\) is rational
\be
\delta_{\rho,z}(\fd[t]) \subseteq (\fd \otimes \fd)[t_1,t_2,z] \textnormal{ and }\delta(\fd(\!(t^{-1})\!)) \subseteq (\fd(\!(t^{-1})\!) \otimes \fd(\!(t^{-1})\!))(\!(z^{-1})\!).
\ee

\subsection{Generalization and associated Poisson-Lie groups}\label{sec:generalization_PoissonLie}
Actually, the construction in Section \ref{sec:lie_bialgebra_structures} is a special case of the following construction of Lie bialgebra structures. Let \(\fp\) be a Lie algebra and \(\fp = \fp_+ \oplus \fp_-\) be a decomposition into closed subalgebras. Then \(\fp \ltimes \fp^* = (\fp_- \ltimes \fp_+^*) \oplus (\fp_+ \ltimes \fp_-^*)\) defines a Manin triple. In particular, if \(\fp_-\) is finite-dimensional, then \(\fp_- \oplus \fp_+^*\) becomes a Lie bialgebra if equipped with the dual of the Lie bracket of \(\fp_+ \ltimes \fp_-^*\). For infinite-dimensional \(\fp\), the dual of the Lie bracket of \(\fp_+ \ltimes \fp_-^*\) does not necessarily take values in the double tensor product of \(\fp_- \oplus \fp_+^*\). However, if we consider appropriate topological Lie algebras \(\fp\) and let \(\fp^*\) denote the continuous dual of \(\fp\), it will take values in a completion of this tensor product and we obtain a topological Lie bialgebra structure on \(\fp_- \oplus \fp_+^*\) in this way. For example, if we put \(\fp = \fg(\CK)\) and \(\fp_+ = \fg(\CO)\) we obtain precisely the Lie bialgebra structures described in Lemma \ref{lem:lie_bialgebra_structures} in this way. We will discuss the quantization of Lie bialgebras defined by Lie algebra decompositions \(\fp = \fp_+ \oplus \fp_-\) in Section \ref{subsubsec:crossprod}.

Lie bialgebras are precisely infinitesimal Poisson-Lie groups. 
Although not of immediate importance to this work, it is interesting to note what the Poisson-Lie groups associated to the Lie bialgebras described in Lemma \ref{lem:lie_bialgebra_structures} are. Therefore, let us assume that \(\fp\) has a meaningful integration to a group \(P\), e.g.\ if \(\fp\) is finite-dimensional, in which case we can consider the simply connected Lie group \(P\) of \(\fp\), or if \(\fp = \fg(\CK)\), in which case \(P = G(\CK)\) is the loop group.

In these cases, the Lie algebra of \(T^*P\) is \(\fp \ltimes \fp^*\), so Poisson brackets on \(T^*P\) compatible with the multiplication of this group correspond to Lie bialgebra structures on \(\fp \ltimes \fp^*\). Therefore, the Lie bialgebra structures on \(\fp \ltimes \fp^*\) defined by Lie algebra decompositions \(\fp = \fp_+ \oplus \fp_-\) correspond to Poisson-Lie structures on \(T^*P\).

\section{Quantization of cotangent Lie algebras}\label{sec:quantakiff}

In this section, we construct a quantization of the Lie bialgebra structures described in Lemma \ref{lem:lie_bialgebra_structures}. The main results of this section can be summarized as follows.

\begin{Thm}\label{Thm:quantizerho}
    Let \(\rho\) be an \(\epsilon\)-graded \(r\)-matrix with coefficients in \(\fd\) and \(\delta_\rho\) be the associated Lie bialgebra structure on \(\fd(\CO)\) given by \eqref{eq:delta_r}. 
    
    There exists a canonical $\epsilon$-graded quantization \(\CA_\hbar(\fd,\rho)\) of \(\delta_\rho\). More precisely, \(\CA_\hbar(\fd,\rho)\) is a topological \(\epsilon\)-graded Hopf algebra over \(\C[\![\hbar]\!]\) such that:
    \begin{enumerate}

        \item Its coproduct \(\Delta_{\rho,\hbar}\colon  \CA_\hbar(\fd,\rho)\to \CA_\hbar(\fd,\rho)\otimes_{\C\lbb\hbar\rbb} \CA_\hbar(\fd,\rho)\) satisfies \(\Delta_{\rho,\hbar} - \Delta_{\rho,\hbar}^{\textnormal{op}} = \hbar \delta_\rho + \CO(\hbar^2)\).
        
    \item If \(r\) depends on the \(t_1-t_2\), \(\CA_\hbar(\fd,\rho)\) also quantizes \(\delta_{\rho,z}\). In particular, there exists a nilpotent Hopf algebra differential $T$ of \(\CA_\hbar(\fd,\rho)\) and another holomorphic coproduct
    \be
        \Delta_{\rho,\hbar, z}\coloneqq (\tau_z\otimes 1)\Delta_{\rho,\hbar}\colon \CA_\hbar(\fd,\rho)\to \CA_\hbar(\fd,\rho)\otimes_{\C\lbb\hbar\rbb} \CA_\hbar(\fd,\rho)
    \ee
    satisfying \(\Delta_{\rho,\hbar,z} - \Delta_{\rho,\hbar,z}^{\textnormal{op}} = \hbar \delta_{\rho,z} + \CO(\hbar^2)\). 

    \item For another \(\epsilon\)-graded \(r\)-matrix \(\rho'\) with coefficients in \(\fd\), the Hopf algebras \(\CA_\hbar(\fd,\rho)\) and \(\CA_\hbar(\fd,\rho')\) are related by a topolgical quantum twist.

    \item The Hopf algebra $\CA_\hbar(\fd, \rho)$ and the quantum twist $F$ have well-defined evaluation at $\hbar=\xi$ for any $\xi \in \C$. 
    
    \item Let $ \CA^\circ_\xi(\fd,\rho)$ be the dense subalgebra of the topological Hopf algebra $\CA_\hbar(\fd,\rho)$ generated over $\C\lbb\hbar\rbb$ by $\fd[t]$. The holomorphic coproduct $\Delta_{\rho,\hbar, z}$ has well-defined evaluation on products of finite-dimensional modules for any $\hbar=\xi\in \C$ and $z=s\in \C^\times$. 

    \item If \(\rho\) is rational, \(\CA^\circ_\hbar(\fd,\rho)\) is a Hopf subalgebra of \(\CA_\hbar(\fd,\rho)\) quantizing the Lie bialgebra \((\fd[t],\delta_\rho)\).
    \end{enumerate}
\end{Thm}

Now for the Lie bialgebra structure defined by Yang's \(r\)-matrix $\gamma$, the results above admits important refinements. Namely, in this case the constructed quantization of $U(\fd(\CO))$ is bi-graded in the sense outlined at the beginning of Section \ref{sec:Takiff} and we show that this is the unique graded quantization respecting both loop grading and $\epsilon$-grading. This motivates us to call the arising Hopf algebra Yangian of \(\fd\) and denoting it by $\CA_\hbar(\fd, \gamma) \coloneqq Y_\hbar (\fd)$ in this case. The following summarizes the refined features of this special case. 

\begin{Thm}\label{Thm:quantizegamma}
    There exists a bi-graded algebra $Y_\hbar(\fd)$ over $\C\lbb\hbar\rbb$ with the following structures, all preserving the bi-grading and all continuous. 
    \begin{enumerate}

        \item It admits a coproduct:
    \be
        \Delta_\hbar: Y_\hbar(\fd)\to Y_\hbar(\fd)\otimes_{\C\lbb\hbar\rbb} Y_\hbar(\fd),
    \ee
    which is a quantization of \(\delta_\gamma\).

  \item There exists a Hopf algebra differential $T$ that acts nilpotently on the generators $\fd[t]$. It defines another holomorphic coproduct:
  \be
\Delta_{\hbar, z}:=(e^{zT}\otimes 1)\Delta_\hbar: Y_\hbar(\fd)\to (Y_\hbar(\fd)\otimes_{\C\lbb\hbar\rbb} Y_\hbar(\fd))\lbb z\rbb,
  \ee
which is a quantization of $\delta_{\gamma,z}$. 

\item Such a bi-graded Hopf algebra quantization is unique,  justifying calling $Y_\hbar(\fd)$ the ``Yangian of $\fd$". 

\item For every \(\epsilon\)-graded \(r\)-matrix \(\rho\) with coefficients in \(\fd\), the Hopf algebra \(\CA_\hbar(\fd,\rho)\) is a quantum twist of the Yangian $Y_\hbar(\fd)$.

\item Let $ Y^\circ_\hbar(\fd)$ be the dense subalgebra of the topological Hopf algebra $Y_\hbar(\fd)$ generated over $\C\lbb\hbar\rbb$ by $\fd[t]$. This is a Hopf subalgebra and the coproduct has well-defined evaluation at any $\hbar=\xi\in \C$, whereas the holomorphic coproduct $\Delta_{\hbar, z}$ has well-defined evaluation on the product of finite-dimensional modules for any $\hbar=\xi\in \C$ and $z=s\in \C^\times$. 

    \end{enumerate}
    
\end{Thm}

We comment that the above structures will endow the category $Y^\circ_\xi(\fd)\Mod$, the category of finite-dimensional modules of $Y^\circ_\xi(\fd)$ the structure of a tensor category, as well as a meromorphic tensor category in the sense of \cite{soibelman1999meromorphic}.  

% This section is structured as follows. In Section \ref{subsec:quantHopf}, we construct the quantization of $\fp_-\ltimes \fp_+^*$ for any Lie algebra splitting $\fp=\fp_+\oplus \fp_-$, using the matched product of $U(\fp_+)$ with $U(\fp_-)$. In Section \ref{subsec:loopapply}, we apply the construction of the previous section to $\fg(\CK)=\fg(\CO)\oplus \fg(r)$.  In Section \ref{subsec:unique}, we prove that there is a unique bi-graded quantization of $U(\fd(\CO))$. 

\subsection{Graded quantization of $U(\fd(\CO))$}\label{subsec:quantHopf}

\subsubsection{Idea of the quantization}
The idea of the quantization is related to the geometry of the double quotient, and can be applied to more general situations. Suppose we have an affine group $G$ which can be represented by a cross-product of subgroups:
\be
G=H\match K.
\ee
The stack $[K\!\setminus \! G/K]$ has the following correspondence, allowing one to define a monoidal structure on $\Coh_K(G/K)$:
\be\label{eq:corrGK}
\btik
& (G\times_K G)/K \arrow[dr]\arrow[dl]&\\
G/K\times G/K & & G/K
\etik
\ee
Here the first map is given by $[g_1, g_2]\to [g_1]\times [g_2]$ and the second given by multiplication $[g_1,g_2]\to [g_1g_2]$. In the case when $G=H\match K$, the quotient $G/K$ is isomorphic to $H$ and the left-action \((k,h) \mapsto k \cdot h\) of $K$ on this quotient is given by the unique decomposition \(kh = (k \cdot h) k^h\) into unique elements \(k \cdot h \in H\) and \(k^h \in K\). We can then describe $\Coh_K(G/K)$ as the category of \(K\)-equivariant \(\C[H]\)-modules, and hope to obtain the monoidal structure by using the coproduct of $\C[H]$ coming from the group multiplication and the action of $K$. However, the action of $K$ on $(G\times_K G)/K\cong H\times H$ is given by a twisted action:
\be
k\cdot  (h_1, h_2)=(k\cdot h_1,k^{h_1}\cdot h_2),
\ee
suggesting that the coproduct on $K$ must be modified. 

Let us consider the case when $H$ and $K$ are finite groups. In this case, denote by $\C[H]$ functions on $H$ and $\C\cdot K$ the group algebra of $K$, and the category $\Coh_K(G/K)$ is simply the category of modules of $A\coloneqq\C\cdot K\# \C[H]$, where $kfk^{-1}=f\circ k^{-1}$.

It is not difficult to show that the monoidal structure of the above correspondence can be re-written in terms of the above coproduct on $A$:
\be
\Delta (\delta_h)=\sum_{h_ih_j=h} \delta_{h_i}\otimes \delta_{h_j},\qquad \Delta (k)=\sum_{h\in H}k\delta_h \otimes k^h.
\ee
Here, \(k \in K\) and for \(h \in H\) the map \(\delta_h \colon H \to \C\) is defined by \(\delta_h(h') = 0\) for \(h' \neq h\) and \(\delta_h(h) = 1\).
Note that this indeed defines a well-defined algebra homomorphism, since for example:
\be
\Delta (k_1)\Delta(k_2)=\sum_{h,h' \in H} k_1\delta_hk_2\delta_{h'}\otimes k_1^h k_2^{h'}=\sum_{h' \in H} k_1k_2\delta_{h'} \otimes k_1^{k_2\cdot h'}k_2^{h'}=\Delta (k_1k_2).
\ee
Here we used the fact that $\delta_{h_1}\delta_{h_2}=\delta_{h_1,h_2}\delta_{h_2}$, and that $(k_1k_2)^h=k_1^{k_2\cdot h}k_2^h$. 

In the following sections, we apply this to Lie algebra splittings $\fg(\CK) = \fg(\CO)\oplus \fg(r)$ of the loop algebra.

\subsubsection{Cross product of Lie algebras and formal groups}\label{subsubsec:crossprod}
Let us return to the setting of Section \ref{sec:Takiff}. Let \(r\) be a generalized \(r\)-matrix with coefficients in \(\fg\) and \(\rho\) be the associated \(\epsilon\)-graded \(r\)-matrix with coefficients in \(\fd\) given by \eqref{eq:rhoandr}.
Consider the Lie algebra decomposition $\fg(\CK) = \fg(\CO) \oplus \fg(r)$ associated to \(r\). Let us recall that the tensor coefficients \(\{r_{a,n}|1\le a\le d,n \ge 0\} \subseteq \fg(r)\) of \(r\) is precisely the dual basis to the topological basis \(\{I^a_n = \epsilon I_at^n\mid 1\le a \le d,n\ge 0\}\) of \(\fg^*(\CO) = \epsilon\fg(\CO) = \fg(r)^*\), where \(\{I_a\}_{a = 1}^d\) was a fixed orthonormal basis of \(\fg\).

The Lie algebra bracket of $\fg(\CK)$ provides:
\begin{itemize}
    \item The left action $\phi_+ \coloneqq p_-\textnormal{ad}|_{\fg(\CO)} \colon \fg(\CO) \to \textnormal{End}(\fg(r))$, where \(p_- \colon \fg(\CK) \to \fg(\CK)\) is the canonical projection onto \(\fg(r)\). 

    \item The right action $\phi_- \coloneqq p_+\textnormal{ad}|_{\fg(r)} \colon \fg(r) \to \textnormal{End}(\fg(\CO))$, where \(p_+ \colon \fg(\CK) \to \fg(\CK)\) is the canonical projection onto \(\fg(\CO)\).

    \item The compatibility $[x,y]=\phi_+(x)y+x\phi_-(y)$, where $x\in \fg(\CO)$ and $y\in \fg(r)$ and we write \(x\phi_-(y) \coloneqq -\phi_-(y)x\). 
    
\end{itemize}
The compatibility condition is to ensure that the Lie bracket satisfies Jacobi identity. For example, for $x_1, x_2\in \fg(\CO)$ and $y\in \fg(r)$, the Jacobi identity:
\be
[[x_1,x_2],y]=[[x_1, y], x_2]+[x_1, [x_2,y]]
\ee
translates to the following condition for \(\phi_-\):
\be
    [x_1,x_2]\phi_-(y)=[x_1\phi_-(y), x_2]+[x_1, x_2\phi_-(y)]-x_2\phi_-(\phi_+(x_1)y)+x_1\phi_-(\phi_+(x_2)y).
\ee
Similarly, the condition for $\phi_+$ reads
\be
    \phi_+(x)[y_1,y_2] = [\phi_+(x)y_1,y_2] + [y_1,\phi_+(x)y_2] + \phi_+(x\phi_-(y_1))y_2  - \phi_+(x\phi_-(y_2))y_1
\ee
for all \(x \in \fg(\CO)\) and \(y_1,y_2 \in \fg(r)\).

The decomposition of the Lie algebra induces a decomposition of the algebra
\be
    U(\fg(\CK))=U(\fg(r)) \otimes U(\fg(\CO)).
\ee
Here, we recall that we complete the tensor product and the universal enveloping algebras using the \((t)\)-adic topology.
This induces a matched pair between $U(\fg(\CO))$ and $U(\fg(r))$ as Hopf algebras, making the algebra $U(\fg(\CK))$ into a bicrossed product of Hopf algebras \cite{majid1990physics}. This means, in particular that there is a left action of $U(\fg(\CO))$ on $U(\fg(r))$ which we denote by $\rhd$, a right action of $U(\fg(r))$ on $U(\fg(\CO))$, which we denote by $\lhd$
both being coalgebra homomorphisms, satisfying:
\be\label{eq:crossedUg}
\begin{aligned}
    &(hg)\lhd a=(h\lhd (g_{(1)}\rhd a_{(1)}))\cdot (g_{(2)}\lhd a_{(2)}), 1\lhd a=\epsilon(a)\\
    & h\rhd (ab)=(h_{(1)}\rhd a_{(1)})\cdot ((h_{(2)}\lhd a_{(2)})\rhd b), h\rhd 1=\epsilon (h)\\
    & h_{(1)}\lhd a_{(1)}\otimes h_{(2)}\rhd a_{(2)}=h_{(2)}\lhd a_{(2)}\otimes h_{(1)}\rhd a_{(1)}.
\end{aligned}
\ee
for \(h,g \in U(\fg(\CO))\) and \(a,b\in U(\fg(r))\). The multiplication on the product can be defined by:
\be\label{eq:bicrossproduct}
(a\otimes h)(b\otimes g)=a(h^{(1)}\rhd b^{(1)})\otimes (h^{(2)}\lhd b^{(2)})g.
\ee
For a detailed description of these formulas, see for example \cite{agore2014classifying}. Let us point out that 
\begin{equation}\label{eq:hopf_as_lie_action}
    h \rhd a = \phi_+(h)a \quad \textnormal{ and } \quad h \lhd a = h \phi_-(a)
\end{equation}
for \(h \in \fg(\CO)\) and \(a \in \fg(r)\).

Let us consider now putting all the algebras over $\C\lbb\hbar\rbb$. There, we are able to use formal expressions of the form \(e^{\hbar x}\) for \(x \in \fg(\CK)[\![\hbar]\!]\) which will turn out to be very useful. Let us collect some result about these. 

\begin{Lem}\label{Lem:expaction}
    The following results are true:

    \begin{enumerate}
        \item The group-like elements of \(U(\fg(\CK))[\![\hbar]\!]\) are precisely of the form $e^x$ for some \(x \in \hbar\fg(\CK)[\![\hbar]\!]\).

        \item For all \(x \in \hbar \fp[\![\hbar]\!]\) exists unique \(x_+ \in \hbar \fg(\CO)[\![\hbar]\!], x_- \in \hbar \fg(r)[\![\hbar]\!]\) such that \(e^x = e^{x_-}e^{x_+}\).

        \item For any \(x,y \in \fg(\CK)\) we have \(e^{\hbar x}e^{\hbar y}=e^{H(\hbar x,\hbar y)}\), where $H(\hbar x,\hbar y)$ is given by the Baker-Campbell-Hausdorff (BCH) formula:
        \be
        H(\hbar x,\hbar y)=\hbar x+\hbar y+\frac{\hbar^2}{2}[x,y]+\frac{\hbar^3}{12}([x,[x,y]]+[y, [y, x]])+\cdots \in \hbar \fp[\![\hbar]\!].
        \ee

        \item For any $x \in \fg(r)$ and $y,z\in \fg(\CO)$, we have:
        \be
        [y,z]\lhd e^{\hbar x}=[y\lhd e^{\hbar x}, z\lhd e^{\hbar x}]+(y\lhd (z\rhd e^{\hbar x}))-(z\lhd (y\rhd e^{\hbar x})). 
        \ee
    \end{enumerate}
\end{Lem}
\begin{proof}
    Let \(a \in U(\fg(\CK))[\![\hbar]\!]\) be group-like, i.e.\ \(\Delta(a) = a \otimes a\). Since the only group-like element of \(U(\fg(\CK))\) is 1, we have \(a - 1 \in \hbar U(\fg(\CK))[\![\hbar]\!]\). Therefore, we can write
    \begin{equation}
      \begin{split}
        \Delta(\textnormal{ln}(a)) &= \sum_{k = 0}^\infty \frac{(-1)^{k+1}\Delta(a-1)^k}{k} = \sum_{k = 0}^\infty \frac{(-1)^{k}((a \otimes a)-1)^k}{k} \\&= \textnormal{ln}(a \otimes a) = \textnormal{ln}((a \otimes 1)(1 \otimes a)) =  
        \textnormal{ln}(a\otimes 1) + \textnormal{ln}(1 \otimes a) \\&= \textnormal{ln}(a) \otimes 1 + 1 \otimes \textnormal{ln}(a).
      \end{split}  
    \end{equation}
    Consequently, \(\textnormal{ln}(a)\) is a primitive element of \(U(\fg(\CK))[\![\hbar]\!] \cong U_{\C[\![\hbar]\!]}(\fg(\CK)[\![\hbar]\!])\), so \(\textnormal{ln}(a) \in \hbar\fg(\CK)[\![\hbar]\!]\) and \(a = e^{\textnormal{ln}(a)}\) is necessarily of the form claimed. On the other hand, \(\Delta (x^n)=\sum_{i+j=n}\binom{n}{i} x^i\otimes x^j\) holds for every \(n \in \N\) and \(x \in \fg(\CK)[\![\hbar]\!]\). Therefore, 
    \be
    \begin{split}
        \Delta(e^{\hbar x}) = \sum_{n = 0}^\infty \frac{\hbar^n}{n!}\Delta(x^n) = \sum_{n = 0}^\infty \sum_{i+j=n}\frac{\hbar^i\hbar^j}{i!j!} x^i\otimes x^j = e^{\hbar x} \otimes e^{\hbar x},
    \end{split}
    \ee
    so \(e^{\hbar x}\) is indeed always group-like, proving 1. 

    For 2., note that \(U(\fg(\CK))[\![\hbar]\!]=U(\fg(r))[\![\hbar]\!]\otimes_{\C[\![\hbar]\!]} U(\fg(\CO))[\![\hbar]\!]\) is an isomorphism of \(\C[\![\hbar]\!]\)-coalgebras, so group-like elements on the left are also group-like elements on the right. Therefore, \(e^{x} = a_- \otimes a_+\) for two group-like elements \(a_+ \in U(\fg(\CO))[\![\hbar]\!]\) and \(a_- \in U(\fg(r))[\![\hbar]\!]\). Repeating the arguments of 1.\ for \(\fg(\CK)\) replaced by \(\fg(\CO)\) and \(\fg(r)\), we can see that \(a_\pm = e^{x_\pm}\) for some \(x_+ \in \hbar\fg(\CO)[\![\hbar]\!]\) and \(x_-\in\hbar\fg(r)[\![\hbar]\!]\) respectively.
    
    The statement 3.\ is the well-known defining relation for the BCH series.

    Let us turn to the proof of 4. Using first equation in \eqref{eq:crossedUg}, we obtain for any $a\in U(\fg(r))[\![\hbar]\!]$:
    \be
        \begin{split}
            [y,z]\lhd a
            &=(y\lhd (z\rhd a))+(y\lhd a_{(1)})(z \lhd a_{(2)}) -(z\lhd (y\rhd a))-(z\lhd a_{(1)})(y \lhd a_{(2)})
            \\&=  \sum_{(a)}[y\lhd a_{(1)}, z\lhd a_{(2)}] + (y\lhd(z\rhd a)) - (z\lhd (y\rhd a)). 
        \end{split}
    \ee
    Putting \(a = e^{\hbar x}\) and using \(\Delta(e^{\hbar x}) = e^{\hbar x} \otimes e^{\hbar x}\), we deduce the desired result.
\end{proof}
Next, we examine the duality of Hopf algebras in the \(\C[\![\hbar]\!]\)-extended case. Therefore, recall that \(\fd(\CK)\) and consequently \(U(\fd(\CK))\) carries the \(\epsilon\)-grading defined by the fact that \(\fg^*\) has degree two. Let us put \(\textnormal{deg}_\epsilon(\hbar) = -2\) in order to extend the \(\epsilon\)-grading to \(U(\fd(\CK))[\![\hbar]\!]\).

\begin{Lem}\label{Lem:group}
The following results hold true:
\begin{enumerate}
    \item One can naturally identify $U(\fg(r))^*\lbb\hbar\rbb$ with $S(\hbar \fg^*(\CO))\lbb\hbar\rbb$. 

    \item For every $f\in S(\hbar \fg^*(\CO))$, $\lag f,e^{\hbar x}\rag=0$ for all $x\in \fg(r)$ implies $f=0$. 

    \item There is a bijection between \(\C[\![\hbar]\!]\)-linear maps \(U(\fg(r))[\![\hbar]\!] \to U(\fg(\CK))[\![\hbar]\!]\) and completed tensors \(( S(\hbar\fg^*(\CO)) \otimes U(\fg(\CK)))[\![\hbar]\!]\). If the map is unital, the associated tensor is invertible and this assignment transforms the point-wise multiplication of maps into the multiplication of tensors.

    \item The element $\CE_r = e^{-\hbar (\epsilon \otimes 1)\tau(r)} \in  S(\hbar\fg^*(\CO)) \otimes U(\fg(r))$ corresponds to the canonical embedding \(U(\fg(r))[\![\hbar]\!] \to U(\fg(\CK))[\![\hbar]\!]\) in the sense of 3.
\end{enumerate}
\end{Lem}

\begin{proof}
    We naturally have:
    \be
        U(\fg(r))^*\lbb\hbar\rbb=S(\fg(r))^*\lbb\hbar\rbb=\overline{S}(\fg(r)^*)\lbb\hbar\rbb=S(\hbar\fg^*(\CO))\lbb\hbar\rbb. 
    \ee
    Here, $\overline{S}$ is the symmetric algebra completed with respect to symmetric degrees. This proves 1.

    For 2.,
    we may assume that $f$ is a homogeneous polynomial of degree $n$. In this case, we have \(
        \lag f, e^{\hbar x}\rag=\hbar^n/n! \lag f, x^n\rag\). Write
    \be
        x^n = \lp\sum_{p = 1}^k c^{a_p,n_p} r_{a_p,n_p}\rp^n=\sum_{p = 1}^k\sum_{k_1+\ldots +k_p=n} (c^{a_1,n_1})^{k_1}\cdots (c^{a_p,n_p})^{k_p}\sigma(r_{a_1,n_1}, \cdots, r_{a_p,n_p}),
    \ee
    where $\sigma(r_{a_1,n_1}, \cdots, r_{a_p,n_p})$ is a total symmetrization of $(r_{a_1,n_1})^{k_1}\cdots (r_{a_p,n_p})^{k_p}$. Since the elements of the form \(\sigma(r_{a_1,n_1}, \cdots, r_{a_p,n_p})\)
    form a basis of the homogeneous polynomials in \(U(\fg(r))\) of degree \(n\), this implies \(
        0 = \lag f, e^{\hbar x}\rag=\hbar^n/n! \lag f, x^n\rag\) implies \(f = 0\).  

    The assignment in 3.\ can be obtained by continuous extension from the map that assigns to an elementary tensor \( f \otimes a \), for \(f \in S(\hbar\fg^*(\CO))[\![\hbar]\!]\) and \(a \in U(\fg(\CK))[\![\hbar]\!]\), the map
    \begin{equation}
        b \mapsto \langle f,b\rangle a.
    \end{equation}
    The map associated to a tensor \(t\) is unital if and only if \(t = 1 + \CO(\hbar)\). This implies in particular that \(t\) is invertible. The fact that multiplications are respected follows
    from 2.\ and the calculation
    \begin{equation}\label{eq:product_tensor_endomorphism}
        \langle f_1f_2, e^{\hbar x} \rangle a_1a_2 = \langle f_1 \otimes f_2, \Delta(e^{\hbar x})\rangle a_1 a_2 = \langle f_1,e^{\hbar x}\rangle \langle f_2,e^{\hbar x} \rangle a_1a_2
    \end{equation}
    for all \(a_1,a_2 \in U(\fg(\CK))[\![\hbar]\!]\) and \(f_1,f_2 \in S(\hbar\fg^*(\CO))[\![\hbar]\!]\).

    For the proof of 4., let us observe that \(\langle -\hbar(\epsilon\otimes 1)\tau(r),e^{\hbar x} \rangle = \hbar x\) holds for all \(x \in \fg(r)\), since the embedding \(\fg(r) \to \fg(\CK)\) is defined by \(-(\epsilon \otimes 1)\tau(r) = \sum_{n = 0}^\infty \sum_{a = 1}^d I^a_n\otimes r_{a,n}\). Similar to \eqref{eq:product_tensor_endomorphism}, we can calculate \(\langle(-\hbar)^n (\epsilon \otimes 1)^n\tau(r)^n,e^{\hbar x} \rangle =\hbar^n x^n\). Therefore, we can calculate:
    \be
        \begin{split}
            \langle \CE_r, e^{\hbar x}\rangle = \sum_{n = 0}^\infty \frac{1}{n!}\langle(-\hbar)^n(\epsilon \otimes 1)^n\tau(r)^n, e^{\hbar x}\rangle = \sum_{n = 0}^\infty \frac{\hbar^n}{n!}x^n = e^{\hbar x}.
        \end{split}
    \ee
    This shows that the map associated to \(\CE_r\) is the embedding \(U(\fg(r))[\![\hbar]\!] \to U(\fg(\CK))[\![\hbar]\!]\) on group-like elements and hence on all elements by 2.
\end{proof}

From this, we can naturally give the algebra $S(\hbar\fg^*(\CO))\lbb\hbar\rbb$ a non-cocommutative Hopf algebra structure, by giving it a new coproduct defined by
\be
\lag \Delta_{\rho,\hbar} (f), a\otimes b\rag=\lag f, ab\rag,
\ee
for $f,g\in S(\hbar\fg^*(\CO))\lbb\hbar\rbb$ and $a,b\in U(\fg(r))\lbb\hbar\rbb$. Since the coproduct \(\Delta_{\rho,\hbar}\) always contains at least one $\hbar$ on $\hbar\fg^*(\CO)$, we can extend it to $S(\fg^*(\CO))[\![\hbar]\!]$ by putting
\be\label{eq:extending_Delta}
    \Delta_{\rho,\hbar}(f) = \hbar^{-1}\Delta_{\rho,\hbar}(\hbar f)
\ee
for \(f \in \fg^*(\CO)\). Using Lemma \ref{Lem:group}.2.,
\be
\lag \Delta_{\rho,\hbar} (I^a_n), e^{\hbar r_{b,i}}\otimes e^{\hbar r_{c,j}}\rag=\lag I^a_n, e^{H(\hbar r_{c,j},\hbar r_{b,i})}\rag, 
\ee
and \([r_{b,i},r_{c,j}] = \sum_{a = 1}^d\sum_{n = 0}^\infty g^{(a,n)}_{(b,i),(c,j)}r_{a,n}\), we have
\be\label{eq:delta_BCH}
\Delta (I^a_n)=I^a_n\otimes 1+1\otimes I^a_n+\frac{\hbar }{2}\sum_{b,c = 1}^d\sum_{i,j = 0}^\infty g^{(a,n)}_{(b,i),(c,j)} I^b_i\otimes I^c_j+\cdots,
\ee
where the right-hand side is dual to the BCH formula.

The left action of $U(\fg(\CO))$ on $U(\fg(r))$ defines a right action of $U(\fg(\CO))[\![\hbar]\!]$ on $S(\fg^*(\CO))\lbb\hbar\rbb$ via
\be\label{eq:defining_hopf_action}
\lag f \circ x, -\rag = \lag  f, x \rhd - \rag\,,\qquad f \in S(\fg^*(\CO))[\![\hbar]\!], x \in U(\fg(\CO))[\![\hbar]\!]. 
\ee
The action restricted to $x \in \fg(\CO)$ is given by algebra derivations, since:
\be
\lag f \cdot g, x\rhd a\rag =\lag f\otimes g,  \Delta (x\rhd a)\rag =\lag f\otimes g, \Delta (x)\rhd \Delta (a)\rag =\lag (f\circ x) \cdot g+f\cdot  (g\circ x), a\rag.
\ee

\begin{Thm}\label{Prop:qHopf}
    The smash-product $\CA_\hbar(\fd,\rho) \coloneqq U(\fg(\CO))[\![\hbar]\!] \#_{\C[\![\hbar]\!]}S(\fg^*(\CO))[\![\hbar]\!]\) does not depend on the choice of \(\rho\). Moreover, it can be made into a Hopf algebra, with coproduct \(\Delta_{\rho,\hbar}\) defined by:
    \be
\Delta_{\rho,\hbar} (x)\coloneqq \CE_r^{-1}(x \otimes 1 + 1 \otimes x)\CE_r,\qquad \lag \Delta_{\rho,\hbar} (f), a\otimes b\rag=\lag f, a\cdot b\rag,
    \ee
    for \(x \in \fg(\CO)\), \(f \in \hbar\fg^*(\CO)\) and \(a,b \in U(\fg(r))[\![\hbar]\!]\).
    The unit is \(1 \in U(\fg(\CO))\) and the counit is defined by $\epsilon(x)=0$ for all \(x \in \fd(\CO)\) respectively.
    %The antipode $S$ is uniquely specified by $\lag S(f), e^{\hbar Y}\rag=\lag f, e^{-\hbar Y}\rag$ for all $f\in \fg^*(\CO)$ and $Y\in \fg(r)$. 
    
    Moreover, the following holds:
    \begin{itemize}
        \item \(\CA_\hbar(\fd,\rho)/\hbar \CA_\hbar(\fd,\rho) \cong U(\fd(\CO))\) as \(\epsilon\)-graded Hopf algebras.
        \item \(\Delta_{\rho,\hbar} - \Delta_{\rho,\hbar}^{\textnormal{op}} = \hbar\delta_\rho + \CO(\hbar^2).\)
    \end{itemize}
    In particular, \(\CA_\hbar(\fd,\rho)\) is an \(\epsilon\)-graded quantization of the Lie bialgebra structure \(\delta_\rho\) on \(\fd(\CO)\).
\end{Thm}

The proof of the theorem will take advantage of the following useful lemma.

\begin{Lem}\label{Lem:twistcoprod}
    The following holds true:
    \begin{enumerate}
        \item The action of $U(\fg(r))$ on $U(\fg(\CO))$ gives rise to a \(\C[\![\hbar]\!]\)-linear map
        \be
            \phi\colon U(\fg(\CO))[\![\hbar]\!]\to (S(\fg^*(\CO)) \otimes U(\fg(\CO)))\lbb\hbar\rbb,
        \ee
        uniquely determined by $\lag \phi(x), e^{\hbar y}\rag=x \lhd e^{\hbar y}$ for \(x \in U(\fg(\CO)),y\in\fg(r)\). Moreover, we have \(\Delta_{\rho,\hbar}(x) = x\otimes 1 + \phi(x)\) for all \(x \in \fg(\CO)[\![\hbar]\!]\) and
        \begin{equation}
            \Delta_{\rho,\hbar}\colon U(\fg(\CO))[\![\hbar]\!] \to (S(\fg^*(\CO)) \otimes U(\fg(\CO)))[\![\hbar]\!]    
        \end{equation}
        is a well-defined algebra homomorphism.

        \item \((\Delta_{\rho,\hbar}\otimes 1)(\CE_r) = \CE_r^{13}\CE_r^{23}\), \((1 \otimes \Delta)(\CE_r) = \CE_r^{12}\CE_r^{13}\) and \(
            (1 \otimes \Delta_{\rho,\hbar})(\CE_r) = \CE_r^{-1,23}(1 \otimes   \Delta)(\CE_r)\CE_r^{23}\).

        \item There is a unique continuous \(\C[\![\hbar]\!]\)-algebra anti-endomorphism \(S\) of \(\CA_\hbar(\fd,\rho)\) such that \(S(f) = -f\) and \(S(x) = -\nabla((S \otimes 1)\phi(x))\) for all \(f \in \fg^*(\CO)\) and \(x \in \fg(\CO)\). Here, \(\nabla \colon \CA_\hbar(\fd,\rho) \otimes_{\C[\![\hbar]\!]} \CA_\hbar(\fd,\rho) \to \CA_\hbar(\fd,\rho)\) is the multiplication map.  
    \end{enumerate}
    
\end{Lem}

\begin{proof}
    Let $x\in \fg(\CO)$. For 1.\ it suffices to show \(\CE_r^{-1}(x\otimes 1+1\otimes x)\CE_r=x \otimes 1+\phi(x)\)
    or equivalently,
    \be
        [x \otimes 1,\CE_r]=\CE_r\phi(x)-(1\otimes x)\CE_r.
    \ee
    As usual, let us evaluate both sides on $e^{\hbar y}$ for some $y\in \fg(r)$. We have:
    \be
        \lag [x \otimes 1,\CE_r], e^{\hbar y}\rag=-\CE_r(x \rhd e^{\hbar y}) = -x \rhd e^{\hbar y}.
    \ee
   since \(e^{\hbar (\epsilon \otimes 1)\tau(r)}\) defines the embedding \(U(\fg(r))[\![\hbar]\!] \to U(\fg(\CK))[\![\hbar]\!]\). On the other hand, we have
    \be
        \lag \CE_r\phi(x), e^{\hbar y}\rag=e^{\hbar y}(x \lhd e^{\hbar y}),
    \ee
    by Lemma \ref{Lem:expaction}.3. Similarly:
    \be
        \lag  (1\otimes x)\CE_r, e^{\hbar y}\rag=xe^{\hbar y}.
    \ee
    We therefore need to show \(
        -x \rhd e^{\hbar y} =e^{\hbar y}(x \lhd e^{\hbar y})-xe^{\hbar y}\),
    which follows from the definition of bi-crossed product in equation \eqref{eq:bicrossproduct}. This completes the proof of 1.

    The first identity in 2.\ follows from
    \begin{equation}
        \begin{split}
            \langle (\Delta_{\rho,\hbar} \otimes 1)(\CE_r), e^{\hbar x} \otimes e^{\hbar y}\rangle &= \langle \CE_r, e^{\hbar x}e^{\hbar y}\rangle \\&= \langle \CE_r^{12}\CE_r^{13}, e^{\hbar x} \otimes e^{\hbar y}\rangle        
        \end{split}
    \end{equation}
    for all \(x,y \in \fg(r)\). The second identity follows similar to Lemma \ref{Lem:expaction}.1.\ and it implies the last identity by definition of \(\Delta_{\rho,\hbar}\) on \(U(\fg(\CO))\).

    In order to prove 3., it suffices to show the following relations:
    \begin{enumerate}
        \item[(i)] For any $x_1, x_2\in \fg(\CO)$, we have $S([x_1,x_2])=[S(x_2), S(x_1)]$.

        \item[(ii)] For any $x\in \fg(\CO), f\in \fg^*(\CO)$ we have $S([x, f])=[S(f), S(x)]$.
    \end{enumerate}
Let us start with proving (i). Write $\phi(x_i)=\sum f_i^{(1)}\otimes x_i^{(2)}$ for \(i \in \{1,2\}\). Then we have:
\be
\begin{split}
    &[S(x_2), S(x_1)]=\sum [S(f_2^{(1)})x_2^{(2)}, S(f_1^{(1)})x_1^{(2)}]
    \\&=\sum \left(S(f_2^{(1)})[x_2^{(2)}, S(f_1^{(1)})] x_1^{(2)}-S(f_1^{(1)})[x_1^{(2)}, S(f_2^{(1)})]x_2^{(2)}+S(f_2^{(1)}) S(f_1^{(1)})[x_2^{(2)},  x_1^{(2)}]\right)
\end{split}
\ee
Evaluating this element of \((S(\fg^*(\CO)) \otimes U(\fg(\CO))[\![\hbar]\!]\) on  $e^{\hbar y}$ for any $y \in \fg(r)$ in the sense of Lemma \ref{Lem:group} yields:
\be\label{eq:rhs_evaluated_antipode}
-x_1\lhd S((x_2 \lhd e^{-\hbar y})\rhd e^{\hbar y})+x_2\lhd S((x_1 \lhd e^{-\hbar y})\rhd e^{\hbar y})+[x_2\lhd e^{-\hbar y}, x_1\lhd e^{-\hbar y}].
\ee
According to Lemma \ref{Lem:expaction}.4.\ and Lemma \ref{Lem:group}.2., it suffices to show that 
\be\label{eq:remaining_identity_antipode}
    S((x \lhd e^{-\hbar y})\rhd e^{\hbar y})=x\rhd e^{-\hbar y}
\ee
holds, since \eqref{eq:rhs_evaluated_antipode} then coincides with the evaluation of \(S([x_1,x_2])\) in \(e^{\hbar y}\). 

Using equation \eqref{eq:crossedUg}, we have:
\be
(x\lhd e^{-\hbar y}) e^{\hbar y}=(x\lhd e^{-\hbar y})\rhd e^{\hbar y}+e^{\hbar y}x.
\ee
Applying the map \(S\) and using
\be
S(e^{\hbar y}x)=-xe^{-\hbar y}=-(x\rhd e^{-\hbar y})-e^{-\hbar y}(x\lhd e^{-\hbar y}).
\ee
we deduce that
\be
-(x\rhd e^{-\hbar y})-e^{-\hbar y}(x\lhd e^{-\hbar Y})= e^{-\hbar y}S(x\lhd e^{-\hbar y})-S\lp (x\lhd e^{-\hbar y})\rhd e^{\hbar y}\rp
\ee
holds. Since both $e^{-\hbar y}(x\lhd e^{-\hbar y})$ and $e^{-\hbar y}S(x\lhd e^{-\hbar y})$ belong to $(U(\fg(r))\otimes \fg(\CO))[\![\hbar]\!]$, the above equalities imply the desired
\be
x\rhd e^{-\hbar y}=S\lp (x\lhd e^{-\hbar y})\rhd e^{\hbar y}\rp.
\ee

To prove the identity (ii), we again write $\phi(x)=\sum f^{(1)}\otimes x^{(2)}$. For any $f\in \fg^*(\CO)$, we have:
\be
[Sf, Sx]=-\sum [S(f), S(f^{(1)})x^{(2)}]=-\sum S(f^{(1)})[S(f), x^{(2)}].
\ee
Evaluating again on an element $e^{\hbar y}$ for $y\in \fg(r)$, the last expression becomes:
\be
-\lag S(f),  (x\lhd e^{-\hbar y})\rhd e^{\hbar y}\rag=-\lag f, S\lp (x\lhd e^{-\hbar y})\rhd e^{\hbar y}\rp\rag.
\ee
Now we use again that $S\lp (x\lhd e^{-\hbar y})\rhd e^{\hbar y}\rp=x\rhd e^{-\hbar y}$ to deduce
\be
-\lag f, x\rhd e^{-\hbar Y}\rag=\lag S[x, f], e^{\hbar Y}\rag.
\ee
which completes the proof. 

\end{proof}

\begin{proof}[Proof of Theorem \ref{Prop:qHopf}]
    Everything is \(\epsilon\)-graded and continuous by construction.
    Furthermore, the action \(\circ\) coincides with the usual action of \(U(\fg(\CO))\) on \(S(\fg^*(\CO))\) modulo \(\hbar\) since,
    \begin{equation}
        \langle f\circ x, e^{\hbar y} \rangle = f(\phi_+(x)y) + \CO(\hbar)  = \langle f\textnormal{ad}^*(x),y\rangle + \CO(\hbar)
    \end{equation}
    holds for all \(x \in \fg(\CO), f\in \fg^*(\CO)\). Therefore, we can see that
    \be
    \CA_\hbar(\fd,\rho)/\hbar \CA_\hbar(\fd,\rho) \cong U(\fg(\CO)) \# S(\fg^*(\CO)) \cong U(\fd(\CO))
    \ee
    holds as \(\epsilon\)-graded topological Hopf algebras. Moreover, the definition of the action \(\circ\) does not depend on the choice of complementary subalgebra \(\fg(r)\), so, as an \(\C[\![\hbar]\!]\)-algebra, \(\CA_\hbar(\fd,\rho)\) is independent of \(\rho\). 
    
    Since \(\delta_\rho\) is, by definition, dual to the bracket of \(\fg(r) \oplus \epsilon\fg(r)\), \eqref{eq:delta_BCH} implies
    \be
        \Delta(f) - \Delta^{\textnormal{op}}(f) = \hbar \delta_\rho(f) + \CO(\hbar^2)
    \ee
    holds for all \(f \in \fg^*(\CO)\).
    The identities \(\CE_r^{\pm 1} = e^{\pm \hbar (\epsilon \otimes 1)\tau(r)}\) and \(\CE_r^{\pm 1,21} = e^{\pm\hbar(\epsilon \otimes 1)r}\) imply:
    \be
    \begin{split}
        \Delta_{\rho,\hbar}(x)-\Delta_{\rho,\hbar}^{\textnormal{op}} (x) &= \hbar (-[(\epsilon\otimes 1)\tau(r),x \otimes 1 + 1 \otimes x] + [(1 \otimes \epsilon)r,x \otimes 1 + 1 \otimes x]) + \CO(\hbar^2) \\&= \hbar \delta_\rho(x) + \CO(\hbar^2).    
    \end{split}
    \ee
    These are the quantization equations.

    We now turn to prove that $\Delta_{\rho,\hbar}$ defines a bialgebra structure on \(\CA_\hbar(\fd,\rho)\). First of all, \(\Delta_{\rho,\hbar}\) is well-defined by Lemma \ref{Lem:twistcoprod}.1.\ and \eqref{eq:extending_Delta}. Since everything is continuous and \(\epsilon\)-graded, this means that we we need to show that $\Delta_{\rho,\hbar}$ is a coassociative algebra homomorphism. The fact that \(\Delta_{\rho,\hbar}(ab) = \Delta_{\rho,\hbar}(a)\Delta_{\rho,\hbar}(b)\) holds for \(a,b \in U(\fg(\CO))[\![\hbar]\!]\) or \(a,b \in S(\fg^*(\CO))[\![\hbar]\!]\) is clear by definition. Since \(\CA_\hbar(\fd,\rho)\) is defined as a smash product with \(U(\fg(\CO))[\![\hbar]\!]\) over \(\C[\![\hbar]\!]\)-algebra, it remains to show that\be
        [\Delta_{\rho,\hbar}(x), \Delta_{\rho,\hbar} (t)]=\Delta_{\rho,\hbar} ([x,t]) = -\Delta_{\rho,\hbar} (t \circ x)
    \ee
holds for all \(x\in\fg(\CO), t\in \fg^*(\CO)\). We evaluate on elements $e^{\hbar y}, e^{\hbar z}$ for \(y,z \in \fg(r)\):
\be
\langle \Delta_{\rho,\hbar}([x,t]), e^{\hbar y}\otimes e^{\hbar z}\rangle=\lag [x,t], e^{\hbar y}e^{\hbar z}\rag=\lag t, x \rhd (e^{\hbar y}e^{\hbar z})\rag.
\ee
On the other hand, Lemma \ref{Lem:twistcoprod}.1. and \eqref{eq:crossedUg} implies:
\be
\begin{split}
    \lag [\Delta_{\rho,\hbar}(x), \Delta_{\rho,\hbar} (t)], e^{\hbar y} \otimes e^{\hbar z}\rag& = 
    \lag t, (x \rhd e^{\hbar y})e^{\hbar z}\rag+\lag [\phi(x), \Delta_{\rho,\hbar}(t)], e^{\hbar y}\otimes e^{\hbar z}\rag.  
    \\&=\lag t, (x \rhd e^{\hbar y})e^{\hbar z}\rag + \sum_i \lag f_i, e^{\hbar y}\rag\cdot  \lag t, e^{\hbar y}(x_i\rhd e^{\hbar z})\rag\\&=
    \lag t, (x\lhd e^{\hbar y})e^{\hbar z}\rag+ \lag t, e^{\hbar y}((x \lhd e^{\hbar y}) \rhd e^{\hbar z})\rag \\&= \lag t, x \rhd (e^{\hbar y}e^{\hbar z})\rag,
\end{split}
\ee
where we wrote $\phi(x)=\sum f_i \otimes x_i$ and used $\langle \phi(x), e^{\hbar y}\rangle=\sum_i \langle f_i,e^{\hbar y}\rangle x_i = x \lhd e^{\hbar y}$ as well as $\Delta (e^{\hbar y})=e^{\hbar y}\otimes e^{\hbar y}$. Therefore, $[\Delta_{\rho,\hbar}(x), \Delta_{\rho,\hbar} (t)]=\Delta_{\rho,\hbar}([x,t])$ holds. 

To show coassociativity of \(\Delta_{\hbar,\rho}\), note that for elements of \(S(\fg^*(\CO))[\![\hbar]\!]\) it follows from associativity of $U(\fg(r))[\![\hbar]\!]$. We need to show $\Delta_{\rho,\hbar}$ is coassociative for an elements \(x\) of \(U(\fg(\CO))[\![\hbar]\!]\). Using Lemma \ref{Lem:twistcoprod}.3.\ this boils down to
\be
    \begin{split}
        (\Delta_{\rho,\hbar} \otimes 1)\Delta_{\rho,\hbar}(x) &= (\CE_r^{12}\CE_r^{13}\CE_r^{23})^{-1}(\Delta \otimes 1)\Delta(x)\CE_r^{12}\CE_r^{13}\CE_r^{23}\\& = (1 \otimes \Delta_{\rho,\hbar})\Delta_{\rho,\hbar}(x).
    \end{split}
\ee

Finally, we show that $S$ from Lemma \ref{Lem:twistcoprod} is the an antipode of \(\CA_\hbar(\fd,\rho)\), i.e.\ that it satisfies the following two equations:
\be
\nabla (S\otimes 1)\Delta_{\rho,\hbar} =\nabla (1\otimes S)\Delta_{\rho,\hbar}=\epsilon.
\ee
Since the algebra is generated by $\fg(\CO)$ and $\fg^*(\CO)$, both of which are zero under $\epsilon$, it suffices to check that the above are equal to zero on these generators. Let $x\in \fg(\CO), f\in \fg^*(\CO)$. The equation $\nabla (S\otimes 1)\Delta_{\rho,\hbar}(x)=0$ is by definition. For $f$, we again evaluate on $e^{\hbar y}$ for \(y \in \fg(r)\):
\be
\lag \nabla (S\otimes 1)\Delta_{\rho,\hbar}(f), e^{\hbar y}\rag=\lag (S\otimes 1)\Delta_{\rho,\hbar}(f), e^{\hbar y}\otimes e^{\hbar y}\rag=\lag f, e^{-\hbar y}e^{\hbar y}\rag=\lag f, 1\rag=0.
\ee
The equation $\nabla (1\otimes S)\Delta_{\rho,\hbar}(f)=0$ is proven similarly, we are left to show that 
\be
    \nabla (1\otimes S)\Delta_{\rho,\hbar}(x)=0.
\ee
Writing $\phi(x)=x\otimes 1+\sum f^{(1)}\otimes x^{(2)}$ and evaluating
\be
\nabla (1\otimes S)\Delta_{\rho,\hbar}(x)=x+\sum f^{(1)} S(x^{(2)})\in (S(\fg^*(\CO))\otimes U(\fg(\CO)))\lbb\hbar\rbb
\ee
on $e^{\hbar y}$ for any \(y \in \fg(r)\) gives \(
x- (x\lhd e^{-\hbar Y})\lhd e^{\hbar Y}=x-x=0\).
This completes the proof. 
\end{proof}

\subsubsection{Twisting and splitting independence of the monoidal category}\label{sec:twisting}

Since the construction of the algebra $\CA_\hbar(\fd,\rho)$ mimics the monoidal structure of the double quotient $\wh{\fg(\CO)}\,\!\!\setminus \wh{\fg(\CK)}/\wh{\fg(\CO)}$, for different choice of \(\epsilon\)-graded \(r\)-matrices \(\rho\), the resulting monoidal category should be equivalent. We show in this section that there is an twisting between the Hopf algebras for different choices of the splitting. This twisting then induces the equivalences on the level of monoidal categories.  

Let \(r_1,r_2\) be two generalized \(r\)-matrices with coefficients in \(\fg\) and \(\rho_1,\rho_2\) be the associated \(\epsilon\)-graded \(r\)-matrices with coefficients in \(\fd\). Clearly, the difference \(\delta_{\rho_1} - \delta_{\rho_2}\) of the associated Lie bialgebra structures is a 1-coboundary \(a \mapsto [t, x \otimes 1 + 1 \otimes x]\) for the holomorphic tensor 
\begin{equation}
    t \coloneqq (1 \otimes \epsilon)s - (\epsilon \otimes 1)\tau(s) = \rho_2 - \rho_1 \in \fd(\CO) \otimes \fd(\CO)    
\end{equation}
for \(s \coloneqq r_2 - r_1 \in \fg(\CO) \otimes \fg(\CO)\). In the language of Lie bialgebras, \(t\) is a (topological) classical twist that transforms \(\delta_{\rho_1}\) in \(\delta_{\rho_2}\). 
The goal of this section is to see that \(t\) can be quantized to a quantum twist between the two comultiplications \(\Delta_{\rho_1,\hbar}\) and \(\Delta_{\rho_2,\hbar}\) of the Hopf algebras \(\CA_\hbar(\fd,\rho_1)\) and \(\CA_\hbar(\fd,\rho_2)\), where we recall that these are canonically isomorphic as \(\C[\![\hbar]\!]\)-algebras. For convenience, we will denote this algebra as \(\CA_\hbar\) in the following.

\begin{Prop}\label{Prop:twistingF}
    The tensor series \(F \coloneqq \CE_{r_2}^{-1}\CE_{r_1}\) is a quantum twist of \(\Delta_{\rho_1,\hbar}\) to \(\Delta_{\rho_2,\hbar}\) that quantizes \(t = \rho_2 - \rho_1\). More precisely, \(F \in (S(\fg^*(\CO))\otimes U(\fg(\CO)))[\![\hbar]\!]\) and: 
    \begin{enumerate}
        \item 
        \(\Delta_{\rho_2,\hbar}=F\Delta_{\rho_1,\hbar}F^{-1}\);
        \item \((\Delta_{\rho_2,\hbar}\otimes 1)(F)F^{12}=(1\otimes \Delta_{\rho_2,\hbar})(F) F^{23}\);
        \item \(F - F^{21} = \hbar t + \CO(\hbar^2).\)
    \end{enumerate}
    If both \(r_1\) and \(r_2\) are rational, \(F \in (S(\hbar\fg^*[t])\otimes U(\fg[t]))[\![\hbar]\!]\) holds.
\end{Prop}

\begin{proof}
    Let us first observe that for every \(x \in \fg(r_1)\) there exist unique \(x_+ \in \fg(\CO)[\![\hbar]\!]\) and \(x_- \in \fg(r_2)[\![\hbar]\!]\) such that 
    \begin{equation}\label{eq:factorization_r12}
        e^{\hbar x} = e^{\hbar x_-}e^{\hbar x_+} 
    \end{equation}
    by virtue of Lemma \ref{Lem:expaction}.2. Moreover, the map \(U(\fg(r_1))[\![\hbar]\!] \to U(\fg(\CK))[\![\hbar]\!]\) associated to \(F\) is uniquely determined by \(e^{\hbar x} \mapsto e^{\hbar x_+}\), since \eqref{eq:factorization_r12} implies \(\CE_{r_1} = \CE_{r_2}F\). Here, we used Lemma \ref{Lem:group}.3 and evaluated at \(e^{\hbar x}\) while extending \(\CE_1\) and \(\CE_2\) to \(U(\fg(\CK))[\![\hbar]\!]\) using the counit of \(U(\fg(\CO))[\![\hbar]\!]\). In particular, observe that \(F\) takes values in \(U(\fg(\CO))[\![\hbar]\!]\) and therefore \(F \in (S(\fg^*(\CO))\otimes U(\fg(\CO)))[\![\hbar]\!]\).

    The statement 1.\ for $a\in U(\fg(\CO))$ follows by definition of \(\Delta_{\rho_1,\hbar}\) and \(\Delta_{\rho_2,\hbar}\) in Theorem \ref{Prop:qHopf}. Consider $f\in S(\fg^*(\CO))\lbb\hbar\rbb$ and any $x,y \in \fg(r_1)$. We have:
    \be
    e^{\hbar x}e^{\hbar y}=e^{\hbar x_-}e^{\hbar x_+}e^{\hbar y_-}e^{\hbar y_+}=e^{\hbar x_-}(e^{\hbar x_+}\rhd^{(2)} e^{\hbar y_-}) (e^{\hbar x_+}\lhd^{(2)} e^{\hbar y_-})e^{\hbar y_+},
    \ee
    where \(\rhd^{(i)},\lhd^{(i)}\) for \(i\in\{1,2\}\) are the left and right actions induced by the decompositions \(U(\fg(\CK))= U(\fg(r_i)) \otimes U(\fg(\CO))\).
    By the definition of $f$, we have:
    \be
    \begin{split}
        \lag \Delta_{\rho_1,\hbar}(f),e^{\hbar x} \otimes e^{\hbar y}\rag &= \lag f,e^{\hbar x}e^{\hbar y}\rag=\lag f, e^{\hbar x_-}(e^{\hbar x_+}\rhd^{(2)} e^{\hbar y_-})\rag
        \\&=\sum_{n = 0}^\infty\frac{1}{n!}\lag f,e^{\hbar x_-}(x_+^n\rhd^{(2)} e^{\hbar y_-})\rag
        \\&= \sum_{n = 0}^\infty\frac{1}{n!}\lag (1 \otimes \textnormal{ad}(x_+)^n)\Delta_{\rho_2,\hbar}(f), e^{\hbar x_-} \otimes e^{\hbar y_-})\rag
        \\& =\lag(1 \otimes e^{\hbar x_+})\Delta_{\rho_2,\hbar}(f)(1 \otimes e^{\hbar x_+}),e^{\hbar x_-} \otimes e^{\hbar y_-}\rag
        \\&=
        \lag(1 \otimes e^{\hbar x_+})\Delta_{\rho_2,\hbar}(f)(1 \otimes e^{\hbar x_+}),e^{\hbar x} \otimes e^{\hbar y}\rag
        \\&=
        \lag F^{-1}\Delta_{\rho_2,\hbar}(f)F,e^{\hbar x} \otimes e^{\hbar y}\rag.
    \end{split}
    \ee
    Here, we used Lemma \ref{Lem:expaction}.2.\ and the fact that \(g(e^{\hbar x}) = g(e^{\hbar x_-})\) and \(g(e^{\hbar y}) = g(e^{\hbar y_-})\) for all \(g \in S(\fg^*(\CO))[\![\hbar]\!]\). This completes the proof of 1.

   Using Lemma \ref{Lem:twistcoprod}.3.\ we see that
   \be  
    \begin{split}
        (1 \otimes \Delta_{\rho_2,\hbar})(F)F^{23} &= \CE_{r_2}^{-1,23}\CE_{r_2}^{-1,13}\CE_{r_2}^{-1,12}\CE^{23}F^{23}(1\otimes \Delta_{\rho_1,\hbar})(\CE_{r_1})
        \\& =\CE_{r_2}^{-1,23}\CE_{r_2}^{-1,13}\CE_{r_2}^{-1,12}\CE^{12}_{r_1}\CE^{13}_{r_1}\CE^{23}_{r_1}
    \end{split}
   \ee
   Similarly,
   \be
        \begin{split}
            &(\Delta_{\rho_2,\hbar} \otimes 1)(F)F_{12} =\CE_{r_2}^{-1,23}\CE_{r_2}^{-1,13}\CE_{r_2}^{-1,12}\CE^{12}_{r_1}\CE^{13}_{r_1}\CE^{23}_{r_1}
        \end{split}
   \ee
   proving 2. Finally,
   \(
        F - F^{21}  = \hbar (\rho_2-\rho_1) + \CO(\hbar)
   \)
   concludes the proof of 1.-3.

   In order to prove that last claim, its sufficient to assume that \(r_1 \coloneqq r\) is rational and \(r_2 = \gamma_\fg\) is Yang's \(r\)-matrix for \(\fg\). Since \(r\) is rational, we have \(t^{-N}\fg[t^{-1}] \subseteq \fg(r) \subseteq t^{N}\fg[t^{-1}]\). This immediately implies that \(F \in (S(\hbar \fg^*(\CO)) \otimes U(\fg[t]))[\![\hbar]\!]\). Moreover, this implies that \(F\) is simply constantly \(1\) on \(U(t^{-N}\fg[t^{-1}])[\![\hbar]\!]\), since \(\CE_\gamma\) and \(\CE_r\) coincide on this space. Therefore, \(F \in (S(\hbar \fg^*[t]) \otimes U(\fg[t]))[\![\hbar]\!]\) due to the PBW theorem and the fact that \(\CE_\gamma\) and \(\CE_r\) are algebra homomorphisms.
   \end{proof}

The existence of the twisting means that $F$ induces an isomorphism of modules
\be
M\otimes_{\Delta_{\rho_1,\hbar}}N\cong M\otimes_{\Delta_{\rho_2,\hbar}} N.
\ee
and Proposition \ref{Prop:twistingF}.2. implies that the ``monoidal structure axiom" 
\be
\btik
(M\otimes_{\Delta_{\rho_1,\hbar}} N)\otimes_{\Delta_{\rho_1,\hbar}} P\rar{=}\dar  & M\otimes_{\Delta_{\rho_1,\hbar}} (N\otimes_{\Delta_{\rho_1,\hbar}} P)\dar\\
(M_s\otimes_{\Delta_{\rho_2,\hbar}} N)\otimes_{\Delta_{\rho_1,\hbar}} P\dar & M\otimes_{\Delta_{\rho_1,\hbar}} (N\otimes_{\Delta_{\rho_2,\hbar}} P)\dar\\
(M_s\otimes_{\Delta_{\rho_2,\hbar}} N)\otimes_{\Delta_{\rho_2,\hbar}} P\rar{=} & M\otimes_{\Delta_{\rho_2,\hbar}}(N\otimes_{\Delta_{\rho_2,\hbar}} P)
\etik
\ee
from \cite[Definition 2.4.1]{etingof2016tensor} holds. Therefore, the representation categories of \(\CA_\hbar(\fd,\rho_1)\) and of \(\CA_\hbar(\fd,\rho_1)\) are equivalent as monoidal categories.

\subsubsection{Evaluation of $\hbar$}\label{sec:evaluation}

We very quickly comment on how the above constructions all have well-defined evaluation at $\hbar=\xi$ for any $\xi$. By rescaling $\fg^*\lbb t\rbb$ using $\epsilon$-grading, we just need to show that this is true for $\hbar=1$. 

The crucial part of the proof is the identification of the dual of $U(\fg(r))\lbb\hbar\rbb$ with the symmetric algebra $S(\hbar\fg^*\lbb t\rbb)\lbb\hbar\rbb$. We note that one can have a similar identification without $\hbar$, simply by noticing that:
\be
U(\fg(r))^*=\varprojlim_i\overline{S}(\fg^*[t]/t^i),
\ee
where $\overline{S}(\fg^*[t]/t^i\fg^*[t])$ is the completed symmetric algebra of the finite-dimensional vector space $\fg^*[t]/t^i\fg^*[t]$. Therefore the evaluation $\CA_1(\fd, \rho)$ is simply the smash product $U(\fg(\CO))\# U(\fg(r))^*$, which is a topological algebra. The proofs of the Hopf structure goes through for this algebra without problem, except that one has to use elements in $U(\fg(r))$ when pairing, since $e^{x}$ does not make sense anymore. However, for each $k\geq 0$, the element $x^k$ still makes sense and that is all one needs to repeat the proofs above. 

By definition, $\CA_1(\fd, \rho)\Mod$ is the category of smooth modules of this algebra. For such a module, the action of $\fg^*\lbb t\rbb$ will be nilpotent, and eventually zero on $t^N \fg^*\lbb t\rbb$. In the previous sections we choose to construct this quantization as a flat deformation over $\C\lbb\hbar\rbb$. The advantages of doing this are first of all it allows the use of $e^{\hbar x}$ which makes all the proofs much more simplified, and secondly it preserves the $\epsilon$-gradedness of the quantization, which is important for uniqueness. 

\subsubsection{Generalization of the quantization scheme}\label{subsubsec:topology}

The Lie bialgebra structures on $\fd(\CO)$ and their quantization scheme can be generalized to the setting from Section \ref{sec:generalization_PoissonLie}. Namely, we could have considered any topological Lie algebra $\fp$ with a given Lie algebra decomposition $\fp=\fp_+\oplus \fp_-$ and consider the associated Manin triple 
\begin{equation}
    \mathfrak{e} \coloneqq \fp \ltimes \fp^* = (\fp_- \ltimes \fp_+^*) \oplus (\fp_+ \ltimes \fp_-^*).    
\end{equation}
If 
the topology can be chosen appropriately, the Lie bracket of \(\fp_+ \ltimes \fp_-^*\) induces a Lie bialgebra structures on \(\fh \coloneqq \fp_- \ltimes \fp_+^*\) by duality. 

Adapting then the arguments from this section so far to \(\fp\) leads to a quantization of this Lie bialgebra structure. This quantization again resembles the monoidal structure on the double quotient \(\widehat\fh \setminus \widehat{\mathfrak{e}} \,/\,\widehat\fh\). Let us consider a finite-dimensional example, in which case the topology is chosen to be discrete.

\begin{Exp}
    Let us consider $\fp=\mathfrak{sl}(2,\C)$ and $\fp_-=\mathfrak{b}$, the Borel subalgebra of upper-triangular matrices. In this case $\fp_+$ is the algebra of lower triangular matrices. We use the standard generators $\{H,E,F\}$ of $\mathfrak{sl}(2,\C)$. The quantization scheme developed in the section so far produces a Hopf algebra \(\CA_\hbar\) that quantizes the Lie bialgebra structure on \(\mathfrak{b} \ltimes \C F^\vee\). The algebra $\CA_\hbar$ is generated by $H, E$ and $F^\vee$, such that:
    \be
[H,E]=2E,\qquad [H,F^\vee]=2F^\vee,\qquad [E, F^\vee]=\hbar (F^\vee)^2. 
    \ee
Note that the third commutation relation is from the following:
\be
[E, F^\vee](e^{\hbar a F})=F^\vee (E\rhd e^{\hbar aF})=F^\vee (a^2\hbar^2 H\rhd F)=a^2\hbar. 
\ee
The coproduct is given by:
\be
\Delta_\hbar (H)=H\otimes 1+1\otimes H,\qquad \Delta_\hbar (E)=E\otimes 1+1\otimes E+\hbar F^\vee\otimes H,\qquad \Delta_\hbar (F^\vee)=F^\vee\otimes 1+1\otimes F^\vee. 
\ee
One can check that this indeed gives a Hopf algebra structure to $\CA_\hbar$. 

\end{Exp}

\subsection{Specializations for loop Lie algebras}\label{subsec:loopapply}

Let us now turn to properties of the quantization \(\CA_\hbar(\fd,\rho)\) of the Lie bialgebra \((\fd(\CO),\delta_\rho)\) which do not admit generalizations from \(\fg(\CK)\) to other topological Lie algebras \(\fp\) in the way outlined in Section \ref{subsubsec:topology}.

First, note that as a topological \(\C[\![\hbar]\!]\)-algebra $\CA_\hbar(\fd, \rho)$ is topologically generated by $\fd[t]$. We will denote by $\CA^\circ_\hbar(\fd, \rho)$ the dense subalgebra generated by $\fd[t]$ over \(\C[\![\hbar]\!]\).

\begin{Prop}\label{Cor:finiteh}
    If \(\rho\) is rational, $\CA^\circ_\hbar(\fd, \rho)$ is a Hopf subalgebra of \(\CA_\hbar (\fd,\rho)\) that quantizes the Lie bialgebra \((\fd[t],\delta_\rho)\).
\end{Prop}

\begin{proof}
    We first show that the algebra structure is well-defined on generators. The only non-trivial part is the commutator between $I_{a,n}$ and $I^b_m$. Let $y\in \fg(\gamma_\fg) = t^{-1}\fg[t^{-1}]$ have highest degree term \(\ell < 0\), then:
    \be
\lag [I_{a,n}, I^b_{m}], e^{\hbar y}\rag =-\lag I^b_m, I_{a,n}\rhd e^{\hbar y}\rag,
    \ee
    where the $\hbar^k$ coefficient is given by $
-\frac{1}{n!}\lag I^b_m, I_{a,n}\rhd y^k\rag. $
% This is a complicated pairing, but by loop grading, the only part that can contribute to this comes from:
% \be
% \sum_{\sum (m_i+1)=m+1-n} \lag I^b_m,  I_{a, n}\rhd \prod_i c_{m_i}I_{-m_i-1}\rag.  
% \ee
When $k$ satisfies \(m +  n +1 < -k\ell\), this coefficient vanishes. In particular, the expression $[I_{a,n}, I^b_{m}]$ is a finite polynomial in $\hbar$ (as well as $I^c_k$) and hence can be evaluated at any $\xi\in \C$. 

We now show that the Hopf algebra structure is well-defined. If this is the case, its clear that \(\CA_\hbar^\circ(\fd,\rho)\) will be a quantization of \((\fd[t],\rho)\). By virtue of Proposition \ref{Prop:twistingF}, we may assume that \(\rho = \gamma\). It is clear on counit and antipode, and we only need to prove it for coproduct. In this case, we again show that $\Delta_\hbar (I^a_n)$ and $\Delta_\hbar (I_{a,n})$ are finite polynomials in $\hbar$. Consider first $\Delta_\hbar (I^a_n)$, by BCH formula:
\be
\lag \Delta_\hbar (I^a_n), e^{\hbar x}\otimes e^{\hbar y}\rag=\lag I^a_n, H(\hbar x, \hbar y)\rag. 
\ee
The degree $\hbar^k$ co-efficient is equal to:
\be
\sum \lag I^a_n, f_k (\hbar x,\hbar y)\rag
\ee
where this $f_k (\hbar x,\hbar y)$ is a sum of $k$-th iterated Lie-bracket of $\hbar x$ with $\hbar y$. Since the loop grading of components of $\hbar x$ and $\hbar y$ are at least $-1$, the $k$-th iterated bracket has degree at least $-k$, and therefore for $k$ large enough the above is zero. Consequently $\Delta_\hbar (I^a_n)$ is a finite polynomial in $\hbar$ as well as $I^c_k$. An identical argument applies to $\Delta_\hbar (I_{a,n})$. This completes the proof. 
\end{proof}

We now turn to the case when $\rho$ depends only on $t_1-t_2$. In this case, the Lie algebra splitting $\fg(\CK)=\fg(\CO)\oplus \fg(r)$
is compatible with the differential $T = \partial_t$, i.e.\ $T(\fg(r))\subseteq \fg(r)$ holds. By dualizing, this $T$ defines a derivation on $U(\fg(r))^*[\![\hbar]\!]=S(\hbar\fg^*(\CO))[\![\hbar]\!]$ and extended to \(S(\fg^*(\CO))[\![\hbar]\!]\). Similarly we have a derivation on $U(\fg(\CO))[\![\hbar]\!]$. Then the \(\C[\![\hbar]\!]\)-linear extension of \(T\) defines an endomorphism of \(\CA_\hbar(\fd,\rho)\).

\begin{Lem}\label{Lem:nilT}
    This $T$ defines a Hopf algebra derivation of $\CA_\hbar(\fd,\rho)$ and acts nilpotently on generators $\fd[t]$. 
\end{Lem}

\begin{proof}
    To show that $T$ defines an algebra derivation, we just need to show that it respects the commutator of $U(\fg(\CO))[\![\hbar]\!]$ and $S(\fg^*(\CO))[\![\hbar]\!]$. For any $x\in \fg(\CO)[\![\hbar]\!]$ and $y\in U(\fg(r))[\![\hbar]\!]$ and $f\in S(\fg^*(\CO))[\![\hbar]\!]$, we have:
    \be
\lag T[x,f], y\rag =-\lag [x,f], Ty\rag =\lag f, x\rhd Ty\rag.
    \ee
    Now since $T$ is a Hopf algebra derivation for $U(\fg(\CK))$, we have:
    \be
x\rhd Ty=T(x\rhd y)-(Tx)\rhd y,
    \ee
    which implies:
    \be
\lag f, x\rhd Ty\rag=\lag f, T(x\rhd y)\rag-\lag f, (Tx)\rhd y\rag.
    \ee
The second term is identified with $\lag [Tx, f], y\rag$ and the first with $\lag [x, Tf], y\rag$.

To show that $\Delta_{\rho,\hbar} T=(T\otimes 1+1\otimes T)\Delta_{\rho,\hbar}$, we first comment that it is clear on $S(\fg^*(\CO))[\![\hbar]\!]$ since it is the dual of an algebra derivation. For $x\in \fg(\CO)$, we need to show:
\be
\Delta_{\rho,\hbar}(Tx)=(Tx)\otimes 1+\phi(Tx)=(Tx)\otimes 1+ (T\otimes 1+1\otimes T)\phi(x).
\ee
Equivalently, we need to show $\phi(Tx)=(T\otimes 1+1\otimes T)\phi(x)$. Pairing both sides with an element of the form $e^{\hbar y}$ for $y\in \fg(r)$, using the fact that:
\be
T(x\lhd e^{\hbar y})=(Tx)\lhd e^{\hbar y}+x\lhd (T e^{\hbar y})
\ee
one can show that the two sides are equal. 

Finally, we show that this acts nilpotently on generators $\fd[t]$. It is clear for $\fg[t]$, and we show it for $\fg^*[t]$. Note that for any generator $y\in \fg(r)$ and any $f\in \fg^*[t]$, decompose $y=y_{\textnormal{sing}}+y_{\textnormal{reg}}$ where $y_{\textnormal{sing}}\in t^{-1}\fg[t^{-1}]$ and $y_{\textnormal{reg}}\in \fg(\CO)$, we find:
\be
\lag f, T^n y\rag=\lag f, T^n y_{\textnormal{sing}}\rag=0 \text{ for large enough } n. 
\ee
This shows that $T$ acts nilpotently on $\fg^*[t]$. 

\end{proof}

This lemma implies that we have a well-defined continuous Hopf algebra map:
\be
\tau_z:=e^{zT}: \CA_\hbar(\fd, \rho)\to \CA_\hbar(\fd,\rho)\lbb z\rbb,
\ee
where $z$ is a formal variable, and the topology on $\CA_\hbar(\fd,\rho)\lbb z\rbb$ is induced by the topologies of $\CA_\hbar(\fd,\rho)$ and $\C\lbb z\rbb$. Since $T$ is nilpotent on generators $\fd[z]$, this map defines an algebra morphism:
\be
\tau_z: \CA_\hbar^\circ(\fd, \rho)\to  \CA^\circ_\hbar (\fd,\rho)[z],
\ee
which can be evaluated at any $z=s\in \C$. Note that the same is \textbf{not true} for $\CA_\hbar(\fd,\rho)$, since in general one can not sum over infinitely many generators. 

Using $\tau_z$, we obtain another meromorphic (in fact holomorphic) coproduct:
\be\label{eq:holomorphic_coproduct}
\Delta_{\rho,\hbar, z}:= (\tau_z\otimes 1)\Delta_{\rho,\hbar}.
\ee
Clearly this is a quantization of $(\fd(\CO), \delta_{\rho,z})$ from Section \ref{sec:quasi_classical_vertex}. 

The coproduct \(\Delta_{\rho,\hbar,z}\) on $\CA_\hbar(\fd, \rho)$ induces a product on its \(\C[\![\hbar]\!]\)-linear dual, which, as a \(\C[\![\hbar]\!]\)-module, is isomorphic to \(S(\fd(\rho))[\![\hbar]\!]\):
\be\label{eq:quantum_vertex_opertation}
\CY_{\rho,\hbar}\colon S(\fd(\rho))[\![\hbar]\!]\otimes_{\C[\![\hbar]\!]} S(\fd(\rho))[\![\hbar]\!]\to S(\fd(\rho))[\![\hbar]\!]\lbb z\rbb.
\ee
This product satisfies the associativity (which is the linear dual of coassociativity of $\Delta_{\rho,\hbar, z}$):
\be
\CY_{\rho,\hbar}(A, z)\CY_{\rho,\hbar}(B, w)=\CY_{\rho,\hbar}( \CY_{\rho,\hbar}(A, z-w)B,w).
\ee
Moreover, it has a unit $\Omega_\rho$ given by the linear dual of the counit $\epsilon: \CA_\hbar(\fd, \rho)\to \C\lbb\hbar\rbb$. However, $\CY_{\rho,\hbar}$ is \textbf{not} commutative since $\Delta_{\rho,\hbar,z} = (\tau_z\otimes 1)\Delta_{\rho,\hbar}$ is not cocommutative. We will show in next section, using another meromorphic coproduct, that, for rational \(\rho\), $\Delta_{\rho,\hbar, z}$ has a spectral $R$-matrix that quantizes \(\varrho\) from \eqref{eq:def_rho}. This makes $S(\fd(\rho))[\![\hbar]\!]$ together with $\CY_{\rho,\hbar}$ into a quantum vertex algebra that quantizes the quasi-classical commutative vertex algebra \(S(\fd(\rho))\) (see Section \ref{sec:quasi_classical_vertex}) in the sense of \cite{Etingof2000quantumvertex}; see Proposition \ref{prop:naiv_quantum_vertex} below. 

\subsubsection{The case of Yang's $r$-matrix \(\rho = \gamma\)}

Let us consider in this section the special case of Yang's $r$-matrix $\rho = \gamma = \gamma_\fd$ given in \eqref{eq:yang_rmat_d}. In this case, $\fd(\gamma)=t^{-1}\fd[t^{-1}]$, which is a graded Lie subalgebra of $\fd(\CK)$ under loop grading as well as $\epsilon$-grading. In this special case, it is easy to see that \(\CA_\hbar(\fd,\gamma)\) is a bi-graded Hopf algebra, i.e.\ it is not only graded with respect to the \(\epsilon\)-grading but also graded with respect to the loop grading by powers of \(t\).

We will from now on denote this special algebra by $Y_\hbar(\fd) \coloneqq \CA_\hbar(\fd,\gamma)$, signifying that it is really the natural version of the Yangian of $\fd$. Furthermore, we will denote its coproducts with \(\Delta_\hbar \coloneqq \Delta_{\gamma,\hbar}\) and \(\Delta_{\hbar,z} = \Delta_{\gamma,\hbar,z}\). This is justified in Section \ref{subsec:unique}, where we show that such a bi-graded quantization of $U(\fd(\CO))$ is unique. In this case, we also denote by $ Y^\circ_\hbar(\fd) \coloneqq \CA_\hbar^\circ(\fd,\gamma)$ the dense subalgebra generated by $\fd[t]$. This is a Hopf subalgebra of \(Y_\hbar(\fd)\) quantizing the Lie bialgebra \((\fd[t],\rho_\gamma)\) by virtue of Proposition \ref{Cor:finiteh}.

The action of $T$ on $Y_\hbar(\fd)$ is very explicit, and is given by the following:
\be
T(I_{a,n})=n I_{a, n-1},\qquad T (I^a_n)=n I^a_{n-1}.
\ee
Recall the evaluation at \(\hbar = \xi \in \C^\times\) from \ref{sec:evaluation}.
Combining Lemma \ref{Lem:nilT} and Proposition \ref{Cor:finiteh}, the category $Y^\circ_\xi(\fd)\Mod$ has the structure of a meromorphic tensor category, via the meromorphic coproduct $\Delta_{\xi, s}$. 

We end this section by specifying Proposition \ref{Prop:twistingF} to \(\rho_1 = \gamma\) and consider its consequence when $\rho$ depends only on $t_1-t_2$.

\begin{Prop}\label{Prop:TFT}
    For every \(\epsilon\)-graded \(r\)-matrix \(\rho\) with coefficients in \(\fd\), the Hopf algebra \(\CA_\hbar (\fd,\rho)\) is obtained from the Yangian \(Y_\hbar(\fd)\) by twisting with an element 
    \be
        F \in (S(\fg^*(\CO)) \otimes U(\fg(\CO)))[\![\hbar]\!].
    \ee
    Furthermore, if $\rho$ depends on the difference of its variables, $F$ satisfies:
    \be
F(T\otimes 1+1\otimes T)=(T\otimes 1+1\otimes T) F.
    \ee
    
\end{Prop}

\begin{proof}
    Let $\CE_\gamma$ and $\CE_\rho$ be defined as in Lemma \ref{Lem:group}, and let $F=\CE_\gamma^{-1}\CE_\rho$ be the quantum twisting as in Proposition \ref{Prop:twistingF}. Since $\CE_\gamma$ and $\CE_\rho$ define algebra homomorphisms that respect the action of $T$, we find
    \be
\CE(T\otimes 1+1\otimes T)=(T\otimes 1+1\otimes T)\CE, \qquad \CE\in \{\CE_\gamma,\CE_\rho\},
    \ee
from which it follows that the same relation for $F$ must hold. 

\end{proof}

\subsubsection{Uniqueness of graded quantization}\label{subsec:unique}
In this section, we want to prove the following claim.

\begin{Thm}\label{Thm:yangunique}
    The universal envelope
     \(U(\fd(\CO))\) has an, up to isomorphism, unique bi-graded continuous quantization, given by the Yangian \(Y_\hbar(\fd)\).
\end{Thm}

We prove this fact using Lie bialgebra cohomology, which controls the existence and uniqueness of quantizations according to \cite[Section 9]{drinfeld1986quantum}. Since we are not aware of a proof of this fact, we will outline one in the following.

Let us assume that we are given two Hopf algebra deformations \(H^{(1)}\) and \(H^{(2)}\) of \(U(\fd(\CO))\) and identify them with
\(U(\fd(\CO))[\![\hbar]\!]\) as \(\C[\![\hbar]\!]\)-modules. Let
\(\nabla_\hbar^{(i)},\Delta_\hbar^{(i)}\) and \(\delta_\hbar^{(i)} \coloneqq \Delta_\hbar^{(i)}-\Delta_\hbar^{(i),\textnormal{op}}\) be the multiplication, comultiplication and co-Poisson structure of \(H^{(i)}\) for \(i \in \{1,2\}\) assume that 
\begin{equation}
    \begin{split}
        &\nabla_\hbar^{(1)}-\nabla_\hbar^{(2)} = \hbar^k \mu + \CO(\hbar^{k+1});\\&
    \Delta_\hbar^{(1)}-\Delta_\hbar^{(2)} = \hbar^k \zeta + \CO(\hbar^{k+1});\\&
    \delta_\hbar^{(1)}-\delta_\hbar^{(2)} = \hbar^{k+1} \xi + \CO(\hbar^{k+2}).
    \end{split}
\end{equation}
We may assume that the units of \(H^{(1)}\) and \(H^{(2)}\) coincide up to \(\hbar^{k+1}\). Indeed, since the unit is unique, we have \(\eta^{(1)}_\hbar - \eta^{(2)}_\hbar = 1 + \hbar^ku + \CO(\hbar^{k+1})\) for the units \(\eta^{(1)}_\hbar\) and \(\eta^{(2)}_\hbar\) of \(H^{(1)}\) and \(H^{(2)}\) and some \(u \in U(\fd(\CO))\). After applying the left multiplication of \(1 - \hbar^ku\), considered as an invertible endomorphism of \(U(\fd(\CO))[\![\hbar]\!]\), we may assume that \(u = 0\). Consequently, \(\mu(1,-) = 0 = \mu(-,1)\) and \(\zeta(1) = 0\) holds.

For all \(a,b,c \in U(\fd(\CO))\)
\be\label{eq:mu_cocycle}
    \begin{split}
        &\nabla_\hbar^{(2)}(\nabla_\hbar^{(2)}(a,b),c) - \nabla_\hbar^{(2)}(a,\nabla_\hbar^{(2)}(b,c))  = 0 = \nabla_\hbar^{(1)}(\nabla_\hbar^{(1)}(a,b),c) - \nabla_\hbar^{(1)}(a,\nabla_\hbar^{(1)}(b,c)) \\&= (\nabla_\hbar^{(2)}+\hbar^k\mu)((\nabla_\hbar^{(2)}+\hbar^k\mu)(a,b),c) - (\nabla_\hbar^{(2)}+\hbar^k\mu)(a,(\nabla_\hbar^{(2)}+\hbar^k\mu)(b,c)) + \CO(\hbar^{k+1})
        \\& =\hbar^k(\mu(ab,c) + a\mu(b,c) - \mu(a,bc) - \mu(a,b)c) + \CO(\hbar^{k+1}).
    \end{split}
\ee
Similarly, the coassiociativity of \(\Delta^{(1)}_\hbar\) and \(\Delta^{(2)}_\hbar\) implies
\be\label{eq:zeta_condition_quantization}
    \begin{split}
        0 &= (\zeta \otimes 1)\Delta(a) + (\Delta \otimes 1)\zeta(a) - (1 \otimes \zeta)\Delta(a) - (1 \otimes \Delta)\zeta(a)
        \\& = \sum_i (a_i \otimes b_i \otimes 1 + \Delta(a_i) \otimes b_i - 1 \otimes a_i \otimes b_i - a_i \otimes \Delta(b_i)),
    \end{split}
\ee
where \(\zeta(a) = \sum_i a_i \otimes b_i\) (the sum is potentially infinite) and \(a \in \fd(\CO)\) was used.
Using the filtration of \(U(\fd(\CO))\), the identity \(\Delta(a_i) = a_i \otimes 1 + 1 \otimes a_i\) follows, which implies that \(a_i, b_i \in \fd(\CO)\). Since both \(H^{(1)}\) and \(H^{(2)}\) are Hopf algebras, they have antipodes \(S^{(1)}\) and \(S^{(2)}\). Since the antipode of a Hopf algebra is unique, we have \(S^{(1)} - S^{(2)} = \hbar^k \sigma + \CO(\hbar^{k+1})\). The antipodes are coalgebra anti-homomorphism, which implies that
\be
    \begin{split}
        &(S^{(2)} \otimes S^{(2)})\Delta_\hbar^{(2)} - \Delta_\hbar^{(2),\textnormal{op}} S^{(2)} = 0 = (S^{(1)} \otimes S^{(1)})\Delta_\hbar^{(1)} - \Delta_\hbar^{(1),\textnormal{op}} S^{(1)} 
        \\&= ((S^{(2)} + \hbar^{k}\sigma) \otimes (S^{(2)} + \hbar^{k}\sigma) (\Delta_\hbar^{(2)} + \hbar^k\zeta) - (\Delta_\hbar^{(2),\textnormal{op}}+\hbar^k\zeta)( S^{(2)} + \hbar^k\sigma) + \CO(\hbar^{k+1})
        \\&= \hbar^k((\sigma \otimes 1 + 1 \otimes \sigma)\Delta(a) + (S \otimes S)\zeta(a) - \Delta(\sigma(a)) - \zeta(S(a))) + \CO(\hbar^{k+1}).
    \end{split}
\ee
holds. For \(a \in \fd(\CO)\), we find that \(2\zeta(a)+(\sigma \otimes 1 + 1 \otimes \sigma)\Delta(a) - \Delta(\sigma(a)) = 0\) holds, since \((S\otimes S)\zeta(a) = \zeta(a) = - \zeta(S(a))\). Consider
\(\Delta^{(1),\prime}_\hbar \coloneqq (1 - \frac{1}{2}\hbar^k \sigma)^{\otimes 2} \Delta^{(1)}_{\hbar}(1 - \frac{1}{2}\hbar^k \sigma)^{-1}\), then
\be
    \Delta^{(1),\prime}_\hbar- \Delta_\hbar^{(2)} = \hbar^k\left(\zeta + \frac{1}{2}(\sigma \otimes 1 + 1 \otimes \sigma)\Delta - \frac{1}{2}\Delta\sigma\right) + \CO(\hbar^{k+1}).
\ee
Replacing \(H^{(1)}\) with the isomorphic Hopf algebra with coproduct \(\Delta^{(1),\prime}_\hbar\), we can make the following very important assumption:
\be
    \zeta(\fd(\CO)) = 0.
\ee

The fact that \(\delta^{(1)}_\hbar\) and \(\delta^{(2)}_\hbar\) are co-Poisson results in
\be
\begin{split}
    0 &= (\Delta \otimes 1)\xi(a) - (1 \otimes \xi)\Delta(a) - (1 \otimes \tau)(\xi \otimes 1)\Delta(a) \\&- (\delta \otimes 1)\zeta(a) + (1 \otimes \zeta)\delta(a) + (1 \otimes \tau)(\zeta \otimes 1)\delta(a) 
\end{split}
\ee
Observe that, since \(\zeta(\fd(\CO)) = 0\), we can see that \(\xi(\fd(\CO)) \subseteq \wedge^2\fd(\CO)\) by repeating the arguments in e.g.\ the proof of \cite[Proposition 6.2.3]{chari_pressley}. Furthermore, the co-Jacobi identity for \(\delta^{(1)}_\hbar\) and \(\delta^{(2)}_\hbar\) implies that \(\widetilde\xi \coloneqq  \xi|_{\fd(\CO)}\) defines a 2-cocycle of \(\fd(\CO)^* = \fd(\gamma) = t^{-1}\fd[t^{-1}]\) with values in \(\fd(\gamma)\).  

The fact that \(\Delta^{(1)}_\hbar\) and \(\Delta^{(2)}_\hbar\) are algebra homomorphisms provides
\be\label{eq:compatibility_mu_zeta}
\begin{split}
    0 = \Delta(\mu(a,b)) + \zeta(ab) - \mu(a,b) \otimes 1 - 1 \otimes \mu(a,b) - \zeta(a)\Delta(b) - \Delta(a)\zeta(b).
\end{split}
\ee
Again, since \(\zeta(\fd(\CO)) = 0\), we can see that \(\widetilde\mu \colon \wedge^2\fd(\CO) \to \fd(\CO)\) defined by \(\widetilde\mu(a,b) \coloneqq \mu(a,b)-\mu(b,a)\) takes values in \(\fd(\CO)\). Moreover, \eqref{eq:mu_cocycle} implies now that \(\widetilde\mu\) is a continuous 2-cocycle of \(\fd(\CO)\) with values in \(\fd(\CO)\).

Finally, since \(\delta_\hbar^{(1)}\) and \(\delta^{(2)}_\hbar\) are 1-cocycles of \(H^{(1)}\) and \(H^{(2)}\) as coalgebras, we get 
\be
\begin{split}
    0 &= \delta(\mu(a,b)) + \xi(ab) - \delta(a)\zeta(b) - \zeta(a)\delta(a) - \xi(a)\Delta(b) - \Delta(a)\xi(b) \\&- (\nabla\otimes \mu + \mu \otimes \nabla)(\delta^{(13)}(a)\Delta^{(24)}(b) + \Delta^{(13)}(a)\Delta^{(24)}(b)).
\end{split}
\ee
Combined again with the fact that \(\zeta(\fd(\CO)) = 0\), this implies that 
    \begin{equation}
        \begin{split}
            c(x_1 + f_1,x_2 + f_2,x_3 + f_3) &= f_1(\widetilde\mu(x_2,x_3))-f_2(\widetilde\mu(x_1,x_3))+ f_3(\widetilde\mu(x_1,x_2)) \\&- (f_1 \otimes f_2)\widetilde\xi(x_3) + (f_1 \otimes f_3)\widetilde\xi(x_2) - (f_2 \otimes f_3)\widetilde\xi(x_1)
        \end{split}
    \end{equation}
    for \(x_1,x_2,x_3 \in \fd(\CO)\) and \(f_1,f_2,f_3 \in \fd(\CO)^* = \fd_{<0}\). In other words, \(c\)
    is a continuous 3-cocycle of \(\fd(\CK) = \fd(\CO) \oplus \fd_{<0}\) with values in \(\C\) satisfying 
    \begin{equation}
        c(\wedge^3\fd(\CO)) = 0 = c(\wedge^3\fd(\gamma)).
    \end{equation}
    Assume now that \(H^{(1)}\) and \(H^{(2)}\) are additionally both \(\epsilon\) and loop graded as Hopf algebras. Then this implies that
    \begin{itemize}
        \item \(c(\wedge^3\fg(\CK)) = 0 = c(\wedge^2\fg^*(\CK) \wedge \fd(\CK))\);
        \item \(c(t^{k_1}\fd,t^{k_2}\fd,t^{k_3}\fd) = 0\) for \(k_1+k_2+k_3 \neq -k\).
    \end{itemize}
    holds.

\begin{Lem}
    The 3-cocycle \(c\) is a coboundary in the sense of \cite[Section 9]{drinfeld1986quantum}. More precisely, \(c = d\chi\) for a 2-chochain \(\chi \colon \wedge^2\fd(\CK) \to \C\) such that \(\chi(\wedge^2\fd(\CO)) = 0 = \chi(\wedge^2 \fd(\gamma))\). 
\end{Lem}
\begin{proof}
    Since \(c\) is a \(3\)-cocycle, for all \(d_1,d_2,d_3,d_4 \in \fd(\CK)\)
    \begin{equation}\label{eq:3cocycle}
        \begin{split}
            0 = dc(d_1,d_2,d_3,d_4) &= c([d_1,d_2],d_3,d_4) - c([d_1,d_3],d_2,d_4) + c([d_1,d_4],d_2,d_3) \\&+ c([d_2,d_3],d_1,d_4)-c([d_2,d_4],d_1,d_3) + c([d_3,d_4],d_1,d_2)
        \end{split}
    \end{equation}
    holds. 

    The assignment \((x,y) \mapsto c(x,y,t^{-k}(-))\) defines a 2-cocycle \(c' \colon\wedge^2\fg \to \fg^{**} \cong \fg\). More precisely,
    \begin{align*}
        f(dc'(x_1,x_2,x_3)) &= f\left(c'([x_1,x_2],x_3)-c'([x_1,x_3],x_2)+c'([x_2,x_3],x_1)\right. \\&\left.- [x_1,c'(x_2,x_3)] + [x_2,c'(x_1,x_3)] - [x_3,c'(x_1,x_2)]\right)\\
        &= c([x_1,x_2],x_3,t^{-k} f)-c([x_1,x_3],x_2,t^{-k} f)+c([x_2,x_3],x_1,t^{-k} f)\\&+c(x_2,x_3,[x_1,t^{-k} f])-c(x_1,x_3,[x_2,t^{-k} f])+c(x_1,x_2,[x_3,t^{-k} f])\\&
        =dc(x_1,x_2,x_3,t^{-k} f) = 0
    \end{align*}
    holds for all \(x_1,x_2,x_3 \in \fg\) and \(f \in \fg^*\). 
    
    Since \(\fg\) is simple, \(c'\) is a 2-coboundary 
    \[c(x_1, x_2,t^{-k} (-)) = d\chi(x_1,x_2) = \chi([x_1,x_2])-[x_1,\chi(x_2)]+[x_2,\chi(x_1)]\] 
    for some \(\chi \colon \fg \to  \fg\).
    Therefore, we can consider an extension \(\chi \colon \wedge^2\fd(\CK) \to \C\) given by \(\chi(\wedge^2\fg(\CK)) = 0 = \chi(\wedge^2\fg^*(\CK))\), \(\chi(\wedge^2\fd(\CO)) = 0 = \chi(\wedge^2\fd(\gamma))\), and
    \begin{equation}
        \chi(t^i x,t^j f) \coloneqq \delta_{i,0}\delta_{j,-k}f(\chi(x)) \eqqcolon -\chi(t^j f,t^i x) \,,\qquad x \in \fg,f \in \fg^*.
    \end{equation}
    We obtain 
    \begin{align*}
        d\chi(x_1,x_2,t^jf) &= \delta_{k,-j}f\left(\chi([x_1,x_2]) + [x_2,\chi(x_1)] - [x_1,\chi(x_2)]\right) = c(x_1,x_2,t^j f).
    \end{align*}
    Replacing \(c\) by \(c-d\chi\) provides \(c(x_1,x_2,t^jd) = 0\) for all \(x_1,x_2 \in \fg\), \(d \in \fd\) and \(j \in \Z\). 
    
    The assignment \(x \mapsto c(x,t^{-i} (-),t^{i-k} (-))\) defines a 1-cocycles
    \(
    c_{i}' \colon \fg \to \fg^*\otimes \fg
    \)
    for every \(i \in \Z\):
     \begin{align*}
         dc'_{i}(x_1,x_2)(x_3,f) &= f(c'_{i}([x_1,x_2])x_3 - [x_1,c'_{i}(x_2)x_3] + c'_{i}(x_2)[x_1,x_3] \\&+ [x_2,c'_{i}(x_1)x_3] - c'_{i}(x_1)[x_2,x_3]) \\&=
         c([x_1,x_2],t^{-i}x_3,t^{i-k} f)+
         c(x_2,[x_1,t^{-i} x_3],t^{i-k} f)+
         c(x_2,t^{-i} x_3,[x_1,t^{i-k} f])\\&+c(x_1,[x_2,t^{-i} x_3],t^{i-k} f)+c(x_1,t^{-i} x_3,[x_2,t^{i-k} f])\\&=dc(x_1,x_2,t^{-i} x_3,t^{i-k} f). 
     \end{align*}
     Here, we used in the last equality that \(c(x_1,x_2,[t^{-i} x_3,t^{i-k} f]) = 0\). Therefore, \(c'_{i} = d\chi_{i}\) for some \(\chi_{i} \in \fg^*\otimes \fg\). Let us define \(\chi \colon \wedge^2 \fd(\CK) \to \C\) via \(\chi(\wedge^2\fd(\CO)) = 0 = \chi(\wedge^2(\fd(\gamma))\), \(\chi(\wedge^2\fg(\CK)) = 0 = \chi(\wedge^2(\fg^*(\CK)))\), and 
     \[\chi(t^{-i}x,t^{j-k} f) \eqqcolon \delta_{ij}f(\chi_{i}(x)) \eqqcolon -\chi(t^{j-k} f,t^{-i} x) \,,\qquad i,j \in \Z, x \in \fg, f \in \fg^*.\] 
     Then 
     \begin{align*}
         d\chi(x_1,t^{-i} x_2,t^{j-k} f) &= \chi([x_1,t^{-i}x_2],t^{j-k}f) - \chi([x_1,t^{j-k} f],t^{-i}x_2) + \chi([t^{i}x_2,t^{j-k} f],x_1) \\&= -\delta_{i,j}f((x_1\cdot \chi_i)(x_2)) = -\delta_{i,j}c_i'(x_1)(x_2,f) = -c(x_1,t^{-i}x_2,t^{j-k} f),
     \end{align*}
    where we used that \(\chi([t^{-i}x_2,t^{j-k}f],x_1)  = 0\) and \(c(x_1,t^{-i}x_2,t^{j-k}f) = 0\) if \(i \neq j\).
    Therefore, after replacing \(c\) with \(c - d\chi\), we have \(c(x,t^id_1,t^j d_2) = 0\) for all \(x \in \fg\), \(d_1,d_2 \in \fd\) and \(i,j \in \Z\).

    Identifying \(\fd(\CK) \cong \fg(\CK)[\epsilon]/(\epsilon^2)\) and plugging \(d_1 = x, d_2 = \lambda_1x_1, d_3 = \lambda_2x_2,d_4 = \lambda_3x_3\) into \eqref{eq:3cocycle} for any \(x, x_1,x_2,x_3 \in \fg\) and \(\lambda_1,\lambda_2,\lambda_3 \in R \coloneqq \CK[\epsilon]/(\epsilon^2)\), we can see that the map \(c_{\lambda_1,\lambda_2,\lambda_3}\colon \fg^{\otimes 3} \to \C\) defined by \((x_1,x_2,x_3) \mapsto c(\lambda_1x_1,\lambda_2x_2,\lambda_3x_3)\) is \(\fg\)-invariant. 
    For all simple Lie algebras there is up to scalar multiple only one antisymmetric map of this form, except for \(\fg = \mathfrak{sl}_n(\C)\) where there is an additional symmetric map \((x_1,x_2,x_3) \mapsto \textnormal{tr}(x_1x_2x_3)\) possible. 
    
    Let us assume that \(c_{\lambda_1,\lambda_2,\lambda_3}(x_1,x_2,x_3) = \kappa([x_1,x_2],x_3)\varphi(\lambda_1,\lambda_2,\lambda_3)\). Then
    \[\varphi \colon R^{\otimes 3} \to \C\] 
    has the following properties:
    \begin{enumerate}
        \item \(\varphi(a,\lambda_1, \lambda_2) = \varphi(\lambda_1, a, \lambda_2) = \varphi(\lambda_1, \lambda_2,a) = 0\) for all \(a \in \C\);
        \item \(\varphi(R^{+,\otimes 3}) = 0 = \varphi((t^{-1}R^{-})^{\otimes 3})\), where \(R^+ = \CO[\epsilon]/(\epsilon^2)\) and \(R^- = \C[t^{-1},\epsilon]/(\epsilon^2)\);
        \item \(\varphi(\CK^{\otimes 3}) = 0 = \varphi((\epsilon\CK)^{\otimes 2} \otimes R)\);
        \item \(\varphi(\lambda_1\lambda_2,\lambda_3,\lambda_4) - \varphi(\lambda_2\lambda_3,\lambda_1,\lambda_4) + \varphi(\lambda_3\lambda_4,\lambda_1,\lambda_2) - \varphi(\lambda_4\lambda_1,\lambda_2,\lambda_3) = 0\).
    \end{enumerate}
    Here, \(\lambda_1,\lambda_2,\lambda_3,\lambda_4 \in R\) and 4.\ follows from \(dc(\lambda_1x,\lambda_2y,\lambda_3x,\lambda_4y) = 0\).
    Let us observe that
    \(\varphi(\lambda_1,\lambda_2,\lambda_3) = \varphi(\lambda_{\sigma(1)},\lambda_{\sigma(2)},\lambda_{\sigma(3)})\) for all bijections \(\sigma \colon \{1,2,3\} \to \{1,2,3\}\) by setting \(\lambda_4 = 1\) in property 4. 
     
    The properties 2.\ and 3.\ combined with the symmetry of \(\varphi\) imply that \(\varphi\) is completely determined by two continuous linear maps \(\varphi_\pm \colon \CO^\pm \otimes \CO^\pm \to \CO^\pm\) given by 
    \begin{equation}
        \varphi(\lambda_1,\lambda_2,\epsilon\lambda_3) = \textnormal{res}_{t = 0}\varphi_\pm(\lambda_1,\lambda_2)\lambda_3\,,\qquad \lambda_1,\lambda_2 \in \CO^\pm, \lambda_3 \in \CO^\mp.
    \end{equation}
    using the fact that \(\CO^+ \coloneqq \CO\) and \(\CO^- \coloneqq \C[t^{-1}]\) are dual to \(t^{-1} \CO^\mp\) respectively. Property 4.\ now translates to the fact that \(\varphi_\pm\) are 2-cocycles. Since \(\CO\) and \(\C[t^{-1}]\) are smooth, \(\varphi_\pm\) are coboundaries, so
    \begin{equation}
    \varphi_\pm(\lambda_1,\lambda_2) = \lambda_1\phi_\pm(\lambda_2) - \phi_\pm(\lambda_1\lambda_2) + \phi_\pm(\lambda_1)\lambda_2,
    \end{equation}
    for all \(\lambda_1,\lambda_2 \in \CO^\pm\) and some \(\phi_\pm \colon \CO^\pm \to \CO^\pm\).

    The condition \(\varphi(\lambda_1,\lambda_2,1)=0\) implies
    \be
        \textnormal{res}_{t = 0}(\phi_\pm(\lambda_1\lambda_2)) = \textnormal{res}_{t = 0}(\phi_\pm(\lambda_1)\lambda_2+\lambda_1\phi_\pm(\lambda_2))
    \ee 
    for all \(\lambda_1,\lambda_2 \in \CO^\pm\). The additional condition \(\lambda_1 = 1\) provides \(\phi_\pm(\C) = 0\). Clearly, \(\phi_+(\CO^+)\subseteq \CO^+\) implies \(\textnormal{res}_{t = 0}\phi_+(\CO) = 0\). Moreover, since \(\varphi_-(\CO^-,\CO^-) \subseteq t^{-k-2}\CO^- \subseteq t^{-2}\CO^-\), we may assume that \(\phi_-(\CO^-) \subseteq t^{-2}\CO^-\) and therefore \(\textnormal{res}_{t = 0}\phi_-(\CO^-) = 0\) as well.
    
    Consider \(\phi \colon \CK \to \CK\) defined by \(\phi|_{\CO^\pm} = \phi_\pm\), which is well-defined, since \(\phi_+(\lambda) = 0 = \phi_-(\lambda)\) for \(\lambda \in \CO^+\cap \CO^- = \C\). Then \(0 = \textnormal{res}_{t = 0}\phi_\pm(\lambda_1\lambda_2) = \textnormal{res}_{t = 0}(\lambda_1\phi(\lambda_2)) + \textnormal{res}_{t = 0}(\lambda_2\phi(\lambda_1))\) for all \(\lambda_1,\lambda_2 \in \CK\) implies
    \be
        \varphi(\lambda_1,\lambda_2,\epsilon \lambda_3) = -\textnormal{res}_{t = 0}\phi(\lambda_1\lambda_2)\lambda_3 = -\textnormal{res}_{t = 0}\phi(\lambda_1\lambda_3)\lambda_2 = -\textnormal{res}_{t = 0}\phi(\lambda_2\lambda_3)\lambda_1.
    \ee
    It is easy to deduce that \(c = d\chi\) for \(\chi(\lambda_1x_1,\epsilon\lambda_2x_2) = \frac{1}{3}\kappa(x_1,x_2)\textnormal{res}_{t = 0}\phi(\lambda_1)\lambda_2\), so \(c\) is then a coboundary.
    Furthermore, \(c(\wedge^3 \fd(\CO)) = 0 = c(\wedge^3\fd(\gamma))\) combined with \(c(1,-,-) = 0\) implies \(\chi(\wedge^2\fd(\CO)) = 0 = \chi(\wedge^2\fd(\gamma))\).

    It remains to prove that if \(\fg = \mathfrak{sl}_n(\C)\) and
    \begin{equation}
        c(\lambda_1 x_1,\lambda_2 x_2,\lambda_3 x_3) = \textnormal{tr}([x_1,x_2]x_3)\varphi(\lambda_1,\lambda_2,\lambda_3) + \textnormal{tr}(x_1x_2x_3)\psi(\lambda_1,\lambda_2,\lambda_3)    
    \end{equation}
    then \(\psi = 0\) holds automatically. Here, we used that \(\kappa\) is proportional to the trace pairing for \(\fg = \mathfrak{sl}_n(\C)\).
    Observe that \(c(\lambda_1x,\lambda_2x,\lambda_3y) = \textnormal{tr}(x^2y)\psi(\lambda_1,\lambda_2,\lambda_3)\) for all \(x,y \in \fg\) and \(\lambda_1,\lambda_2,\lambda_3 \in R\) implies \(\psi(\lambda_2,\lambda_1,\lambda_3) = - \psi(\lambda_1,\lambda_2,\lambda_3)\) since \(c\) is skew-symmetric. But since the linear maps \((x,y) \mapsto \textnormal{tr}([x,y]^2)\) and \((x,y) \mapsto \textnormal{tr}([x,y]xy)\) are linearly independent, \(dc(\lambda_1x,\lambda_2y,\lambda_3x,\lambda_4y) = 0\) implies 
    \begin{equation}
        \psi(\lambda_1\lambda_2,\lambda_3,\lambda_4) - \psi(\lambda_2\lambda_3,\lambda_1,\lambda_4) + \psi(\lambda_3\lambda_4,\lambda_1,\lambda_2) - \psi(\lambda_4\lambda_1,\lambda_2,\lambda_3) = 0
    \end{equation}
    and therefore \(\psi(\lambda_2,\lambda_1,\lambda_3) = \psi(\lambda_1,\lambda_2,\lambda_3)\) for \(\lambda_4 = 1\). Combined, we deduce that \(\psi = 0\), concluding the proof.
\end{proof}
Identifying \(\chi\) with a linear map \(\chi \colon \fd(\CO) \to \fd(\CO)\), we have
    \begin{align*}
        & \widetilde\mu(x_1,x_2)= [\chi(x_1),x_2] + [x_1,\chi(x_2)] - \chi([x_1,x_2])
        \\& (f_1 \otimes f_2)\xi= (f_1\chi \otimes f_2)\delta + (f_1 \otimes f_2\chi)\delta-(f_1 \otimes f_2)\delta\chi.
    \end{align*}
    Using the same argument as in \cite[Corollary XVIII.1.3]{Kassel1995quantum}, we can see that there exists an extension \(\chi \colon U(\fd(\CO))\to U(\fd(\CO))\) such that \(\mu(a,b) = \chi(a)b + a\chi(b)- \chi(ab)\).
    Consequently, the deformations \(H^{(1)}\) and \(H^{(2)}\) are equivalent to order \(\hbar^{k+1}\) via the invertible endomorphism 
    \be
        (1 - \hbar^{k} \chi) \colon U(\fd(\CO))[\![\hbar]\!] \to U(\fd(\CO))[\![\hbar]\!],
    \ee
    since
    \begin{equation}
        \begin{split}
            &\nabla_\hbar^{(1)}((1 - \hbar^k\chi)a,(1-\hbar^k\chi)b) - (1+\hbar^k\chi)\nabla_\hbar^{(2)}(a,b) \\&=  \mu(a,b) + \chi(ab) -\chi(a)b - a\chi(b) ´+ \CO(\hbar^{k+1}) = \CO(\hbar^{k+1}); 
        \\&\delta^{(1)}_\hbar((1-\hbar^{k}\chi)a) - (1 - \hbar^k\chi) \otimes (1-\hbar^k\chi) \delta^{(2)}_\hbar(a)\\&= -\delta(\chi(a)) + \xi(a) + (\chi \otimes 1)\delta(a) + (1 \otimes \chi)\delta(a) + \CO(\hbar^{k+2})= \CO(\hbar^{k+2}).
        \end{split}
    \end{equation}
Moreover, since after this transformation \(\mu = 0\), we can deduce that \(\zeta = 0\) holds as well, by using \(\zeta(\fd(\CO)) = 0\) and \eqref{eq:compatibility_mu_zeta} inductively. In summary, \(H^{(1)}\) and \(H^{(2)}\) coincide to order \(\hbar^{k+1}\) after this transformation. 
Inductively on \(k\), we can deduce that \(H^{(1)}\) and \(H^{(2)}\) are isomorphic as bi-graded Hopf algebras over \(\C[\![\hbar]\!]\), proving Theorem \ref{Thm:yangunique}.

\section{Quantum $R$-matrix}\label{sec:MeroR}

In this section, we show that just like the usual Yangian for $\fg$, our Hopf algebra $Y^\circ_\hbar(\fd)$ admits a spectral $R$-matrix $R(z)$, and moreover, this $R$-matrix factorizes into a product of meromorphic twisting matrices, very much similar to Yangian of semisimple Lie algebras \cite{gautam2021meromorphic}. In fact, this will turn out to be an extremely convenient way to compute this $R$-matrix. Note that we must use the algebra $ Y^\circ_\hbar(\fd)$ instead of $Y_\hbar(\fd)$ because in general $\Delta_{\hbar, z}(a)$ for $a\in Y_\hbar(\fd)$ is a formal power series in \(z\) and \(z^{-1}\), and algebra multiplication is usually not meaningful in $Y_\hbar(\fd)\lbb z,z^{-1}\rbb$ (as $R(z)$ will have infinite poles).

The idea of the construction will be as follows. We first use the affine Kac-Moody vacuum vertex algebra $V_0(\fg)$ to show that $ Y^\circ_\hbar(\fd)$ admits another meromorphic coproduct $\Delta_z$:
\be
\Delta_z\colon  Y^\circ_\hbar(\fd)\to ( Y^\circ_\hbar(\fd)\otimes_{\C\lbb\hbar\rbb}  Y^\circ_\hbar(\fd))\lpp z^{-1}\rpp.
\ee
The main statement about this is the following proposition.

\begin{Prop}
    The  $\Delta_z$ satisfies the following conditions. 

    \begin{enumerate}
        \item It is weakly coassociative and weakly cocommutative in the following sense:
        \be
(\Delta_{z_1}\otimes 1)\Delta_{z_2}=(1\otimes \Delta_{z_2}) \Delta_{z_1+z_2},\qquad \Delta_{z}^{\textnormal{op}}=(\tau_z\otimes \tau_z) \Delta_{-z}.
        \ee

        \item The action of $T$ is a coderivation for $\Delta_z$:
        \be
\Delta_zT=(T\otimes 1+1\otimes T)\Delta_z.
        \ee

          \item It is the Taylor expansion at $z^{-1}$ of a well-defined rational function when evaluated on any tensor product of finite-dimensional smooth modules. 
        
    \end{enumerate}
\end{Prop}

We then show, in an extremely similar way to the proof of Proposition \ref{Prop:twistingF}, that there exists a meromorphic twisting, which we denote by $R_s(z)$ that intertwines $\Delta_z$ and $\Delta_{\hbar, z}$, satisfying a cocycle condition. This is summarized into the following: 

\begin{Thm}\label{Thm:meroR}
    There exists an element $R_s(z)\in (Y^\circ_\hbar (\fd)\otimes Y_\hbar^\circ (\fd))\lbb z^{-1}\rbb$, such that: 

    \begin{enumerate}
         \item The element $R_s(z)$ intertwines the two coproducts: $R_s(z)^{-1}\Delta_{\hbar, z}(a) R_s(z)=\Delta_z(a)$ for all \(a \in Y^\circ_\hbar(\fd)\).

    \item It satisfies the associativity axiom
    \begin{equation}
        (\Delta_{z_1}\otimes 1)(R_s(z_2)^{-1})R_s^{12}(z_1)^{-1}=(1\otimes \Delta_{z_2})(R_s(z_1+z_2)^{-1}) R_s^{23}(z_2)^{-1}.
    \end{equation} 

    \item It satisfies $(\tau_{z_1}\otimes \tau_{z_2})R_s(z_3)=R_s(z_3+z_1-z_2)$. 
    
    \item It is the Taylor expansion at $z^{-1}$ of a well-defined analytic function when evaluated on any tensor product of finite-dimensional smooth modules. 

    \item It is the unique element with these properties in $(S(\fg^*(\CO))\lbb\hbar\rbb\otimes_{\C\lbb\hbar\rbb} U(\fg(\CO))\lbb\hbar\rbb)\lpp z^{-1}\rpp$. 
    
    \end{enumerate}
\end{Thm}

As a consequence of this and the weak cocommutativity of $\Delta_z$, we construct the meromorphic $R$-matrix of $\Delta_{\hbar, z}$. 

\begin{Thm}\label{Thm:fullR}
    Let $R(z)=R_{s}^{21}(-z)R_s(z)^{-1}$, then:
      \begin{enumerate}
      
        \item \(
            (\tau_z\otimes 1)\Delta_{\hbar}^{\textnormal{op}}(a)=R(z)  (\tau_z \otimes 1)\Delta_{\hbar}(a) R(z)^{-1}\) for all \(a \in Y^\circ_\hbar(\fd)\);
        
        \item \((\Delta_{\hbar, z_1}\otimes 1) R(z_2)=R^{13}(z_1+z_2)R^{23}(z_2)\), \((1\otimes \Delta_{\hbar, z_2})R(z_1+z_2)=R^{13}(z_1+z_2)R^{12}(z_2)\) and \(R(z)\) is a solution of quantum Yang-Baxter equation:
        \be
            R^{12}(z_1)R^{13}(z_1+z_2)R^{23}(z_2)=R^{23}(z_2)R^{13}(z_1+z_2)R^{12}(z_1);
        \ee

        \item \(R(z)\) satisfies $(\tau_{z_1}\otimes \tau_{z_2})R(z_3)=R(z_3+z_1-z_2)$; 

        \item \(R(z)\) is the Taylor expansion at $z=\infty$ of a well-defined analytic function when evaluated on any tensor product of finite-dimensional smooth modules. 

    \end{enumerate}
\end{Thm}

\begin{Rem}
    In the above statements, any equality between two formal series means that they are the expansion of the same common series in two different domains of convergence, just as the case of vertex algebras in \cite[Section 1]{frenkel2004vertex}. In fact, this will remain true throughout the section. More details on what is meant by this will be given in Remark \ref{rem:equality_of_series}.  From this notion of equality of series, it will always follow that as soon as two series can be evaluated in \(\hbar\) at some $s\in \C^\times$, the equality will mean equality as meromorphic functions. 
\end{Rem}

Using Proposition \ref{Prop:twistingF} and Proposition \ref{Prop:TFT}, we can adjust Theorem \ref{Thm:fullR} when \(\gamma\) is replaced by other \(\epsilon\)-graded \(r\)-matrices \(\rho\). This will be the content of Section \ref{sec:twisting_R_matrices}.

The proof of Theorem \ref{Thm:meroR} and Theorem \ref{Thm:fullR} turns out to be extremely natural. The idea is to recognize the intertwining operator $\CY$ as an element in $(S(\fg^*(\CO))\lbb\hbar\rbb\otimes_{\C\lbb\hbar\rbb} \End(M)[\![\hbar]\!])\lbb z,z^{-1}\rbb$ for any smooth $\fg(\CK)$-module $M$ and show that it can be written as $\CY=\CE_\gamma(z)\cdot R_s(z)$, where $\CE_\gamma(z)$ is the tensor series associated to the map $U(\fg_{<0})\to U(\fg(\CK))\lbb z\rbb\to \End(M)\lbb z\rbb$ given by $A\mapsto e^{-zT}(A)$. As a consequence, $R_s(z)$ can be seen as a rotation from multiplication to intertwining operators:
\be
\btik
e^{-zT}(A)\cdot B\rar{R_s} &  \CY(A,z)B
\etik
\ee
which then allows us to rotate in the following way:
\be
\btik
e^{-zT}(A)\cdot B\rar{R_s} &  \CY(A,z)B\rar{=} & e^{zT} \CY(B,-z)A\rar{R_s^{-1}} &e^{zT} (e^{zT}(B)\cdot A)
\etik
\ee
Here the equality in the second arrow is the commutativity condition of vertex algebras (also called skew-symmetry, see \cite[Proposition 3.2.5]{frenkel2004vertex}).

The structure of this section is as follows. In Section \ref{subsec:quantVHopf}, we construct the coproduct $\Delta_z$ using the vertex algebra structure of $V_0(\fg)$. In Section \ref{subsec:MeroR}, we construct the twisting matrix $R_s(z)$ and prove Theorem \ref{Thm:meroR}. In Section \ref{subsec:fullR}, we construct the $R$-matrix and prove Theorem \ref{Thm:fullR}. The generalization of Theorem \ref{Thm:fullR} if \(\gamma\) is replaced by other \(\epsilon\)-graded \(r\)-matrices is discussed in Section \ref{sec:twisting_R_matrices}.

\subsection{Meromorphic tensor product for the quantization}\label{subsec:quantVHopf}

The vacuum vertex algebra $V_0(\fg)$ at level $0$ is by definition
\be
V_0(\fg):=\Ind_{\fg(\CO)}^{\fg(\CK)}(\C).
\ee
as a \(\fg(\CK)\)-module. Its intertwining operator $\CY\colon V_0(\fg) \to \textnormal{End}(V_0(\fg))[\![z,z^{-1}]\!]$ is defined so that $\CY(I_{a, -1}, z)=\sum_{n \in \Z} I_{a, n}z^{-n-1}$, where we recall that \(I_{a,n} = I_a t^n\). As a vector space, we can identify $V_0(\fg)$ with $U(\fg_{<0})$, and as such, it has a cocommutative coproduct $\Delta$. We would like to first show that this coproduct is a morphism of vertex algebras. The following is clear:

\begin{Prop}\label{Cor:coYcommut}
    Let $V_0(\fg)^{\otimes 2}$ be the tensor product vertex algebra of $V_0(\fg)$. Then there is an embedding of vertex algebras $\Delta': V_0(\fg)\to V_0(\fg)^{\otimes 2}$ defined by:
    \be\label{eq:vertex_embedding}
     \CY(I_{a,-1}, z)\stackrel{\Delta'}\longmapsto \CY(I_{a,-1}, z)\otimes 1+1\otimes \CY(I_{a,-1}, z). 
    \ee
    Moreover, under the identification $V_0(\fg)\cong U(\fg_{<0})$, this $\Delta'$ is identified with the standard coproduct $\Delta$ of $U(\fg_{<0})$ and therefore the following diagram commutes: 
    \be
    \btik
        U(\fg_{<0})\otimes U(\fg_{<0})\rar{\CY}\dar{\Delta} & U(\fg_{<0})\lpp z\rpp\dar{\Delta}\\
        U(\fg_{<0})^{\otimes 2}\otimes U(\fg_{<0})^{\otimes 2} \rar{\CY^{\otimes 2}} & U(\fg_{<0})^{\otimes 2}\lpp z\rpp
    \etik
    \ee
    In particular, \(\CY\) is a morphism of cocommutative coalgebras in this way.
\end{Prop}
\begin{proof}
    The vertex algebra $V_0(\fg)$ is strongly generated by $\CY(I_{a,-1}, z)$, and therefore we only need to check that the right-hand side in \eqref{eq:vertex_embedding} has the correct OPE, which is simple. 
    
    The usual coproduct of \(U(\fg_{<0})\) and \(\Delta'\) coincide on primary fields $I_{a, -n}$. We use induction. Suppose we now that they coincides on $B$, we show that they coincides on $I_{a, -n}B$ for all $I_{a, -n}$. We have
    \be
    \begin{aligned}
        \Delta'(I_{a,-n} B)&=\Delta'\left(\mathrm{res}_{z=0}\frac{1}{z^n}\CY(I_{a,-1}, z)B\right)=\mathrm{res}_{z=0}\frac{1}{z^n} \CY(\Delta'(I_{a,-1}), z)\Delta'(B)\\
        &= \mathrm{res}_{z=0}\frac{1}{z^n}(\CY(I_{a,-1}, z)\otimes 1+1\otimes \CY(I_{a,-1}, z))\Delta(B)\\&=\Delta(I_{a,-n})\Delta(B)=\Delta(I_{a,-n}B),
    \end{aligned}
    \ee
    and this is the proof. 
\end{proof}

We can now use the duality between \(U(\fg_{<0}) = U(\fg(\gamma_\fg))\) and \(S(\hbar\fg^*(\CO))[\![\hbar]\!]\) from Lemma \ref{Lem:group} in order to dualize \(\CY\). In the following, recall that \(\CO\) is \(\C[\![t]\!]\), and it will be convenient to remain with the latter notation sometimes.

\begin{Lem}
    For any $f\in S(\hbar\fg^*[t])[\![\hbar]\!]$, define $\CY^\vee(f)\in (S(\fg^*\lbb t\rbb)\otimes S(\fg^*\lbb t\rbb)\lbb\hbar]\!][\![z,z^{-1}\rbb$ by
    \be
\lag \CY^\vee(f), A\otimes B\rag=\lag f, \CY(A, z)B\rag.
    \ee
    for any \(A, B \in U(\fg_{<0})[\![\hbar]\!]\).
    Then $\CY^\vee(f)$ is valued in $(S(\hbar\fg^*[t])\otimes S(\hbar\fg^*[t]))\lbb\hbar\rbb\lpp z^{-1}\rpp$ and defines an algebra homomorphism. 
    
\end{Lem}

\begin{proof}
    We must first show that $\CY^\vee$ is well-defined. We can write $\CY(A, z)B=\sum_n \CY_n(A)B z^n$, where $A \otimes B \mapsto \CY_n(A)B\) defines a linear map \(U(\fg_{<0})\otimes U(\fg_{<0}) \to U(\fg_{<0})\). The \(\C[\![\hbar]\!]\)-linear extension of this map has a well-defined dual $\CY_n^\vee \colon S(\hbar\fg^*\lbb t\rbb) \to (S(\hbar\fg^*[\![t]\!])\otimes S(\hbar\fg^*[\![t]\!]))\lbb\hbar\rbb$ according to Lemma \ref{Lem:group}. The element $\CY^\vee$ is simply given by $\CY^\vee=\sum_{n\in\Z}  \CY^\vee_nz^n$. 
    
    We use the loop grading to show that $\CY^\vee$ is valued in $(S(\hbar\fg^*[t])\otimes S(\hbar\fg^*[t]))\lbb\hbar\rbb\lpp z^{-1}\rpp$. Let $f$ have loop degree $m>0$ and $A, B$ have loop degree $-n, -p$ respectively, where $n,p>0$.  For every $\ell$ we have
    \begin{equation}\label{eq:pairing_CYvee}
        \lag \CY^\vee(f)_\ell, A\otimes B\rag=\lag f, \CY_\ell(A)B\rag,    
    \end{equation}
    where $\CY_\ell(A)B$ is in loop degree degree $-n-p-\ell$. Note that for $p,n$ large enough, $\ell+n+p$ will be larger than $m$ and therefore the pairing \eqref{eq:pairing_CYvee} is zero. This implies that $\CY^\vee(f)$ is valued in $(S(\hbar\fg^*[t])\otimes S(\hbar\fg^*[t]))\lbb\hbar\rbb\lbb z,z^{-1}\rbb$. We are left to show that $ \CY^\vee(f)_\ell=0$ for $\ell$ large enough. This is true because $A, B$ are both negatively graded, so when $\ell > m$, $n+\ell+p>m$ no matter what \(n,p\) are. Therefore, the pairing \eqref{eq:pairing_CYvee} is identically zero. It is clear that $\CY^\vee$ is an algebra homomorphism, thanks to Proposition \ref{Cor:coYcommut}.
\end{proof}

 The restriction of $\CY^\vee$ to $\hbar\fg^*[t]$ involves at least one order of $\hbar$, so we can extend \(\CY^\vee\) to a coproduct $\Delta_z\colon S(\fg^*[t])\lbb\hbar\rbb \to (S(\fg^*[t])\otimes S(\fg^*[t]))\lbb\hbar\rbb\lpp z^{-1}\rpp$ by writing
 \begin{equation}\label{eq:Delta_z_definition}
     \Delta_z(f) \coloneqq \hbar^{-1}\CY^\vee(\hbar f)
 \end{equation}
for \(f \in \fg^*[t]\).

\begin{Lem}\label{Lem:weakassoc}
    The algebra homomorphism \(\Delta_z\) defined in \eqref{eq:Delta_z_definition} satisfies the following associativity condition: for any $A,B,C \in V_0(\fg)$ and \(f \in S(\fg^*[t])[\![\hbar]\!]\)
    \be
        \langle (\Delta_{z_1}\otimes 1)\Delta_{z_2}(f), A\otimes B\otimes C\rangle =\langle  (1\otimes \Delta_{z_2}) \Delta_{z_1+z_2}(f), A\otimes B\otimes C\rangle
    \ee
    holds, where the equality means both sides are the expansion of the same meromorphic function in two different domain of convergence.

\end{Lem}

\begin{proof}
    By definition, we have
    \(
\lag (\Delta_{z_1}\otimes 1)\Delta_{z_2}(f), A\otimes B\otimes C\rag=\lag f, \CY(\CY(A, z_1)B, z_2)C\rag
    \). By the choice of $f$ and $A, B, C$, this is an element in $\C\lbb\hbar\rbb [z_2, z_2^{-1}]\lpp z_1\rpp$. 
    Locality of the vertex algebra \(V_0(\fg)\) implies that
    \be
        \begin{split}
            \lag f, \CY(\CY(A, z_1)B, z_2)C\rag&=\lag f, \CY(A,z_1+z_2)\CY(B, z_2)C\rag \\&= \lag (1\otimes \Delta_{z_2}) \Delta_{z_1+z_2}(f), A\otimes B\otimes C\rag
        \end{split}
    \ee
    holds. From \cite[Chapter 3]{frenkel2004vertex}, we find that both sides are the expression of the same element (borrowing their notation)
    \be
 f_{A, B, C}\in   \C [z_1^\pm, z_2^\pm, \frac{1}{z_1+z_2}]\lbb\hbar\rbb
    \ee
    in the specific domains of the two sides. This completes the proof. 
\end{proof}

\begin{Rem}
    In the following, we will call the coassociativity condition of Lemma \ref{Lem:weakassoc} \textnormal{weak coassociativity}. It is precisely the dual of the weak associativity condition satisfied by vertex algebras. 
\end{Rem}

Define now
\be
\Delta_z (x)=\tau_z (x)\otimes 1+1\otimes x=e^{zT}(x)\otimes 1+1\otimes x,
\ee
for $x\in \fg[t]$ in order to obtain an linear map 
\be\label{eq:complete_Delta_z}
    \Delta_z \colon  Y^\circ_\hbar(\fd)\to ( Y^\circ_\hbar(\fd)\otimes_{\C\lbb\hbar\rbb}  Y^\circ_\hbar(\fd))\lpp z^{-1}\rpp.
\ee
under consideration of \eqref{eq:Delta_z_definition}.

\begin{Prop}
    The map \(\Delta_z\) from \eqref{eq:complete_Delta_z} is well-defined weakly coassociative algebra homomorphism. Moreover, it is \textnormal{weakly cocommutative} in the sense that \(
\Delta_{z}^{\textnormal{op}}=(\tau_z\otimes \tau_z) \Delta_{-z}
    \) holds.
\end{Prop}

\begin{proof}
    It is clear by now that \(\Delta_z\) is an algebra homomorphism on \(U(\fg[t])\) and \(S(\fg^*[t])\). Therefore, we just need to see that \(\Delta_z\) is compatible with the commutation relation of elements in $\fg[t]$ and $S(\fg^*[t])$ with each other. More precisely, we need to show that 
    \be
        \Delta_z([I_{a,n},f]) = [\Delta_z(I_{a,n}),\Delta_z(f)]
    \ee
    for all \(f \in \fg^*[t]\). To this end, we evaluate both sides with elements \(A,B \in V_0(\fg) = U(\fg_{<0})\). The left-hand side becomes:
    \be
        \begin{split}
            \lag \Delta_z[I_{a, n},f], A\otimes B\rag&= \lag [I_{a,n},f],\CY(A,z)B\rag =-\lag f, I_{a, n}\rhd \CY(A, z)B\rag\\&=-\mathrm{res}_{w=\infty}w^n\lag f, \CY(I_{a, -1}, w) \CY(A, z)B\rag.
        \end{split}
    \ee
    By residue formula and locality, the argument of \(-f\) can be re-written as:
    \be
    \begin{aligned}
        &\mathrm{res}_{w=\infty}w^n\CY(I_{a,-1}, w) \CY(A, z)B
        \\&=\mathrm{res}_{w=0} w^n \CY(A, z) \CY(I_{a,-1}, w)B+\mathrm{res}_{w=z} w^n \CY(\CY(I_{a,-1}, w-z)A, z)B\\ &=\CY(A, z)I_{a, n} B+ \sum_{m\leq n}\binom{n}{m} z^{n-m} \CY(I_{a, m}A, z) B.
    \end{aligned}
    \ee
     This coincides beautifully with $\lag [\Delta_z(I_{a, n}), \Delta_z(f)], A\otimes B\rag$, proving that \(\Delta_z\) is indeed an algebra homomorphism. 
    
    The weak coassociativity and cocommutativity on the generators of $\fg[t]$ is clear.  We are left to prove the weak cocommutativity for elements $f \in S(\fg^*[t])$. Evaluating again on $A,B\in V_0(\fg)$, we see that:
    \be
        \begin{split}
            \lag \Delta_{z}^{\textnormal{op}}(f), A\otimes B\rag&=\lag f, \CY(B, z)A\rag=\lag f, e^{-zT}\CY(A, -z)B\rag \\&=\lag \Delta_{-z}\tau_z(f), A\otimes B\rag 
            \\& = \lag (\tau_z \otimes \tau_z)\Delta_{-z},A \otimes B\rag.
        \end{split}
    \ee
    In the last expression we used that $T$ is clearly a coderivation for $\Delta_z$. This concludes the proof.
\end{proof}

\begin{Rem}
    Note that the minus sign in the proof above, used in the weak commutativity of the intertwining operators, is due to our definition of $T$. It differs from the usual definition in vertex algebras by a sign. 
\end{Rem}

\begin{Rem}
    Note that we need to define $\Delta_z$ on $ Y^\circ_\hbar(\fd)$ because otherwise $\Delta_z$ is valued in $(Y_\hbar(\fd)\otimes_{\C\lbb\hbar\rbb}Y_\hbar(\fd))\lbb z,z^{-1}\rbb$, which is not an algebra.
    
\end{Rem}

\begin{Rem}\label{Rem:dualVA}
    The loop-graded dual of $Y^\circ_\hbar(\fd)$ can be identified with \(V_0(\fg) \otimes S(t^{-1}\fg^*[t^{-1}]))[\![\hbar]\!]\)
    and in doing so $\Delta_z$ becomes the linear dual of the intertwining operator of this vertex algebra. Here, $\CV:= V_0(\fg) \otimes S(t^{-1}\fg^*[t^{-1}]) $ is the tensor product vertex algebra, where the vertex algebra structure on $S(t^{-1}\fg^*[ t^{-1}])$ is commutative.
\end{Rem}

Finally, we show that the structures above have well-defined evaluation at $\hbar=\xi\in \C$ and $z=s\in \C^\times$, if we consider smooth modules of $Y^\circ_\xi(\fd)$. 

\begin{Prop}
    Let $\xi\in \C$ and $s\in \C^\times$, then for two finite-dimensional smooth modules \(M,N\) of $Y^\circ_\xi(\fd)$, the coproduct $\Delta_z$ has a well-defined evaluation at $z=s$ on \(M \otimes N\).
\end{Prop}

\begin{proof}
    It is another simple consequence of loop grading. The statement is clear for the generators $\fg[t]$, and therefore we only need to consider $\fg^*[t]$. Note that the VOA $V_0(\fg)$ is graded by the loop grading and that $\CY$ respects this grading, and therefore $\Delta_z$ respects the loop grading as well. However, since $V_0(\fg)$ is graded in negative degrees with finite-dimensional graded pieces, its dual $V_0(\fg)^*=S(\hbar \fg^*(\CO))$ is positively topologically graded with finite-dimensional graded pieces. In particular, for each $I^a_n$, the coproduct $\Delta_z(I^a_n)$ is of the form:
    \be
\Delta_z(I^a_n)=\sum_{m} z^m f(I^b_k\otimes 1, 1\otimes I^c_j) 
    \ee
    where $f(I^b_k\otimes 1, 1\otimes I^c_j)$ is an element in degree $n+1-m$. Moreover, it is zero for $m$ large enough, due to the fact that the OPE is a Laurent series. For $m$ very small, then $n+1-m$ is very large, and therefore $f(I^b_k\otimes 1, 1\otimes I^c_j)$ acts trivially eventually on smooth modules. We see that the action of $\Delta_z(I^a_n)$ on \(M \otimes N\) is a polynomial of $z$ and consequently is well-defined when evaluated at $z=s\in \C^\times$. 
    
\end{proof}

\begin{Rem}
    Unfortunately, \(\Delta_z\) \textbf{does not} define a meromorphic tensor structure on the category of modules of $ Y^\circ_\hbar(\fd)$. This is because we can only take the tensor products of smooth modules, but the results are not smooth anymore.
    
\end{Rem}

\subsection{Existence and uniqueness of a meromorphic twisting matrix}\label{subsec:MeroR}

We have seen that the algebra $ Y^\circ_\hbar(\fd)$ has two coproducts
\be
\begin{split}
&\Delta_{\hbar, z} = (\tau_z \otimes 1)\Delta_{\hbar}:  Y^\circ_\hbar(\fd)\to ( Y^\circ_\hbar(\fd)\otimes  Y^\circ_\hbar(\fd))[z];
\\&\Delta_z:  Y^\circ_\hbar(\fd)\to ( Y^\circ_\hbar(\fd)\otimes  Y^\circ_\hbar(\fd))\lpp z^{-1}\rpp.    
\end{split}
\ee
essentially defined by identifying $\overline{S}(\hbar\fg^*\lp\CO\rp)$ with $U(\fg_{<0})^*$ and $V_0(\fg)^*$ respectively. In this section, we show that there is a twisting matrix rotating these two coproducts into each other. More precisely, we show that there exists an element $R_s(z)$ such that
\be
R_s(z)^{-1}\Delta_{\hbar,z}(a) R_s(z)=\Delta_z(a)
\ee
holds for all \(a \in Y^\circ_\hbar(\fd)\). To do so, let us first collect some easy identities involving \(\CY\).

\begin{Lem}\label{lem:identitites_for_Y}
    For $A=\sum_{n = 0}^k c^a_n \cdot I_{a,-n-1}\in \fg_{<0}$ the following identities hold:
    \begin{enumerate}
        \item For any \(B \in U(\fg_{< 0})\) we have:
    \be
    \CY(A, z)B=e^{-zT}(A)B +\sum_{i,j=0}^\infty (-1)^i \binom{i+j}{i}\frac{c^a_iI_{a,j}}{z^{i+j+1}}\rhd B. 
    \ee
    Moreover, the first term can be calculated using:
    \be
    e^{-zT}(A)=\sum_{n = 0}^kc^a_n\sum_{m= n}^\infty\binom{m}{n} z^{m-n}I_{a,-m-1}. 
    \ee

    \item Writing $\CY(A,z)=A(z)_{\textnormal{reg}}+A(z)_{\textnormal{sing}}$ for the decomposition of \(\CY(A,z) = A(z)\) into its regular part $A(z)_{\textnormal{reg}}$ and its singular part $A(z)_{\textnormal{sing}}$ we have
    \be
:\CY(A^n,z):=:\CY(A,z)^n:=\sum_{m = 0}^n \binom{n}{m} A(z)_{\textnormal{reg}}^m A(z)_{\textnormal{sing}}^{n-m}.
    \ee
    for every \(n > 0\).
    \end{enumerate}
\end{Lem}

\begin{proof}
    By definition of $\CY$, we have:
    \be
\CY(A, z)=\sum_{n = 0}^k c^a_n \CY(I_{a,-n-1}, z)=\sum_{n = 0}^k c^a_n \frac{\pd_z^{n}}{n!} \CY(I_{a,-1}, z).
    \ee
The regular part of this is equal to
\be
\sum_{n = 0}^kc^a_n\sum_{m= n}^\infty\binom{m}{n} z^{m-n}I_{a,-m-1}B=\sum_{n = 0}^k c^a_ne^{-zT}(I_{a,-n})B,
\ee
where as the singular part is given by
\be
\sum(-1)^n c^a_n \sum_{m\geq 0} \binom{m+n}{m} \frac{I_{a,m}\rhd B}{z^{m+n+1}}.
\ee
This completes the proof of 1. 

Let us prove 2.\ using induction. Suppose the claim is true for $n-1$. Using
    \be
    \begin{split}
        :\CY(A,z)^n:\,&=\,:\CY (A,z) \CY (A,z)^{n-1}:
        \\&=A(z)_{\textnormal{reg}}:\CY (A,z)^{n-1}:+:\CY(A,z)^{n-1}:A(z)_{\textnormal{sing}}        
    \end{split}
    \ee
and the induction assumption, we see that
\be
\begin{split}
    &\sum_{m = 0}^{n-1} \binom{n-1}{m} X(z)_{\textnormal{reg}}^{m+1} X(z)_{\textnormal{sing}}^{n-1-m}+\sum_{m = 0}^{n-1} \binom{n-1}{m} X(z)_{\textnormal{reg}}^{m} X(z)_{\textnormal{sing}}^{n-m}\\&=\sum_{m = 0}^{n} \binom{n}{m} X(z)_{\textnormal{reg}}^m X(z)_{\textnormal{sing}}^{n-m},
\end{split}
\ee
by virtue of the binomial coefficient formulas. 
\end{proof}

Let us now consider the following elements:
\be\label{eq:rsingreg}
\begin{aligned}
    & r_{\textnormal{sing}}(z):=\hbar\sum_{i,j = 0}^\infty(-1)^i\binom{i+j}{i}\frac{I^a_i\otimes I_{a,j}}{z^{i+j+1}}\in ( Y^\circ_\hbar(\fd)\otimes_{\C\lbb\hbar\rbb} U(\fg[t])\lbb\hbar\rbb) \lbb z^{-1}\rbb,\\
    & r_{\textnormal{reg}}(z):=\hbar \sum_{j = 0}^\infty\sum_{i=0}^j \binom{j}{i} z^{j-i}I^a_i\otimes I_{a,-j-1}\in (Y_\hbar(\fd)\otimes_{\C\lbb\hbar\rbb} U(\fg_{<0})\lbb\hbar\rbb)\lbb z\rbb. 
\end{aligned}
\ee
Note that both of them can be written effectively as $\hbar \frac{I^a\otimes I_a}{t_2-z-t_1}$. More precisely, \(r_{\textnormal{sing}}\) is the expansion of this function in $z = \infty$ whereas \(r_{\textnormal{reg}}\) is the expansion in $z = 0$. Consider $e^{r_{\textnormal{reg}}(z)}, e^{r_{\textnormal{sing}}(z)}$, where the multiplication is taken in the corresponding algebras from \eqref{eq:rsingreg}. We can then naturally view both of these as elements in the vector space $(S(\fg^*\lbb t\rbb)\otimes U(\fg(\CK)))\lbb\hbar\rbb\lbb z, z^{-1}\rbb$. Therefore, both give rise to elements in
\be
\Hom_\C \lp V_0(\fg), U(\fg(\CK))\lbb z,z^{-1 }\rbb\rp,
\ee
such that $\lag r_{\textnormal{sing}}(z), I_{a,-n}\rag = I_{a,-n}(z)_{\textnormal{sing}}$ and $\lag r_{\textnormal{reg}}(z), I_{a,-n}\rag =  I_{a,-n}(z)_{\textnormal{reg}}$ for all \(n \ge 1\). Similarly, we can treat the intertwining map $\CY$ as an element in the above vector space. From the above lemma, we deduce:

\begin{Prop}\label{Prop:YR}
 Let $M$ be any smooth module of $\fg(\CK)$. Treat $\CY, e^{r_{\textnormal{reg}}(z)}$ and $e^{r_{\textnormal{sing}}(z)}$ as elements in the vector space:
 \be
(S(\fg^*\lbb t\rbb)\otimes \End(M))\lbb\hbar\rbb\lbb z,z^{-1}\rbb.
 \ee
 Then $e^{r_{\textnormal{reg}}(z)}$ and $e^{r_{\textnormal{sing}}(z)}$ have a well-defined multiplication in this space, and:
 \be\label{eq:Y_and_rmatrices}
\CY= e^{r_{\textnormal{reg}}(z)}\cdot e^{r_{\textnormal{sing}}(z)}.
 \ee
Moreover, the following identity in $(Y_\hbar(\fd) \otimes_{\C[\![\hbar]\!]} \End(M)[\![\hbar]\!])\lbb z^\pm\rbb$ holds
\be\label{eq:[xR]}
[\Delta_z(x),e^{r_{\textnormal{reg}}(z)} e^{r_{\textnormal{sing}}(z)}]=0,
\ee
for any \(x \in \fg[t]\), which is well-defined since $\Delta_z(x)$ is a polynomial in \(z\).
\end{Prop}

\begin{proof}

For each $k>0$ we can consider $e^{r_{\textnormal{reg}}(z)}$ and $e^{r_{\textnormal{sing}}(z)}$ as elements in:
\be
(S(\fg^*[t]/t^k\fg^*[t])\lbb\hbar\rbb\otimes_{\C\lbb\hbar\rbb} \End(M)\lbb\hbar\rbb)\lbb z,z^{-1}\rbb,
\ee
and $r_{\textnormal{sing}}(z)^n (1\otimes m)$ is a polynomial in $z^{-1}$ for all $m \in M$ and all $n>0$. Therefore \cite[Lemma 2.2.3]{frenkel2004vertex} applies and $e^{r_{\textnormal{reg}}}\cdot e^{r_{\textnormal{sing}}}(1\otimes m)$ is well-defined for any $k$ (since the multiplication is well-defined for any finite order of $\hbar$). Now taking a limit of $k$ we get the desired well-definiteness. 

To show the equality \eqref{eq:Y_and_rmatrices}, due to Lemma \ref{Lem:group}, we only need to evaluate both sides on elements of the form $A=e^{\hbar y}$ for \(y \in \fg_{<0}\). We have:
    \be
\lag e^{r_{\textnormal{sing}}(z)}, e^{\hbar y}\rag=\langle \sum_{n = 0}^\infty \frac{r_{\textnormal{sing}}(z)^n}{n!}, \sum_{n = 0}^\infty \frac{1}{n!} \hbar^ny^n\rangle=\sum_{n = 0}^\infty \frac{\hbar^n}{n!} \langle r_{\textnormal{sing}}(z), y\rangle^n=e^{\hbar y(z)_{\textnormal{sing}}},
\ee
and similarly \(\lag e^{r_{\textnormal{sing}}(z)}, e^{\hbar y}\rag = e^{\hbar y(z)_{\textnormal{reg}}}\). Consequently, 
    \be
e^{-zT}(e^{\hbar y})=e^{\hbar e^{-zT}(y)}=e^{\hbar y(z)_{\textnormal{reg}}}=\lag e^{r_{\textnormal{reg}}(z)}, e^{\hbar y}\rag.
    \ee
Moreover, the identity
\be
e^{\hbar y(z)_{\textnormal{reg}}}e^{\hbar y(z)_{\textnormal{sing}}} m=:\CY(e^{\hbar y}, z):m
\ee
holds for all \(m \in M\), due to Lemma \ref{lem:identitites_for_Y}.2. This completes the proof of \eqref{eq:Y_and_rmatrices}.

In order to prove \eqref{eq:[xR]}, note that for any $x\in \fg[t]$ and \(A \in V_0(\fg)\) we have:
\be
[x, \CY(A,z)]=\CY(e^{zT}(x)\rhd A, z).
\ee
The element $[1\otimes x, e^{r_{\textnormal{reg}}(z)} e^{r_{\textnormal{sing}}(z)}]$ is in $(S(\fg^*(\CO))\lbb\hbar\rbb\otimes_{\C\lbb\hbar\rbb} \End(M)\lbb\hbar\rbb)\lbb z,z^{-1}\rbb$ and can be identified with $[x, \CY(-, z)]$. Furthermore, the element $[e^{zT}(x)\otimes 1, e^{r_{\textnormal{reg}}(z)} e^{r_{\textnormal{sing}}(z)}]$ is identified with $-\CY(e^{zT}(x)\rhd -, z)$. This completes the proof.
\end{proof}

\begin{Rem}
    From the proof, we see that $e^{r_{\textnormal{reg}}(0)}=\CE_\gamma$ for Yang's $r$-matrix $\gamma$. Therefore, we have $e^{r_{\textnormal{reg}}(z)}=(\tau_z\otimes 1)\CE_\gamma$.  
    
\end{Rem}

Let us now define:
\be\label{eq:Rsmz}
\begin{aligned}
    &R_s(z):=e^{r_{\textnormal{sing}}(z)}\in ( Y^\circ_\hbar(\fd) \otimes_{\C\lbb\hbar \rbb}  Y^\circ_\hbar(\fd))\lbb z^{-1}\rbb;\\
    &\CE_\gamma(z):=(\tau_z\otimes 1) \CE_\gamma=e^{r_{\textnormal{reg}}(z)}\in ( Y^\circ_\hbar(\fd) \otimes_{\C\lbb\hbar \rbb} U(\fg_{<0})\lbb\hbar\rbb)\lbb z\rbb. 
\end{aligned}
\ee

\begin{Prop}
    The tensor series $R_s(z)$ satisfies \( (\tau_{z_1}\otimes \tau_{z_2})R_s(z_3)=R_s(z_3+z_1-z_2)\) and
    \be\label{eq:rotation_Deltahbar_Deltaz}
        R_s(z)^{-1}(\tau_z\otimes 1)\Delta_\hbar(a) R_s(z)=\Delta_z(a)
    \ee
    holds for all \(a \in Y^\circ_\hbar(\fd)\).
\end{Prop}

\begin{proof}
    The first identity about translation follows from the identity:
    \be
    (\tau_{z_1}\otimes \tau_{z_2}) \frac{I^a\otimes I_a}{t_2-z_3-t_1}=\frac{I^a\otimes I_a}{t_2+z_2-z_3-t_1-z_1},
    \ee
    where the RHS is expanded in the region where $|z_3|>|z_1|, |z_2|$. 
    
    To show \eqref{eq:rotation_Deltahbar_Deltaz}, first consider $f\in S(\fg^*[t])\lbb\hbar\rbb$. It suffices to evaluate both sides on $(e^{\hbar x}, e^{\hbar y})$ for \(x,y\in \fg_{<0}\). The right-hands side reads:
    \be
\lag \Delta_z(f), e^{\hbar x}\otimes e^{\hbar y}\rag=\lag f, \CY(e^{\hbar x}, z)e^{\hbar y}\rag=\lag f, :e^{\hbar x(z)}:e^{\hbar y}\rag,
    \ee
Writing $x=\sum c^a_i I_{a, -i-1}$ and using Lemma \ref{lem:identitites_for_Y}, this becomes
\be
\lag f, e^{\hbar e^{-zT}(x)} e^{\hbar \sum_{i,j\geq 0} (-1)^i \binom{i+j}{i}\frac{c^a_iI_{a,j}}{z^{i+j+1}}}e^{\hbar y}\rag.
\ee
On the other hand, if we write $\Delta_\hbar (f)=\sum_{(f)} f_{(1)}\otimes f_{(2)}$ and use 
\be
    R_s(z)^{-1} (-) R_s(z) = e^{-\textnormal{ad}(r_{\textnormal{sing}}(z))}(-),
\ee
the left-hand side in \eqref{eq:rotation_Deltahbar_Deltaz} takes the form:
\be
\sum \lag \tau_z (f_{(1)}), e^{\hbar x}\rag\otimes \lag f_{(2)}, \lag R_s(z), e^{\hbar x}\rag\rhd e^{\hbar y}\rag. 
\ee
Here, we evaluate $R_s(z)$ at $e^{\hbar x}$ and act on the second factor. By definition, this is equal to
\be
\lag f, e^{\hbar e^{-zT}(x)}\cdot \lag R_s(z), e^{\hbar x}\rag \rhd e^{\hbar y})=\lag f, e^{\hbar e^{-zT}(x)} e^{\hbar \sum_{i,j\geq 0} (-1)^i \binom{i+j}{i}\frac{c^a_iI_{a,j}}{z^{i+j+1}}}e^{\hbar y}\rag,
\ee
which beautifully coincides with the previously calculated right-hand side in \eqref{eq:rotation_Deltahbar_Deltaz}. 

For $x\in \fg[t]$, we would like to compute:
\be\label{eq:Rconj1}
e^{r_{\textnormal{sing}}(z)}\Delta_z(x)e^{-r_{\textnormal{sing}}(z)},
\ee
which using equation \eqref{eq:[xR]}, is equal to:
\be\label{eq:Rconj2}
e^{-r_{\textnormal{reg}}(z)}\Delta_z(x)e^{r_{\textnormal{reg}}(z)}. 
\ee
Note that we compute this inside $(Y_\hbar(\fd) \otimes_{\C[\![\hbar]\!]} \End(M)[\![\hbar]\!])\lbb z, z^{-1}\rbb$ for any smooth module \(M\) of \(\fg(\CK)\), since equation \eqref{eq:[xR]} holds true here. Equation \eqref{eq:Rconj1} is valued in the image of the algebra $(Y_\hbar (\fd)\otimes_{\C[\![\hbar]\!]} U(\fg(\CO))[\![\hbar]\!])\lpp z^{-1}\rpp$, whereas equation \eqref{eq:Rconj2} is valued in $(Y_\hbar(\fd) \otimes_{\C[\![\hbar]\!]} \End(M)[\![\hbar]\!])\lbb z\rbb$. Therefore, we conclude that both are in fact valued in the image of $(Y_\hbar (\fd)\otimes_{\C[\![\hbar]\!]} U(\fg(\CO))[\![\hbar]\!])[z]$.  

In this algebra, we can compute
\be
e^{-r_{\textnormal{reg}}(z)}\Delta_z(x)e^{r_{\textnormal{reg}}(z)}=(\tau_z\otimes 1)(\CE_\gamma^{-1}\Delta(x)\CE_\gamma)=(\tau_z\otimes 1)\Delta_{\hbar}(x),
\ee
thanks to the definition of $\Delta_\hbar$ in Theorem \ref{Prop:qHopf}.
\end{proof}

We can think of this $R_s(z)$ as giving an isomorphism:
\be
R_s(z)\colon (M\otimes N)\lpp z^{-1}\rpp \cong (\tau_z(M)\otimes N)\lpp z^{-1}\rpp,
\ee
where $M, N$ are two modules of $ Y^\circ_\hbar(\fd)$ and the tensor product on the left-hand side uses $\Delta_z$ while the one on the right-hand side uses $\Delta_{\hbar,z}$ to define the module structure. 

Of course, we have seen that any two modules of $ Y^\circ_\hbar(\fd)$ have another meromorphic (in fact, holomorphic) tensor product, namely using $\Delta_{\hbar, z}=(\tau_z\otimes 1)\Delta_\hbar$. This is weakly coassociative because of the coassociativity of $\Delta_\hbar$ and since $T$ is a coalgebra coderivation. The existence of $R_s(z)$ means that for any two objects, these two meromorphic tensor products are isomorphic. Namely, for any $M, N$, we have:
\be
M\otimes_{\Delta_z} N\cong M\otimes_{\Delta_{\hbar, z}} N.
\ee
We now prove the following statement, which is the meromorphic version of the ``monoidal structure axiom" for $R_s(z)$. 

\begin{Thm}\label{Thm:Rcocycle}
    The tensor series $R_s$ satisfies:
    \be\label{eq:Rcocycle}
        (\Delta_{z_1}\otimes 1)(R_s(z_2)^{-1})R_{s}^{12}(z_1)^{-1}=(1\otimes \Delta_{z_2})(R_s(z_1+z_2)^{-1}) R_{s}^{23}(z_2)^{-1}. 
    \ee

\end{Thm}
\begin{Rem}\label{rem:equality_of_series}
    We first explain in what sense \eqref{eq:Rcocycle} is even a well-defined equality. The LHS of equation \eqref{eq:Rcocycle} is naturally an element in the algebra
\be
(Y^\circ_\hbar(\fd) \otimes Y^\circ_\hbar(\fd) \otimes Y^\circ_\hbar(\fd))\lpp z_1^{-1}\rpp\lpp z_2^{-1} \rpp
\ee
where the RHS is naturally an element in the algebra
\be
(Y^\circ_\hbar(\fd) \otimes Y^\circ_\hbar(\fd) \otimes Y^\circ_\hbar(\fd)) \lpp z_2^{-1}\rpp \lpp (z_1+z_2)^{-1}\rpp .
\ee
Recall from Remark \ref{Rem:dualVA} that the tensor product vertex algebra $\CV = V_0(\fg)\otimes S(\fg^*_{<0})$  is dual to $Y_\hbar^\circ (\fd)$ with coproduct \(\Delta_z\). For every $A, B, C\in \CV$, we have
\be
\lag (\Delta_{z_1}\otimes 1)(R_s(z_2)^{-1})R_{s}^{12}(z_1)^{-1}, A\otimes B\otimes C\rag\in \C [z_2,z_2^{-1}]\lpp z_1\rpp\lbb\hbar\rbb
\ee
 where as
 \be
\lag (1\otimes \Delta_{z_2})(R_s(z_1+z_2)^{-1}) R_{s}^{23}(z_2)^{-1}, A\otimes B\otimes C\rag \in \C [z_1+z_2,(z_1+z_2)^{-1}]\lpp z_2\rpp\lbb\hbar\rbb.
 \ee
 The equality \eqref{eq:Rcocycle} simply means that these elements are the expansion of the same element from
 \be
\C [z_1, z_1^{-1}, z_2,z_2^{-1},(z_1+z_2)^{-1}] \lbb\hbar\rbb
 \ee
 into the respective series. In what follows, the equality of two formal series is always understood in this sense. 
\end{Rem}

The proof of Theorem \ref{Thm:Rcocycle} will be based on the following lemma, which is of independent utility.

\begin{Lem}\label{LemdeltazR}
    The identity \(
        (1\otimes \Delta_{z_2})R_s(z_1+z_2)= R_{s}^{12}(z_1)R_{s}^{13}(z_1+z_2)
    \) holds.
\end{Lem}

\begin{proof}
    We will prove:
    \be
        (1\otimes \Delta_{z_2})r_{\textnormal{sing}}(z_1+z_2)=r_{\textnormal{sing}, 12}(z_1)+ r_{\textnormal{sing}, 13}(z_1+z_2).
    \ee
    From this one derives the lemma since $r_{\textnormal{sing}, 12}$ and $r_{\textnormal{sing}, 13}$ commute with each other. We have:
    \be
(1\otimes \Delta_{z_2})\left(\frac{\sum_{a = 1}^nI^a\otimes I_a}{t_2-z_1-z_2-t_1}\right)=\frac{\sum_{a = 1}^nI^a\otimes 1\otimes I_a}{t_3-z_1-z_2-t_1}+\frac{\sum_{a = 1}^nI^a\otimes I_a\otimes 1}{(t_2+z_2)-z_1-z_2-t_1}
    \ee
    where the second term is $r_{\textnormal{sing}, 12}(z_1)$. This completes the proof. 
\end{proof}

\begin{proof}[Proof of Theorem \ref{Thm:Rcocycle}]
Using Lemma \ref{LemdeltazR}, the RHS of equation \eqref{eq:Rcocycle} is equal to
\be
R_{s}^{12}(z_1)^{-1}R_{s}^{13}(z_1+z_2)^{-1} R_s^{23}(z_2)^{-1},
\ee
whereas the LHS is
\be\label{eq:to_prove_for_Rs_cocycle}
 (\Delta_{z_1}\otimes 1)(R_s(z_2)^{-1})R_{s}^{12}(z_1)^{-1}=R_s^{12}(z_1)^{-1} (\Delta_{\hbar, z_1}\otimes 1)(R_s(z_2)^{-1})
\ee
by virtue of Equation \eqref{eq:rotation_Deltahbar_Deltaz}. This implies that it remains to prove the following equation:
\be\label{eq:Rsproof}
(\Delta_{\hbar, z_1}\otimes 1)(R_s(z_2))=R_s^{23}(z_2) R_{s}^{13}(z_1+z_2). 
\ee

Recall $R_s(z)$ and $\CE_\gamma(z)$ from equation \eqref{eq:Rsmz} and let $\CY (-, z)_{\textnormal{sing}}$ be the series $R_s(z)(-)$ and $\CY(-, z)_{\textnormal{reg}}$ be the series \(R_s(z)(-)\) and $\CE_\gamma (z)(-)$ respectively. The following is true, thanks to Proposition \ref{Prop:YR}:
\be
\CY (A, z)=\sum \CY (A^{(1)}, z)_{\textnormal{reg}}\CY (A^{(2)}, z)_{\textnormal{sing}}.
\ee

Let us pair both sides of \eqref{eq:to_prove_for_Rs_cocycle} with $e^{\hbar x}\otimes e^{\hbar y}$ for $x, y\in \fg_{<0}$. The LHS, as an element in $U(\fg)(\CK)$, reads:
\be
\CY \lp e^{\hbar x(z_1)_{\textnormal{reg}}}\cdot e^{\hbar y}, z_2\rp_{\textnormal{sing}},
\ee
where we used the same notation as in the proof of Proposition \ref{Prop:YR}. The RHS of \eqref{eq:to_prove_for_Rs_cocycle} gives
\be
\CY (e^{\hbar y}, z_2)_{\textnormal{sing}}\CY (e^{\hbar x}, z_1+z_2)_{\textnormal{sing}}.
\ee
Now multiply both sides with
\be
\CY \lp e^{\hbar x(z_1)_{\textnormal{reg}}}\cdot e^{\hbar y}, z_2\rp_{\textnormal{reg}}=e^{\hbar x(z_1+z_2)_{\textnormal{reg}}} e^{\hbar y (z_2)_{\textnormal{reg}}}
\ee
from the left. The LHS becomes
\be
\CY \lp e^{\hbar x(z_1)_{\textnormal{reg}}}\cdot e^{\hbar y}, z_2\rp,
\ee
which by definition of $\CY$ is precisely the normal-ordered product
\be
\norm{\CY (e^{\hbar x(z_1)_{\textnormal{reg}}}, z_2)\CY (e^{\hbar y}, z_2)}=\norm{\CY (e^{\hbar x}, z_1+z_2)\CY (e^{\hbar y}, z_2)}.
\ee
On the other hand, the RHS becomes
\be
\begin{split}
    &e^{\hbar x(z_1+z_2)_{\textnormal{reg}}}  e^{\hbar y(z_2)_{\textnormal{reg}}}\CY (e^{\hbar y}, z_2)_{\textnormal{sing}} \CY (e^{\hbar x}, z_1+z_2)_{\textnormal{sing}}\\&= e^{\hbar x(z_1+z_2)_{\textnormal{reg}}} \CY (e^{\hbar y}, z_2) \CY (e^{\hbar x}, z_1+z_2)_{\textnormal{sing}}.
\end{split}
\ee
Looking at the proof of Lemma \ref{lem:identitites_for_Y}.2, we find that this is exactly the normal-ordered product of $\CY (e^{\hbar x}, z_1+z_2)$ and $\CY (e^{\hbar y}, z_2)$. Since we multiplied the invertible element $e^{\hbar x(z_1+z_2)_{\textnormal{reg}}} e^{\hbar y (z_2)_{\textnormal{reg}}}$, we conclude equation \eqref{eq:Rsproof} holds, and therefore equation \eqref{eq:Rcocycle} holds. 
\end{proof}

We can in fact show that such an element $R_s(z)$ is unique, if we assume that it is of certain form. 

\begin{Prop}
    The element $R_s(z)$ is the unique element in 
    \be
    (S(\fg^*(\CO))\lbb\hbar\rbb\otimes_{\C\lbb\hbar\rbb} U(\fg(\CO))\lbb\hbar\rbb)\lbb z^{-1}\rbb
    \ee
    such that $R_s(z)^{-1}\Delta_{\hbar, z}(a) R_s(z)=\Delta_z(a)$ holds for all \(a \in Y^\circ_\hbar(\fd)\).
    
\end{Prop}

\begin{proof}
    Let $\wt{R}(z)$ be another such element, namely:
    \be
\wt{R}(z)^{-1}\Delta_{\hbar, z}(a) \wt{R}(z)=\Delta_z(a).
    \ee
holds for all \(a \in Y^\circ_\hbar(\fd)\). Let $a = f\in S(\fg^*[t])\lbb\hbar\rbb$. For any $X,Y\in \fg_{<0}$, we have
\be
\lag \wt{R}(z)^{-1}\Delta_{\hbar, z}(f) \wt{R}(z) ,e^{\hbar X}\otimes e^{\hbar Y}\rag=\lag f, e^{\hbar X(z)_{\textnormal{reg}}}\cdot \lag \wt{R}(z), e^{\hbar X}\rag\rhd e^{\hbar Y}\rag,
\ee
which implies that:
\be
 e^{\hbar X(z)_{\textnormal{reg}}}\cdot \lag \wt{R}(z), e^{\hbar X}\rag\rhd e^{\hbar Y}= e^{\hbar X(z)_{\textnormal{reg}}}\cdot \lag R_s(z), e^{\hbar X}\rag\rhd e^{\hbar Y}.
\ee
In other words
\be
\lag \wt{R}(z), e^{\hbar X}\rag \rhd e^{\hbar Y}= \lag R_s(z), e^{\hbar X}\rag\rhd e^{\hbar Y},
\ee
for any $X,Y$. In \cite{frenkel2007langlands}, it is stated that since our vacuum vertex algebra is of level $0\ne -h^\vee$, the map $U(\fg(\CK))\to \mathrm{End}(V_0(\fg))$ is injective. Therefore, the above implies that 
\be
    \lag R_s(z), e^{\hbar X}\rag=\lag \wt{R}(z), e^{\hbar X}\rag
\ee
for all $X \in \fg_{<0}$, and hence $R_s(z)=\wt{R}$. 
\end{proof}

\subsection{The quantum $R$-matrix}\label{subsec:fullR}
We now use $R_s(z)$ to construct the full $R$-matrix. The idea is that the existence of $R_s(z)$ means that:
\be
M\otimes_{\Delta_z} N\cong M\otimes_{\Delta_{\hbar, z}}N.
\ee
On the other hand, weak cocommutativity of $\Delta_z$ implies that:
\be
M\otimes_{\Delta_z} N\cong \tau_z \lp N\otimes_{\Delta_{-z}}M\rp.
\ee
Using $R$ again, we find
\be
M\otimes_{\Delta_{\hbar, z}}N= \tau_z \lp N\otimes_{\Delta_{\hbar, -z}}M\rp,
\ee
thus swapping the order of $M$ and $N$. Therefore, the full $R$-matrix can be obtained from the composition of the above isomorphisms. 

More concretely, note that since $\Delta_z$ satisfies \(
\Delta_{z}^{\textnormal{op}}=(\tau_z\otimes \tau_z) \Delta_{-z}\) we have
\be
((\tau_z\otimes 1)\Delta_\hbar(a))^{\textnormal{op}}=(R_s(z)\Delta_z(a)R_s(z)^{-1})^{\textnormal{op}}=R_{s}^{21}(z)(\tau_z\otimes \tau_z) \Delta_{-z}(a) R_{s}^{21}(z)^{-1}
\ee
for any \(a \in Y^\circ_\hbar(\fd)\). On the other hand:
\be
\begin{split}
(\tau_z\otimes \tau_z) \Delta_{-z}(a)&=(\tau_z\otimes \tau_z) (R_s(-z)^{-1} (\tau_{-z}\otimes 1)\Delta_\hbar(a) R_s(-z))\\&=R_s(-z)^{-1} (1\otimes \tau_z)\Delta_\hbar(a) R_s(-z). 
\end{split}
\ee
Here, we used the fact that $(T\otimes 1+1\otimes T)R_s(z)=0$. In summary:
\be
((\tau_z\otimes 1)\Delta_\hbar(a))^{\textnormal{op}}=R_{s}^{21}(z)R_s(-z)^{-1} (1\otimes \tau_{z})\Delta_\hbar(a) R_s(-z) R_{s}^{21}(z)^{-1}.
\ee
Let us define $R(z)=R_{s}^{21}(-z)R_s(z)^{-1}\in ( Y^\circ_\hbar(\fd)\otimes  Y^\circ_\hbar(\fd))\lpp z^{-1}\rpp$. Then the previous identity takes the simple form
\be
((\tau_z\otimes 1)\Delta_\hbar(a))^{\textnormal{op}}=R(-z)(1\otimes \tau_{z})\Delta_\hbar(a)  R(-z)^{-1},
\ee
Equivalently, we get the first identity of the following theorem.

\begin{Thm}\label{thm:Rmatrix_in_section}
    The $R$-matrix \(R(z) = R_s^{21}(-z)R_s(z)^{-1}\) defined above satisfies 
    the following identities:
    \begin{enumerate}
        \item \((\tau_z\otimes 1) \Delta_\hbar^{\textnormal{op}}(a)=R(z) (\tau_z\otimes 1) (\Delta_\hbar(a)) R(z)^{-1}\) for all \(a \in Y^\circ_\hbar(\fd)\);

        \item \((\Delta_{\hbar, z_1}\otimes 1)R(z_2)=R^{13}(z_1+z_2)R^{23}(z_2)\);

        \item \((1\otimes \Delta_{\hbar, z_2})R(z_1+z_2)=R^{13}(z_1+z_2)R^{12}(z_1)\);

        \item \((\tau_{z_1}\otimes \tau_{z_2})R(z_3)=R(z_3+z_1-z_2).\)
    \end{enumerate}
Here, the equalities again mean that two Laurent series are coming from the same function on different regions of convergence.  
In particular, $R$ is a solution to the quantum Yang-Baxter equation:
\be
R^{12}(z_1)R^{13}(z_1+z_2)R^{23}(z_2)=R^{23}(z_2)R^{13}(z_1+z_2)R^{12}(z_1).
\ee
\end{Thm}

\begin{proof}
The first identity is proven in the arguments before the theorem. The remaining identities 2.\&3.\ are deduced by repeating the proof of \cite[Theorem 7.1]{gautam2021meromorphic}, after replacing \(R^-\) with $R_s(z)^{-1}$ and $R^0$ with \(1\), using the fact that $R_s^{-1}$ satisfies the corresponding cocycle condition from Theorem \ref{Thm:Rcocycle} and the fact that \(R_s\) intertwines the coproducts $\Delta_z$ and $\Delta_{\hbar,z}= (\tau_z\otimes 1)\Delta_\hbar$ in the sense of \eqref{eq:rotation_Deltahbar_Deltaz}. 
\end{proof}

This finishes the construction of $R_s$ and $R$, and the proof of Theorem \ref{Thm:meroR} and Theorem \ref{Thm:fullR}.

\subsubsection{Twisting of the \(R\)-matrix}\label{sec:twisting_R_matrices}
We now turn to the consequence of combining the twisting procedure from Section \ref{sec:twisting} and the existence of the \(R\)-matrix. Let \(r\) be a generalized \(r\)-matrix with coefficients in \(\fg\) and \(\rho\) be the associated \(\epsilon\)-graded \(r\)-matrix with coefficients in \(\fd\). Assume that \(\rho\) depends on the difference \(t_1-t_2\) of its variables. Then \(\CA_\hbar(\fd,\rho)\) is a twist of \(Y_\hbar(\fd)\) via 
\begin{equation}
    F \coloneqq \CE_r^{-1} \CE_\gamma \in (S(\fg^*(\CO)) \otimes U(\fg(\CO)))[\![\hbar]\!].
\end{equation}
This can be used to adjust Theorem \ref{Thm:fullR} to \(\CA_\hbar(\fd,\rho)\). However, in doing so one has to be careful, since \(F(z) \coloneqq (\tau_z \otimes 1)F \in (\CA_\hbar(\fd,\rho) \otimes_{\C[\![\hbar]\!]} \CA_\hbar(\fd,\rho))[\![z]\!]\) in general and therefore multiplications of \(F(z)\) and \(R(z)\) are not well-defined.

There are two options to resolve this problem: first one can only consider the evaluation on smooth modules, where \(R(z)\) is truncated to a polynomial in \(z^{-1}\), or to consider rational \(\rho\) in which case each $\hbar$-order of $\tau_z\otimes 1 (F)$ is a polynomial in $z$ and the above multiplication still has a meaning. 

Let us consider the former option first. Consider two smooth modules \(M,N\) of \(Y_\hbar(\fd)\) which are topologically free \(\C[\hbar]\!]\), where we recall that \(Y_\hbar(\fd)\) is isomorphic to \(\CA_\hbar (\fd,\rho)\) as \(\C[\![\hbar]\!]\)-algebra. This means that \(M = U[\![\hbar]\!],N = V[\![\hbar]\!]\) for two vector spaces \(U,V\) and \(I_at^k, I^at^k\) act trivially for all \(a\in\{1,\dots,d\}\) and \(k > K\) for a sufficiently large \(K \in \N\). We have isomorphisms:
\begin{equation}
    F_{M,N}(z) \colon M \otimes_{\Delta_{\hbar,z}} N \to M \otimes_{\Delta_{\rho,\hbar,z}} N \textnormal{ and } R_{M,N}(z) \colon M \otimes_{\Delta_{\hbar,z}} N \to \tau_z(N \otimes_{\Delta_{\hbar,-z}} M).
\end{equation} 
which are defined by invertible elements 
\begin{equation}
    F_{M,N}(z) \in \End(U \otimes V)[\![z]\!][\![\hbar]\!] \textnormal{ and }R_{M,N}(z) \in \End(U \otimes V)[z^{-1}][\![\hbar]\!].
\end{equation}
Here, the fact that \(R_{M,N}(z)\) is polynomial in every \(\hbar\)-coefficient is crucial and follows immediately from the smoothness of \(M,N\) and the definition of \(R\). In particular, multiplications of \(R_{M,N}(z)\) and \(F_{M,N}(z)\) are well-defined in \(\End(U \otimes V)(\!(z)\!)[\![\hbar]\!]\).
Combined with \((\tau_z \otimes \tau_z)(F) = F\)
by virtue of Proposition \ref{Prop:TFT} and Theorem \ref{thm:Rmatrix_in_section} we obtain the following.

\begin{Prop}
    For any topologically free smooth modules \(M,N\) of \(\CA_\hbar(\fd,\rho)\) the expression 
        \begin{equation}
            R_{\rho;M,N}(z)\coloneqq F_{N,M}(-z) R_{M,N}(z) F_{M,N}(z)^{-1}    
        \end{equation}
        defines an isomorphism
        \begin{equation}
            M \otimes_{\Delta_{\rho,\hbar,z}} N \cong \tau_z(N \otimes_{\Delta_{\rho,\hbar,-z}} M).
        \end{equation}
    Furthermore,         \((\tau_{z_1}\otimes \tau_{z_2})R_{\rho;M,N}(z_3)=R_{\rho;M,N}(z_3+z_1-z_2)\) holds.
\end{Prop}

In order to make sense of the analog of the cocycle condition and Yang-Baxter equation in this setting, consider three finite topologically free smooth modules \(M_1,M_2,M_3\). In particular, \(M_i = V_i[\![\hbar]\!]\) for some finite-dimensional vector space \(V_i\) and every \(i \in \{1,2,3\}\). Then 
\be
\begin{split}
    &R_{\rho;M_1 \otimes_{\Delta_{\rho,\hbar,z_1}} M_2,M_3}(z_2) \in \End(V_1 \otimes V_2 \otimes V_3)[\![z_1]\!](\!(z_2)\!)[\![\hbar]\!] 
    \\&R_{\rho;M_1,M_3}(z_1+z_2)R_{\rho;M_2,M_3}(z_2) \in \End(V_1 \otimes V_2 \otimes V_3)(\!(z_1+z_2)\!)(\!(z_2)\!)[\![\hbar]\!].
\end{split}
\ee
where we identified \(M_1 \otimes_{\Delta_{\rho,\hbar,z_1}} M_2 \cong (V_1 \otimes V_2)[\![z_2]\!][\![\hbar]\!]\).
When stating that they are equal, as in the proof of equation \eqref{eq:Rcocycle}, we mean again that they are expansions of the same element in  
\be\label{eq:cocycle_large_space}
    \End(V_1 \otimes V_2 \otimes V_3)\lbb z_1, z_2\rbb [ z_1^{-1}, z_2^{-1}, (z_1+z_2)^{-1}] [\![\hbar]\!].
\ee
Similarly, we have
\begin{equation}
    \begin{split}
        &R_{\rho;M_1, M_2 \otimes_{\Delta_{\rho,\hbar,z_2}}M_3}(z_1+ z_2) \in \End(V_1 \otimes V_2 \otimes V_3)[\![z_2]\!](\!(z_1+z_2)\!)[\![\hbar]\!]
    \\&R_{\rho;M_1,M_3}(z_1+z_2)R_{\rho;M_1,M_2}(z_2) \in \End(V_1 \otimes V_2 \otimes V_3)(\!(z_2)\!)(\!(z_1+z_2)\!)[\![\hbar]\!]
    \end{split}
\end{equation}
and their equality means that they are expansions of the same element in \eqref{eq:cocycle_large_space}.
Finally, both \(R_\rho^{12}(z_1)R_\rho^{13}(z_1+z_2)R_\rho^{23}(z_2)\) and \(R_\rho^{23}(z_2)R_\rho^{13}(z_1+z_2)R_\rho^{12}(z_1)\) are elements of \eqref{eq:cocycle_large_space} as well.
Using Theorem \ref{thm:Rmatrix_in_section} and Proposition \ref{Prop:twistingF}, we obtain the following statement. 

\begin{Prop}
   For three finite topologically free smooth modules \(M_1,M_2,M_3\) we have 
        \begin{equation}
            \begin{split}
                &R_{\rho;M_1 \otimes_{\Delta_{\rho,\hbar,z_1}} M_2,M_3}(z_2) = R_{\rho;M_1,M_3}(z_1+z_2)R_{\rho;M_2,M_3}(z_2); 
                \\&R_{\rho;M_1, M_2 \otimes_{\Delta_{\rho,\hbar,z_2}}M_3}(z_1+ z_2) = R_{\rho;M_1,M_3}(z_1+z_2)R_{\rho;M_1,M_2}(z_2);
            \end{split}
        \end{equation}
        In particular, the following version of the quantum Yang-Baxter equation holds:
        \be
            \begin{split}
                &R_{\rho,M_1,M_2}(z_1)R_{\rho;M_1,M_3}(z_1+z_2)R_{\rho;M_2,M_3}(z_2)\\&=R_{\rho;M_2,M_3}(z_2)R_{\rho;M_1,M_3}(z_1+z_2)R_{\rho;M_1,M_2}(z_2).
            \end{split}
        \ee
\end{Prop}

Let us now turn to the second possibility to resolve the aforementioned problem. Therefore, we assume that \(\rho\) is rational for the reminder of this section.

\begin{Thm}\label{thm:twisted_Rmatrix}
If \(\rho\) is rational, the multiplication 
\be
    R_\rho(z) \coloneqq F^{21}(-z)R(z)F(z)^{-1}
\ee
is well-defined and has the following properties:
\begin{enumerate}
      
        \item \(
            (\tau_z\otimes 1)\Delta_{\rho,\hbar}^{\textnormal{op}}(a)=R_\rho(z)  (\tau_z \otimes 1)\Delta_{\rho,\hbar}(a) R_\rho(z)^{-1}\) for all \(a \in \CA_\hbar^\circ(\fd,\rho)\);
        
        \item \((\Delta_{\rho,\hbar, z_1}\otimes 1)R_\rho(z_2)=R_\rho^{13}(z_1+z_2)R^{23}_\rho(z_2)\), \((1\otimes \Delta_{\rho,\hbar, z_2})R_\rho(z_1+z_2)=R_\rho^{13}(z_1+z_2)R_\rho^{12}(z_2)\) and \(R_\rho(z)\) is a solution of quantum Yang-Baxter equation:
        \be
            R_\rho^{12}(z_1)R_\rho^{13}(z_1+z_2)R_\rho^{23}(z_2)=R_\rho^{23}(z_2)R_\rho^{13}(z_1+z_2)R_\rho^{12}(z_1);
        \ee

        \item \(R_\rho(z)\) satisfies $(\tau_{z_1}\otimes \tau_{z_2})R_\rho(z_3)=R_\rho(z_3+z_1-z_2)$;

\end{enumerate}
\end{Thm}

The well-definiteness of \(R_\rho(z)\) follows from the fact that for rational \(\rho\) we know that \(F \in (S(\fg^*[t]) \otimes U(\fg[t]))[\![\hbar]\!]\); see Proposition \ref{Prop:twistingF}.
The proof of Theorem \ref{thm:twisted_Rmatrix}.1.\ then follows immediately from Theorem \ref{Thm:fullR} and Proposition \ref{Prop:twistingF}. The proof of 
Theorem \ref{thm:twisted_Rmatrix}.2.\&3.\ is then the same as for Theorem \ref{Thm:fullR}.2.\&3.

Recall from the Section \ref{subsec:loopapply} that the coproduct \(\Delta_{\rho,\hbar,z}\) defines a product 
\begin{equation}
    \CY_{\rho,\hbar} \colon S(\fd(\rho))[\![\hbar]\!] \otimes_{\C[\![\hbar]\!]} S(\fd(\rho))[\![\hbar]\!] \to S(\fd(\rho))[\![\hbar]\!] [\![z]\!]   
\end{equation}
by identifying the continuous linear dual of \(\CA_\hbar(\fd,\hbar)\) with \(S(\fd(\rho))[\![\hbar]\!]\) as a \(\C[\![\hbar]\!]\)-module. Using the \(R\)-matrix \(R_\rho(z)\) for rational \(\rho\) we obtain the following result.

\begin{Prop}\label{prop:naiv_quantum_vertex}
    The tuple $(S(\fd(\rho))[\![\hbar]\!], \CY_{\rho,\hbar}, \Omega_\rho, T, R_\rho)$ defines a quantum vertex algebra that quantizes the quasi-classical commutative vertex algebra \(S(\fd(\rho))\) described in Section \ref{sec:quasi_classical_vertex}.
\end{Prop}

\section{Geometry of the equivariant affine Grassmannian and the Yangian}\label{sec:RavGrG}

In this section, we review the geometry of the equivariant affine Grassmannian $[G(\CO)\!\setminus\!\Gr_G]$. The main objective of this section is to motivate the construction of $Y_\hbar(\fd)$ in relation to this geometric object. In particular, we show how the coproducts $\Delta_\hbar$ and $\Delta_z$ naturally arise from such relations. We then formulate a conjecture relating the monoidal-factorization category $\Coh ([G(\CO)\!\setminus\!\wh{\Gr}_G])$ with modules of $Y_\hbar(\fd)$. 

This section is structured as follows. In Section \ref{subsec:equivariantgrassmannian}, we recall structures of $[G(\CO)\!\setminus\!\Gr_G]$ as well as its formal completion at identity $[G(\CO)\!\setminus\!\wh{\Gr}_G]$, and derive its relation with $Y_\hbar(\fd)$ (as well as general $\CA_\hbar(\fd, \rho)$) in Proposition \ref{Prop:abequiv}. In Section \ref{subsec:monoidal}, we recall the monoidal-factorization structure on $[G(\CO)\!\setminus\!\wh{\Gr}_G]$, and conjecture its relation to $Y_\hbar(\fd)$, of which we give some justification in Proposition \ref{Prop:dualVOA}.

\subsection{Recollection on the equivariant affine Grassmannian}\label{subsec:equivariantgrassmannian}

\subsubsection{The affine Grassmannian}

Let $G$ be a semisimple Lie group. The \textit{affine Grassmannian} associated to $G$, denoted by $\Gr_G$, can be defined as a quotient:
\be
\Gr_G:=G(\CK)/G(\CO).
\ee
Alternatively, it can be defined as the moduli space of $G$ bundles over $\mathbb{D}:=\mathrm{Spec}(\CO)$ with a trivialization over $\mathbb{D}^\times:=\mathrm{Spec}(\CK)$. It turns out that $\Gr_G$ is a classical ind-scheme; its geometry is well studied in the literature. For some details about this space, see \cite{zhu2016introduction}. The \textit{equivariant affine Grassmannian} is defined to be the moduli space of $G$ bundles on $\mathbb{B}:=\mathbb{D}\cup_{\mathbb{D}^\times}\mathbb{D}$, or in other words $\mathrm{Maps}(\mathbb{B}, BG)$. This $\mathbb{B}$ is the algebro-geometric definition of an infinitely small sphere, or a formal bubble. Alternatively, one can use the following realization:
\be
\mathrm{Maps}(\mathbb{B}, BG)\cong [G(\CO)\!\setminus \! \Gr_G]=[G(\CO)\! \setminus \! G(\CK)/G(\CO)],
\ee
where the right-hand side is the quotient stack of $\Gr_G$ by the left action of $G(\CO)$. As stated in introduction, interests in the study of the equivariant affine Grassmannian stem from its relation to the geometric Langlands correspondence and 4d $\CN=2$ holomorphic-topological theories. Following the physical predictions of \cite{kapustin2006holomorphic, kapustin2006wilson}, the authors of \cite{cautis2019cluster, cautis2023canonical} used the category of coherent sheaves on this space and its variation to define the category of line operators for the HT twist of 4d $\CN=2$ gauge theories. In \cite{niu2022local}, the second author used this category to realize Schur indices as characters of derived endomorphisms between objects. 

In the rest of this section, we will focus on $[G(\CO)\!\setminus \!\Gr_G]$ and its category of coherent sheaves $\Coh_{G(\CO)}(\Gr_G)$. The space $\Gr_G$ is an example of an ind-scheme, and therefore it is not straight-forward to define the category of coherent sheaves. Fortunately, in the last decade there has been much progress in algebraic geometry in extending the theory of coherent sheaves from finite type schemes and stacks to infinite-type. The structure of a DG ind-scheme and its category of sheaves are studied in \cite{gaitsgory2014dg, gaitsgory2019study, gaitsgory2017study} for locally almost finite type, and \cite{raskin2020homological} in general. The strong machinery in these works applies to $[G(\CO)\!\setminus \!\Gr_G]$, making it possible to access the category of sheaves. 

A reasonable DG ind-scheme, as defined in \cite[Definition 6.8.1]{raskin2020homological}, is a convergent prestack $X=\varinjlim X_i$ such that each $X_i$ is a quasi-compact, quasi-separated and eventually coconnective scheme, and that $X_i\to X_{j}$ are  almost finitely-presented closed embeddings. Let $H$ be a classical affine group scheme that acts on $X$. Then the quotient stack $[X/H]$ is called a weakly renormalizable pre-stack following \cite[Definition 6.28.1]{raskin2020homological}, and one can define the category $\mathrm{IndCoh}^*([X/H])$ via a right Kan extension:
\begin{equation}
\IndCoh^*([X/H]):= \lim\limits_{f: S\to [X/H] \text{ flat}} \IndCoh^*(S),
\end{equation}
where the limit is taken over all reasonable DG ind-schemes flat over $[X/H]$, using the functoriality of $f^{*,\IndCoh}$. For each $i$, one can similarly define $\IndCoh^*([X_i/H])$, and as discussed in \cite[Section 4.1]{niu2022local}, one can show that:
\be
\IndCoh^*([X/H])=\varinjlim\limits_{i} \IndCoh^*([X_i/H]).
\ee
The subcategory $\Coh ([X/H])$ is defined to be the image of $\Coh ([X_i/H])$ in the above ind-completion. 

It is well-known that $\Gr_G$ has a $G(\CO)$-equivariant stratification:
\be
\Gr_G=\varinjlim_i \Gr_{G, i},
\ee
which makes $\Gr_G$ into a reasonable ind-scheme. Applying the sheaf theory of \cite{raskin2020homological} to $X=\Gr_G$ and $H=G(\CO)$, we obtain the category $\Coh( [G(\CO)\!\setminus\!\Gr_G])$. The definition of this seems abstract, but on each strata $\Gr_{G, i}$, the category $\Coh([G(\CO)\!\setminus\!\Gr_{G, i}])$ is the derived category of $G(\CO)$-equivariant coherent sheaves on $\Gr_{G, i}$, and that an object in $\Coh( [G(\CO)\!\setminus\!\Gr_G])$ is just the image of an object in $\Coh([G(\CO)\!\setminus\!\Gr_{G, i}])$ under push-forward. 

\subsubsection{Formal completion of the equivariant affine Grassmannian}

To make contact with $Y_\hbar(\fd)$, we use a formal completion process. Let $e\in \Gr_G$ be the trivial coset, which is clearly a $G(\CO)$-equivariant subscheme. In \cite{gaitsgory2014dg}, the authors introduced formal completions in derived algebraic geometry, which was used in \cite{niu2022local} to compute derived endomorphism of identity object. Already showing up there is the realization that although the geometry of $\Gr_G$ is difficult, its formal completion around $e$ is much easier. 

Let $\mathcal{X}$ be a prestack; then its de-Rham stack is defined by:
\begin{equation}
\mathcal{X}_{\textnormal{dR}}(S)=\mathcal{X}(S_{\textnormal{red}}),
\end{equation} 
where $S_{\textnormal{red}}$ is the reduced scheme of $S$. Given a morphism of prestacks $\mathcal{X}\to \mathcal{Y}$, the formal completion is defined by(\cite[Section 6.1]{gaitsgory2014dg}):
\begin{equation}
\widehat{\mathcal{Y}_{\mathcal X}}:=\mathcal{Y}\times_{\mathcal{Y}_{\textnormal{dR}}}\mathcal{X}_{\textnormal{dR}}.
\end{equation}
This operation respects filtered colimits as explained in \cite[6.1.3]{gaitsgory2014dg}: if $\mathcal{X}=\varinjlim \mathcal{X}_n$ and $\mathcal{Y}=\varinjlim \mathcal{Y}_n$ such that the map $\mathcal{X}\to \mathcal{Y}$ comes from a system of maps $\mathcal{X}_n\to \mathcal{Y}_n$, then: 
\begin{equation}
\widehat{\mathcal{Y}_{\mathcal{X}}}=\varinjlim\widehat{\mathcal{Y}_{n, \mathcal{X}_n}}.
\end{equation}

Now assume that $X$ is a locally almost finite type DG scheme, $Y$ an almost finite type DG ind-scheme and $i \colon X\to Y$ is an embedding, then by \cite[Proposition 6.3.1]{gaitsgory2014dg}, $\widehat{Y_X}$ is a DG ind-scheme. Moreover, from the above we see that:
\begin{equation}
\widehat{Y_X}\cong \varinjlim \widehat{Y_{n, X}},
\end{equation}
which in particular means that:
\begin{equation}\label{eq:sheafformal}
\IndCoh(\widehat{Y_X})\cong \varinjlim \IndCoh (\widehat{Y_{n, X}}).
\end{equation}

Applying this to $X=[e]$ and $Y=\Gr_G$, we obtain the formal completion $\wh{\Gr}_{G, [e]}$. It turns out to that one can give $\wh{\Gr}_{G, [e]}$ the structure of a formal group over $[e]$ in the sense of \cite[Chapter 7]{gaitsgory2019study}. Let $G_{<0}:=G(t^{-1}\C[t^{-1}])$, then it is known that $G_{<0}$ is an open subset of $\Gr_G$ containing $e$, and therefore:
\be
\wh{\Gr}_{G, [e]}\cong \wh{G}_{<0, e}.
\ee
Another amazing result of \cite[Chapter 7]{gaitsgory2019study}, generalizing the work of \cite{lurie2011derived}, is that there is an equivalence between formal groups and Lie algebra objects over stacks in derived geometry. In our situation, the formal group is over a point, and so the Lie algebra is a genuine Lie algebra in the category of vector spaces. It is clear from the definition of $G_{<0}$ that:
\be
\mathrm{Lie}(\wh{G}_{<0, [e]})=\fg_{<0}.
\ee
This is an identification as an injective limit of Lie algebras. In our situation, where the Lie algebra is a classical one, one can identify $\wh{G}_{<0, e}$ with $\wh{\fg}_{<0}$, the formal completion of the ind-vector space $\fg_{<0}$ at $0$, such that the formal group law is given by Baker-Campbell-Hausdorff formula; see \cite[Construction 2.2.13.]{lurie2011derived}. In other words, let $\omega \coloneqq \omega_{\wh G_{<0, e}}$ be the dualizing sheaf of the formal group $\wh G_{<0, e}$ (whose existence is guaranteed by \cite{gaitsgory2017study}), then the global sections of $\omega$ form precisely the universal enveloping algebra of $\fg_{<0}$, as a Hopf algebra:
\be
\Gamma(\wh G_{<0, e},\omega)=U(\fg_{<0}).
\ee
Therefore, one can represent the formal group using the dual $U(\fg_{<0})^*$, which is precisely the filtered dual used in Section \ref{sec:Takiff}. In particular, one can represent $\wh{G}_{<0, e}$ by an injective limit of affine schemes:
\be
\C[\wh{G}_{<0, e}]=\varprojlim\limits_{k, n} S\lp \fg^*(\CO/(t^k))\rp/I_k^n.
\ee
Here $I_k$ is the ideal of $S\lp \fg^*(\CO/(t^k))\rp$ generated by $\fg^*(\CO/(t^k))$. The action of $G(\CO)$ on $\wh{G}_{<0, e}$ translates to the action of $\fg(\CO)$ on $S\lp \fg^*(\CO)\rp$ of \eqref{eq:defining_hopf_action}, and obviously each quotient is a $G(\CO)$-equivariant ring. We find, therefore, using equation \eqref{eq:sheafformal}, that we have an equivalence of categories:
\be
\Coh_{G(\CO)}\lp \wh{\Gr}_{G, e}\rp\simeq \varinjlim\limits_{k,n} S\lp \fg^*(\CO/(t^k))\rp/I_k^n\Mod^{G(\CO)}, 
\ee
where the right-hand side is the category of finite-dimensional $G(\CO)$-equivariant modules of the algebra $S\lp \fg^*(\CO/(t^k))\rp/I_k^n$. 

Let us now denote by $\CA_1(\fd,\rho)\Mod_G$ the category of finite-dimensional smooth modules of $\CA_1(\fd,\rho)$ ($\CA_\hbar(\fd,\rho)$ evaluated at $\hbar=1$), where the action of $\fg$ integrates to an algebraic action of $G$. Since for various $\rho$ the algebras $\CA_1(\fd, \rho)$ are canonically isomorphic to the smashed product of $U(\fg(\CO))$ with $\overline{S}(\fg^*(\CO))$, the above considerations lead to the following result.

\begin{Prop}\label{Prop:abequiv}
    There is an equivalence of abelian categories:
    \be
\Coh_{G(\CO)}\lp \wh{\Gr}_{G, e}\rp^{\heartsuit}\simeq \CA_1(\fd, \rho)\Mod_G.
    \ee
    
\end{Prop}

\begin{Rem}
    The category $\Coh_{G(\CO)}\lp \wh{\Gr}_{G, e}\rp$ is smaller than $\Coh_{G(\CO)}\lp \Gr_{G}\rp$. There is a canonical morphism $j: \wh{\Gr}_{G, e}\to \Gr_G$ such that one obtains a fully-faithful functor
    \be
j_*: \Coh_{G(\CO)}\lp \wh\Gr_{G, e}\rp\longrightarrow \Coh_{G(\CO)}\lp \Gr_{G}\rp.
    \ee
One can use this functor to identify $\Coh_{G(\CO)}\lp \wh\Gr_{G, e}\rp$ with the full-subcategory of $ \Coh_{G(\CO)}\lp \Gr_{G}\rp$ whose pullback to the complement of $e$ is trivial. 

\end{Rem}

\subsection{Monoidal factorization structure and $R$-matrix}\label{subsec:monoidal}

\subsubsection{Monoidal structure from correspondences}

As we have already seen from Section \ref{subsec:quantHopf}, when there exists group $G$ and $H\subset G$ a subgroup, the stack $[H\!\setminus\!G/H]$ enjoys the following correspondence as in equation \eqref{eq:corrGK}:
\be
\btik
& (G\times_H G)/H \arrow[dr]\arrow[dl]&\\
G/H\times G/H & & G/H
\etik
\ee
which induces a monoidal structure on the category $\Coh ([H\!\setminus\!G/H])$. The machinery of derived geometry makes this applicable to $[G(\CO)\!\setminus\!G(\CK)/G(\CO)]$, thanks to the fact that $\Gr_G$ is locally almost of finite type. It is, however, very difficult to compute the monoidal tensor product, since the operation involves taking global sections over closed sub-varieties of $\Gr_G$, which are themselves bundles over flag varieties. Our main objective in this section is to show that the computation simplifies significantly when one restricts to $\wh{\Gr}_G$. 

\begin{Prop}
    There is a monoidal structure on $\Coh ([G(\CO)\!\setminus\!\wh{\Gr}_G])$ such that $j_*$ is a monoidal functor. Moreover, this monoidal structure is classical, in the sense that it is exact on the heart of the ordinary $t$-structure. 
\end{Prop}

\begin{proof}
    Most of these statements are pretty straightforward. The monoidal structure comes from identifying $\wh{\Gr}_G$ with $\wh{G(\CK)}_{G(\CO)}/G(\CO)$, where $\wh{G(\CK)}_{G(\CO)}$ is the formal completion of $G(\CK)$ over $G(\CO)$. The fact that  $j_*$ is monoidal follows from the base change diagram:
    {\footnotesize\[
\btik
& (\wh{G(\CK)}_{G(\CO)}\times_{G(\CO)} \wh{G(\CK)}_{G(\CO)})/G(\CO)\arrow[dd, "j"] \arrow[dr]\arrow[dl]&\\
\wh{G(\CK)}_{G(\CO)}/G(\CO)\times \wh{G(\CK)}_{G(\CO)}/G(\CO) \arrow[dd, "j"]& & \wh{G(\CK)}_{G(\CO)}/G(\CO)\arrow[dd, "j"]\\
& (G(\CK)\times_{G(\CO)} G(\CK))/G(\CO) \arrow[dr]\arrow[dl]&\\
G(\CK)/G(\CO)\times G(\CK)/G(\CO) & & G(\CK)/G(\CO)
\etik
    \]}
To prove the statement about the monoidal structure being classical, we use that $\wh{G(\CK)}_{G(\CO)}$ has a splitting:
\be
\wh{G(\CK)}_{G(\CO)}\cong  \wh{G}_{<0}\times G(\CO),
\ee
and we can re-write the correspondence diagram as:
\be\label{eq:corresGr}
\btik
& (\wh{G}_{<0}\times G(\CO))\times_{G(\CO)}\wh{G}_{<0} \arrow[dr, "m"]\arrow[dl, "p"]&\\
\wh{G}_{<0}\times \wh{G}_{<0} & & \wh{G}_{<0}
\etik
\ee
Here the map $m$ is multiplication of $\wh{G}_{<0}$, which is classical and exact since $\wh{G}_{<0}$ is a formal completion of a classical ind-affine group scheme, and $p^*$ is clearly classical and exact. This completes the proof. 

\end{proof}

We have seen from Proposition \ref{Prop:abequiv} that the category $\Coh_{G(\CO)}(\wh{\Gr}_G)^{\heartsuit}$ is represented by the category of modules of $\CA_1(\fd, \rho)$. We show that this is a monoidal equivalence.

\begin{Prop}\label{Prop:E1equiv}
    The equivalence of Proposition \ref{Prop:abequiv} is an equivalence of monoidal categories. 
    
\end{Prop}

\begin{proof}
    Due to the fact that various $\CA_1(\fd, \rho)$ are twisted equivalent to each other, we only need to prove this for $Y_1(\fd)$. Clearly in equation \eqref{eq:corresGr} the morphism $m$ is given by multiplication law of the formal group $\wh{G}_{<0}$, which is the linear dual of the multiplication of the universal enveloping algebra of $\fg_{<0}$, therefore the functor $m_*$ is identified with the coproduct $\Delta_1$ of \(Y_1(\fd)\) on $\overline{S}(\fg^*(\CO))$. We need to identify the pull-back $p^*$ with the coproduct of $\Delta_1$ restricted to $U(\fg(\CO))$. 

    Let $M, N$ be two objects in $\Coh_{G(\CO)}(\wh{\Gr}_G)^{\heartsuit}$. The pullback $p^*(M\boxtimes N)$ is the $G(\CO)\times G(\CO)$-equivariant sheaf $M\boxtimes \C[G(\CO)]\boxtimes N$ on $(\wh{G}_{<0}\times G(\CO))\times \wh{G}_{<0}$, where the second factor of $G(\CO)$ acts on both $\C[G(\CO)]$ and $N$. One identifies this with a $G(\CO)$-equivariant sheaf on $\wh{G}_{<0}\times \wh{G}_{<0}$ by taking invariants with respect to the second copy of $G(\CO)$.

    Before taking this invariants, let us note that the action of the first copy of $G(\CO)$ is already interesting, and is not simply given by the usual action on the tensor product of $M$ and $\C[G(\CO)]$. For this, if we let $M=\overline{S}(\fg^*(\CO))$, then this action is given by left multiplication by $G(\CO)$, which is given by:
    \be
g\cdot (g_1, g_2)=(g\rhd g_1, (g\lhd g_1)\cdot g_2).
    \ee
Therefore the action of $G(\CO)$ (or better to say the Lie algebra $\fg(\CO)$) on $\overline{S}(\fg^*(\CO))\otimes \C[G(\CO)]$ is given precisely by the formula $\Delta_1$. Since any equivariant coherent sheaf is a quotient of $\overline{S}(\fg^*(\CO))\otimes V$ for some $G(\CO)$-module $V$, we conclude that the left action of $G(\CO)$ on $M\boxtimes \C[G(\CO)]$ coincides with that given by the coproduct $\Delta_1$. 

Taking invariants on the right, one identify $(\C[G(\CO)]\otimes N)^{G(\CO)}$ with $N$ via the natural map
\be
N\to \C[G(\CO)]\otimes N,
\ee
which is just the matrix coefficients of $N$. Since the first copy of $G(\CO)$ acts non-trivially on $\C[G(\CO)]$, it acts non-trivially on $N$, and this action is precisely the action via the second factor of $\Delta_1$. We conclude that if we view $\lp M\boxtimes \C[G(\CO)]\boxtimes N\rp^{G(\CO)}$ as a $G(\CO)$-equivariant sheaf on $\wh{G}_{<0}\times \wh{G}_{<0}$, then the $G(\CO)$-action coincides with the one determined by the coproduct $\Delta_1$. 

To conclude the statement, we must show that the associativity isomorphisms are equal as well. This amounts to showing that the associativity on the coherent sheaf category is trivial upon identifying $\wh\Gr_G\cong \wh G_{<0}$. We only give an outline of the proof here and interested readers may fill in the details. Let $C=[G(\CO)\!\setminus\! \wh G(\CK)\times_{G(\CO)}\wh\Gr_G]$ and $X=[G(\CO)\!\setminus\!\wh\Gr_G]$. The monoidal structure on $\Coh(X)$ is defined by the base-change diagram:
\be
\btik
& C\arrow[dl, swap, "p"] \arrow[dr, "m"]&\\
X\times X & & X
\etik.
\ee
Consider the following base-change diagram, which computes the monoidal product $M*(N*P)$:
\be
\btik
 & X\times C \arrow[dl, swap, "1\times p"] \arrow[dr, "1\times m"]& & C \arrow[dl, swap, "p"] \arrow[dr, "m"]& & \\
X^{\times 3} & & X^{\times 2} & & X.
\etik
\ee
We add to it the base-change diagram:
\be
\btik
 & & \wt{C}\arrow[dl, swap, "\wt{p}"]\arrow[dr, "\wt{m}"] & \\
 & X\times C \arrow[dl, swap, "1\times p"] \arrow[dr, "1\times m"]& & C \arrow[dl, swap, "p"] \arrow[dr, "m"]& & \\
X^{\times 3} & & X^{\times 2} & & X
\etik.
\ee
Here, $\wt{C}=G(\CO)\!\setminus\! (\wh G(\CK)\times_{G(\CO)}\wh G(\CK)\times_{G(\CO)}\wh G(\CK))/G(\CO)$ and the maps $\wt{p}, \wt{m}$ are the evident maps. By base-change, we have a natural isomorphism:
\be\label{eq:natural}
p^* (1\times m)_*\longrightarrow \wt{m}_*\wt{p}^*.
\ee
This induces a natural isomorphism:
\be\label{eq:naturaltrans}
m_*p^* (1\times m)_*(1\times p)^*\longrightarrow m_*\wt{m}_*\wt{p}^*(1\times p)^*.
\ee
Note that the monoidal product $M*(N*P)$ is defined by the functor on the left of equation \eqref{eq:naturaltrans}. On the other hand, the following is also a base-change diagram
\be
\btik
 & & \wt{C}\arrow[dl, swap, "\wt{p}'"]\arrow[dr, "\wt{m}'"] & \\
 & C\times X \arrow[dl, swap, "p\times 1"] \arrow[dr, "m\times 1"]& & C \arrow[dl, swap, "p"] \arrow[dr, "m"]& & \\
X^{\times 3} & & X^{\times 2} & & X
\etik
\ee
which gives rise to another natural isomorphism:
\be\label{eq:naturaltrans2}
m_*p^* (m\times 1)_*(p\times 1)^*\longrightarrow m_*\wt{m}'_*(\wt{p}')^*(p\times 1)^*.
\ee
Note that the monoidal product $(M* N)*P$ is computed by the functor on the left of equation \eqref{eq:naturaltrans2}. Now we have that $m_*\wt{m}'_*(\wt{p}')^*(p\times 1)^*=m_*\wt{m}_*\wt{p}^*(1\times p)^*$ since $m\circ \wt{m}'=m\circ \wt{m}$ and $(p\times 1)\circ \wt{p}'=(1\times p)\circ \wt{p}$. The associativity is then a combination of the two natural isomorphisms in equation \eqref{eq:naturaltrans} and equation \eqref{eq:naturaltrans2}. 

To show that this associativity isomorphism is trivialized by identifying $\wh\Gr_G\cong \wh G_{<0}$, we only need to show that the natural transformation in equation \eqref{eq:natural} is trivialized. Rewriting this base-change diagram using the above isomorphism, we obtain the following:
\be
\btik
G(\CO)\!\setminus\!(\wh G_{<0})^{\times 3}\rar{\wt{m}}\dar{\wt{p}} & G(\CO)\!\setminus\!\wh G_{<0}^{\times 2}\dar{p}\\
G(\CO)\!\setminus\!\wh G_{<0}\times G(\CO)\!\setminus\!\wh G_{<0}^{\times 2}\rar{1\times m} & (G(\CO)\!\setminus\!\wh G_{<0})^{\times 2}
\etik
\ee
Up to $G(\CO)\times G(\CO)$-equivariance, it is simply a base-change of classical ind-affine varieties, and therefore the base-change isomorphism is trivial. This completes the proof.

\end{proof}

\subsubsection{Monoidal factorization structure}

The equivariant affine Grassmannian $[G(\CO)\!\setminus\!\Gr_G]$ also has a classical factorization structure \cite{latyntsev2023factorisation}. We will not go into detail about the definition of factorization algebras, but will give a quick heuristic review. A \textit{factorization space} $X$ (as in \cite[Definition 3.2.1]{kapranov2004vertex}) over a smooth complex curve $C$ is a collection of spaces $X_{C^I}$ over $C^I$ where $I$ is any finite set, such that:

\begin{enumerate}
    \item $X_\emptyset=X$, and $X_{C^I}$ are formally smooth over $C^I$. 

    \item For every $J\twoheadrightarrow I$, there are two isomorphisms:
    \be
\Delta^{J/I, *} X_{C^J} \cong X_{C^I},\qquad j^{J/I, *} \prod_i X_{C^{J_i}}\cong j^{J/I, *} X_{C^J}.
    \ee
    Here $\Delta$ is the embedding of diagonals, and $j$ is the embedding of complements of diagonals.

    \item The above isomorphisms are compatible under composition $K\twoheadrightarrow J\twoheadrightarrow I$. 
\end{enumerate}

This is a scheme/stack analogue of the definition of factorization algebras due to \cite{beilinson2004chiral}. Given such a factorization space, one can apply any linearization functor to obtain a factorization algebra, which is a sheaf-theoretic version of a vertex algebra. Moreover, when restricting this factorization algebra over a formal disk $\mathbb{D}$, one obtains a vertex algebra.  

One of the first examples of a factorization space is the affine Grassmannian $\Gr_G$. The space $\Gr_{G, C^I}$ is defined roughly by \cite{beilinson2004chiral} as:
\be\label{eq:factgr}
\Gr_{G, C^I}=\left\{(c^I, P, \gamma)\, \middle|\,\begin{aligned}
     &{}\hspace{2cm} c^I\in C^I, \\&  P \text{ is a principle } G \text{ bunlde over } X, \\&\gamma \text{ trivialization of } P \text{ away from } c^I
\end{aligned}\right\}\Bigg /\text{equivalences}.
\ee
This factorization space is related to the affine Kac-Moody algebra in the following way. Let $\omega_{\Gr_{G, C^I}}$ be the dualizing sheaf of $\Gr_{G, C^I}$, and let $L_{C^I}$ be the factorization line bundle corresponding to the  determinant line bundle of $\Gr_{G, C^I}$. Then $L_k:=\pi_*(\omega\otimes L^{\otimes k})$ is a factorization algebra over $C$, where $\pi: \Gr_{G, C^I}\to C^I$ is the natural projection. The following well-known result can be derived from the Borel-Weil-Bott theorem for $\Gr_G$, proven in \cite{kumar2012kac, mathieu1988formules}.

\begin{Thm}
    When $C=\mathbb{D}$ and $k\geq 0$, the factorization algebra $L_k$ over $C$ can be identified with the simple quotient of the affine vertex algebra $V_k(\fg)$. 
    
\end{Thm}

Let $j: \wh{\Gr}_{G, C^I}\to \Gr_{G, C^I}$ be the embedding of the formal completion at $[e]$, which is a factorization sub-scheme. Let $V_k:=\pi_*(\omega\otimes j^*(L^{\otimes k}))$, then it is in fact easier to show the following statement; see for instance \cite[Proposition 20.4.3]{frenkel2004vertex}.

\begin{Thm}
    When $C=\mathbb{D}$, the factorization algebra $V_k$ over $C$ can be identified with the universal Kac-Moody vertex algebra $V_k(\fg)$. 
\end{Thm}

The factorization structure in equation \eqref{eq:factgr} is compatible with the action of $G(\CO)_{C^I}$, whose factorization structure is defined similarly. Therefore, one obtains a factorization structure on $[G(\CO)\!\setminus\!\Gr_G]$. This factorization structure is compatible with the monoidal structure in the sense that the correspondence diagram is a diagram of factorization spaces. In particular, the category $\Coh_{G(\CO)} (\Gr_G)$ is a monoidal-factorization category. This monoidal factorization structure was used in \cite{cautis2019cluster} to obtain a renormalized r-matrix, via a generalization of the Eckmann-Hilton argument. We will recall their argument here, but again only on a heuristic level. The monoidal-factorization structure of $\CC:=\Coh_{G(\CO)} (\Gr_G)$ means that there exists two ``multiplications":
\be
\otimes: \CC\boxtimes \CC\to \CC,\qquad \otimes_{z,w}:\CC_z\boxtimes \CC_w\to \CC_{z, w}, 
\ee
such that there exists a natural commutative diagram:
\be
\btik
\CC_z\boxtimes \CC_z\boxtimes \CC_w\boxtimes \CC_w\arrow[rr, "{\otimes^{1,2}\boxtimes\otimes^{3,4}}"]
\dar{\otimes_{z,w}^{1,3}\boxtimes \otimes_{z,w}^{2,4}} && \CC_z\boxtimes \CC_w\dar{\otimes_{z,w}}\\
\CC_{z,w}\boxtimes \CC_{z,w}\arrow[rr, "\otimes"] && \CC_{z, w}
\etik.
\ee
Moreover, this category has a unit object $\mathbbm{1}$ for both monoidal structures, which is the structure sheaf of the identity coset. In particular, for any object $M\in \CC_X$, the object $M_z\boxtimes \mathbbm{1}_w$ has a well-defined extension from $\CC_{C^2\setminus \Delta}$ to $\CC_{C^2}$, whose restriction to diagoanl is simply $M$. This object is denoted by $\eta_1(M, \mathbbm{1})$, and similarly $\eta_2(\mathbbm{1}, M)$. Now for two objects $M, N$, we have:
\be
\begin{aligned}
& (M\otimes N)_w\cong \Delta^*_{z\to w}\lp \eta_1(M_z, \mathbbm{1}_w)\otimes \eta_2(\mathbbm{1}_z,N_w)\rp\\ 
&(N\otimes M)_w\cong \Delta^*_{z\to w} \lp \eta_2(\mathbbm{1}_z, N_w)\otimes \eta_1(M_z,\mathbbm{1}_w)\rp
\end{aligned}
\ee
and moreover, using the compatibility of two multiplications, we have a canonical isomorphism:
\be
j^*_{z\ne w}\lp \eta_1(M_z, \mathbbm{1}_w)\otimes \eta_2(\mathbbm{1}_z,N_w)\rp\cong j^*_{z\ne w}\lp\eta_2(\mathbbm{1}_z, N_w)\otimes \eta_1(M_z,\mathbbm{1}_w) \rp.
\ee
Denote by $C_1(M,N):=\eta_1(M_z, \mathbbm{1}_w)\otimes \eta_2(\mathbbm{1}_z,N_w)$ and $C_2(M,N):=\eta_2(\mathbbm{1}_z, N_w)\otimes \eta_1(M_z,\mathbbm{1}_w)$. The two isomorphisms combine into the following diagram:
{\footnotesize\[
\btik
& j^*_{z\ne w}C_1(M,N)\cong j^*_{z\ne w}C_2(M,N)\arrow[dr]\arrow[dl] &\\
\left(j^*_{z\ne w}C_1(M,N)\right)/C_1(M,N)\dar{\cong} & & \left(j^*_{z\ne w}C_2(M,N)\right)/C_2(M,N)\dar{\cong}\\
(M\otimes N)_w[\pd^n \delta_{z-w}] & & (N\otimes M)_w[\pd^n \delta_{z-w}]
\etik.
\]}
Assuming the compactness of the objects involved, there exists $N$ such that 
\begin{equation}
    (z-w)^NC_1(M,N)\subseteq C_2(M,N),    
\end{equation}
and if we choose $N$ to be the smallest such $N$, then the above gives a morphism:
\be\label{eq:renormr}
\btik
((M\otimes N)_w+C_1(M,N))/C_1(M,N)\rar{(z-w)^N} & \left(\frac{1}{z-w}C_2(M,N)\right)/C_2(M,N)\rar{\cong}& (N\otimes M)_w
\etik.
\ee
This morphism is called the renormalized \(r\)-matrix in \cite{cautis2019cluster}. 

All these discussions work equally well with the formal completion $[G(\CO)\!\setminus\!\wh{\Gr}_G]$. At this point, it is tempting to compare the factorization structure as well as this renormalized \(r\)-matrix with those obtained from the algebraic considerations in the sections \ref{subsec:quantVHopf} and \ref{subsec:fullR}. However, two immediate problems rise. The first is that as far as we know, there is no notion of ``meromorphic product" of two objects in a factorization category. Indeed, in the definition of factorization categories, one is given a coherent product $\CC_z\boxtimes\CC_w\to \CC_{z,w}$, which is in most cases an equivalence. It is not clear to us how to obtain from this a functor $\CC_z\boxtimes \CC_w\to \CC_w\lpp (z-w)^{-1}\rpp$, which is what the meromorphic tensor product gives. Furthermore, the renormalized \(r\)-matrix discussed above only concerns the lowest loop degree part of the full \(r\)-matrix, and in principle one should expect a map from $\left(j^*_{z\ne w}C_1(M,N)\right)/C_1(M,N)$ to $\left(j^*_{z\ne w}C_2(M,N)\right)/C_2(M,N)$. We will not attempt to resolve these issues in this paper. However, we believe that the structures $Y_\hbar(\fd)$ possesses should have counterparts on $[\wh{G}(\CO)\!\setminus\!\wh{\Gr}_G]$. To give evidence of this, we prove the following Proposition \ref{Prop:dualVOA}. 

To formulate it, note that we can naturally identify $Y_1(\fd)$, or more specifically the subalgebra generated by $\fg[t]$ and $S(\fg^*(\CO))$, with the linear dual of $\C[\wh{G}(\CO)]\otimes V_0(\fg)$, which has the natural tensor product vertex algebra structure, and contains the vertex algebra $\CV$ from Remark \ref{Rem:dualVA}. Here, $\C[\wh{G}(\CO)]$ is the algebra of functions on the formal completion $\wh{G}(\CO)$ of $G(\CO)$ at the neutral element \(e\). We have seen from the aforementioned remark how this vertex algebra structure is what is used to define the meromorphic coproduct on $ Y^\circ_\hbar(\fd)$. 

\begin{Prop}\label{Prop:dualVOA}
    Let $\pi: \wh{\Gr}_G\to [\wh{G}(\CO)\!\setminus\!\wh{\Gr}_G]$ be the natural projection. Then one can identify
    \be
\Gamma(\wh{\Gr}_G, \pi^*\pi_*(\omega))\cong \C[\wh{G}(\CO)]\otimes V_0(\fg)
    \ee
    as vertex algebras. 
    
\end{Prop}

\begin{proof}
    Let $\wt{\pi}: \wh{G}(\CO)\times \wh{\Gr}_G\to \wh{\Gr}_G$ be the projection onto the second factor and 
    \be
        m\colon \wh{G}(\CO)\times \wh{\Gr}_G\to \wh{\Gr}_G
    \ee
    be action on the left. We claim that $\Gamma(\wh{\Gr}_G, \pi^*\pi_*(\omega))$ can be alternatively computed as $\Gamma(\wh{\Gr}_G, m_*(\CO_{\wh{G}(\CO)}\boxtimes \omega))$. This follows from the following base-change diagram:  
    \be
\btik
\wh{G}(\CO)\times \wh{\Gr}_G\rar{\wt{\pi}} \dar{m} & \wh{\Gr}_G\dar{\pi}\\
\wh{\Gr}_G\rar{\pi} & \wh{G}(\CO)\!\setminus\!\wh{\Gr}_G
\etik.
    \ee
This means that $\pi^*\pi_*(\omega)\cong m_*\wt{\pi}^*(\omega)=m_*(\CO_{\wh{G}(\CO)}\boxtimes \omega)$. Therefore we have an isomorphism of vertex algebras:
\be
\Gamma(\wh{\Gr}_G, \pi^*\pi_*(\omega))\cong \Gamma(\wh{G}(\CO)\times \wh{\Gr}_G, \CO_{\wh{G}(\CO)}\boxtimes \omega)=\C[\wh{G}(\CO)]\otimes V_0(\fg).
\ee
This completes the proof.

\end{proof}

This suggests that there should be a meromorphic tensor product on $\Coh([\wh{G}(\CO)\!\setminus\!\wh{\Gr}_G])$, which refines the factorization structure. Indeed, this category can be naturally identified with the category of comodules of $\C[\wh{G}(\CO)]\otimes V_0(\fg)$, and the vertex algebra structure induced from the factorization structure defines a functor between the categories of comodules. Again, we do not attempt to develop this here, but leave it as a conjecture. 

\begin{Conj}\label{Conj:merorfac}
    Under the equivalence of Proposition \ref{Prop:abequiv}, the meromorphic tensor product coming from $\Delta_z$ of \(Y_1(\fd)\) corresponds to a refinement of the factorization structure as above, admitting the limit $z\to w$. The renormalized \(r\)-matrix of equation \eqref{eq:renormr} corresponds the lowest $z^{-1}$ degree part of the full \(R\)-matrix $R(z)$.
\end{Conj}

\newpage

   \bibliographystyle{amsalpha}
   % Sets bio to AMS style format
   %\bibliographystyle{amsalpha-fi-arxlast}
   
   \bibliography{Yangian}

\end{document}